\documentclass[11pt,letterpaper, reqno]{amsart}

\renewcommand{\Im}{\operatorname{Im}}

\newcommand{\defeq}{\stackrel{\rm{def}}{=}}

\usepackage{amssymb}
\usepackage{amsthm}
\usepackage{amsxtra}
\usepackage{graphicx}
\usepackage{color}
\usepackage{inputenc}

\newcommand{\R}{\mathbb R}
\newcommand{\C}{\mathbb C}
\newcommand{\N}{\mathbb N}

\newcommand{\peq}{\hspace*{0.10in}}

\newtheorem{theorem}{Theorem}[section]
\newtheorem{definition}[theorem]{Definition}
\newtheorem{proposition}[theorem]{Proposition}
\newtheorem{lemma}[theorem]{Lemma}
\newtheorem{corollary}[theorem]{Corollary}

\newtheorem{claim}[theorem]{Claim}
\theoremstyle{remark}
\newtheorem{remark}[theorem]{Remark}

\numberwithin{equation}{section}

\newcommand{\ds}{\displaystyle}

\begin{document}

\title[Solutions to generalized Hartree equation]
{Global behavior of solutions to \\ 
the focusing generalized Hartree Equation}

\author[Anudeep K. Arora]{Anudeep Kumar Arora}
\address{Department of Mathematics  \& Statistics\\
Florida International University,  Miami, FL, USA}
\curraddr{}
\email{simplyandy7@gmail.com}
\thanks{} 

\author[Svetlana Roudenko]{Svetlana Roudenko}
\address{Department of Mathematics \& Statistics\\
Florida International University,  Miami, FL, USA}
\curraddr{}
\email{sroudenko@fiu.edu}
\thanks{}



\date{} 

\begin{abstract}
We study the global behavior of solutions to the nonlinear {\it generalized} Hartree equation, where the nonlinearity is of the non-local type and is expressed as a convolution, 
$$ 
\qquad \qquad iu_t + \Delta u  + (|x|^{-(N-\gamma)} \ast |u|^p)|u|^{p-2}u=0, \quad x \in \R^N, t\in \R.
$$ 
Our main goal is to understand behavior of $H^1$ (finite energy) solutions of this equation in various settings. In this work we make an initial attempt towards this goal. We first investigate the $H^1$ local wellposedness and small data theory. We then, in the intercritical regime ($0<s<1$), classify the behavior of $H^1$ solutions under the mass-energy assumption $\mathcal{ME}[u_0]<1$, identifying the sharp threshold for global versus finite time solutions via the sharp constant of the corresponding convolution type Gagliardo-Nirenberg interpolation inequality (note that the uniqueness of a ground state is not known in the general case). In particular, depending on the size of the initial mass and gradient, solutions will either exist for all time and  scatter in $H^1$, or blow up in finite time, or diverge along an infinity time sequence. To either obtain $H^1$ scattering or divergence to infinity, in this paper we employ the well-known concentration compactness and rigidity method of Kenig-Merle \cite{KM06} with the novelty of studying the nonlocal nonlinear potential given via convolution with negative powers of $|x|$ and different, including fractional, powers of nonlinearities.  
\end{abstract}

\maketitle

\tableofcontents

\section{Introduction}
Consider the focusing generalized Hartree, or Schr\"odinger - Hartree, equation of the form
\begin{equation}\label{gH}
i u_t + \Delta u  + \left( \frac1{|x|^{N-\gamma}} \ast |u|^p \right)|u|^{p-2} u=0, ~~ x \in \R^N, ~t\in \R, ~0<\gamma<N, ~p\geq 2. 
\end{equation}
Here, the function 
$u(x,t)$ is complex-valued and $\ast$ denotes the convolution operator in $\R^N$. 

The equation \eqref{gH} is a generalization of the standard Hartree equation with $p=2$, 
\begin{equation}\label{H}
iu_t+\Delta u + \left( \frac1{|x|^{N-\gamma}} \ast |u|^2 \right) u=0, ~~ x \in \R^N, 
\end{equation}
which can be considered as a classical limit of a field equation describing a
quantum mechanical non-relativistic many-boson system interacting through a
two body potential  $V(x) = \frac1{|x|^{N-\gamma}}$, see \cite{GV80}.
How it arises as an effective evolution equation in the mean-field limit of many-body quantum systems can be traced to work of Hepp \cite{H74}, see also \cite{GV80}, \cite{Sp80}, \cite{BGM2000}, \cite{BEGMY}, \cite{FGS07}. 
Lieb \& Yau \cite{LY87} mention it in a context of developing theory for stellar collapse, and in particular, in the boson particles setting. 
A special case of the convolution with $\frac1{|x|}$ in $\mathbb{R}^3$ is referred to as the Coulomb potential, which goes back to work of Lieb \cite{Lieb77} and has been intensively studied since then, see reviews \cite{FTY2002}, \cite{FTY2000}. With $\gamma =2$ and $N=3$, a pseudo-relativistic version of this equation arises in the mean field limit of weakly interacting molecules and bosonic atoms (for example, see \cite{FJL07}, \cite{FL04}), taking the form
\begin{equation}\label{boson}
iu_t -\sqrt{-\Delta+m^2\,} \,u  + (|x|^{-1}\ast |u|^2)u=0,\quad x\in\R^3, 
\end{equation}
which has recently generated many interesting questions about the dynamics of its solutions.

Unlike the standard nonlinear Schr\"odinger equation with pure nonlinearity $|u|^{p-1} u$, the distinct feature of the Hartree equation \eqref{H} is that it models systems with long-range interactions. Possible experimental realizations of such interaction, where the power in the convolution changes, include the interaction of ultracold Rydberg atoms that have large principal quantum numbers \cite{L-10}. These interactions between atoms in highly excited Rydberg levels are long range and dominated by dipole-dipole-type forces (the strength of the interaction between Rb atoms is about $10^{12}$ times stronger than that between Rb atoms in the ground state \cite{SWM10}). The spatial dependence of interactions may be $1/|x|^3$ for small $|x|$ and $1/|x|^6$ for  larger $|x|$. Other powers such as $1/|x|^2$ are also possible, see \cite{OGKA2000}. Even more general, the potential can incorporate not only radial dependence, but also angular dependence $\frac{\theta(x)}{|x|^{N-\gamma}}$ \cite{L-10}, though, in this work we will not consider this case.  

The equation \eqref{gH} can be written (in terms of the wave function $u$ and the potential $V$) as the Schr\"odinger - Poisson system  of the form
	\begin{align}\label{SP}
	\begin{cases}
	iu_t +\Delta u + V|u|^{p-2}u=0\\
	-\Delta V = (N-2)|\mathbb{S}^{N-1}| \, |u|^p.
	\end{cases}
	\end{align}
This can be thought of as an electrostatic version of the Maxwell-Schr\"odinger system, which describes the interaction between the electromagnetic field and the wave function related to a quantum non-relativistic charged particle (see, for example, \cite{CG04} and \cite{Lieb03}). 

With numerous applications, it makes sense to develop a unified mathematical theory of solutions behavior for the general equation \eqref{gH}. 
For that purpose we consider initial data in the $H^1$ space, $u_0(x)\in H^1(\R^N)$, so that we can study finite Hamiltonian or finite energy solutions (definitions below).  
The local existence of $H^1$ solutions is available in the standard Hartree equation \eqref{H} from the work of Ginibre \& Velo \cite{GV80}, see also Cazenave \cite{C03}. 
We prove the local well-posedness in $H^1$ for the generalized Hartree \eqref{gH} with $p\geq 2$ in Section \ref{pre}. 
	
Denote the maximal time existence interval of solutions to \eqref{gH} by $(T_*, T^*)$. We say a solution is global in forward time if $T^* = +\infty$ (and similarly for the backward time). If $(T_*, T^*)= \R$, the solution is said to be global. The global existence for \eqref{gH} is delicate due to the focusing nature of the nonlinearity, and is investigated in this paper. During their lifespan, solutions to \eqref{gH} satisfy mass, energy (Hamiltonian) and momentum conservations: 
\begin{align*}
\quad M[u(t)] &\defeq \int_{\R^N}^{}|u(x,t)|^2\,dx =M[u_0],\\
\quad E[u(t)] &\defeq \frac{1}{2}\int_{\R^N}^{}|\nabla u(x,t)|^2 \, dx- \frac{1}{2p} \int_{\R^N}\left(\frac1{|x|^{N-\gamma}}\ast|u(\,\cdot\,,t)|^p \right)|u(x,t)|^{p}\,dx =E[u_0],\\
\quad P[u(t)] &\defeq \Im \int_{\R^N}\bar{u}(x,t)\nabla u(x,t) \,dx =P[u_0].
\end{align*}

The equation \eqref{gH} has several invariances: if $u(x,t)$ is a solution to \eqref{gH}, so is $\tilde{u}(x,t)$: 
\begin{itemize}
\item \textit{Spatial translation:} for a fixed $x_0\in \R^N$, $\ds \tilde{u}(x,t)=u(x-x_0,t)$.

\item \textit{Time translation:} for a fixed $\tau\in\R$, $\ds \tilde{u}(x,t)=u(x,t+\tau)$.

\item \textit{Time reversal:} $\ds \tilde{u}(x,t)= \overline{u}(x, -t)$.

\item \textit{Phase rotation:} for a fixed $\theta \in [0,\pi)$, $\ds \tilde{u}(x,t)= e^{i\theta} u(x,t)$.

\item \textit{Spatial rotation:} for a fixed $R \in SO(N)$, $\ds \tilde{u}(x,t)=u(R^{-1} x,t)$.

\item \textit{Galilean transformation:} for a fixed $\xi_0\in\R^N$, $\ds \tilde{u}(x,t)=e^{i(x\cdot\xi_0-t|\xi_0|^2)} \, u(x-\xi_0t,t)$.

\item \textit{Scaling:} for a fixed $\lambda\in(0,\infty)$, $\ds \tilde{u}(x,t)=\lambda^{\frac{\gamma+2}{2(p-1)}}\, u(\lambda x,\lambda^2t)$.

\item \textit{Pseudo-conformal transformation:} If $p=1+\frac{\gamma+2}{N}$, then $\ds \tilde{u}(x,t)=\frac1{|t|^{N/2}} u\left(\frac{x}{t},-\frac1{t}\right) \,  e^{\frac{i|x|^2}{4t}}$.
\end{itemize}

The equation \eqref{gH} is referred to as the $\dot{H}^{s_c}$ - critical, if the $\dot{H}^{s_c}$ norm of the solution is invariant under the scaling. The critical scaling index $s_c$ coming from the scaling invariance is defined as 
\begin{equation}
\label{E:s}
s_c=\frac{N}{2}-\frac{\gamma+2}{2(p-1)}.
\end{equation}
If $s_c=0$, or $\displaystyle{p=1+\frac{\gamma+2}{N}}$, the equation \eqref{gH} is referred to as the mass-critical (or $L^2$-critical). For the standard Hartree nonlinearity ($p=2$), the mass-critical case corresponds to $N-\gamma =2$, and thus, occurs only in dimensions $N>2$ with the nonlinearity $\left(\frac1{|x|^{2}} \ast|u|^2\right)u$ regardless of the dimension. If $s_c=1$, or $\displaystyle{p=1+\frac{\gamma+2}{N-2}}$, the problem is called the energy-critical (or $\dot{H}^1$ - critical). For the standard Hartree nonlinearity ($p=2$), the energy-critical case corresponds to $N-\gamma =4$, which implies that it occurs only in dimensions $N$ greater than 4 with the nonlinearity $\left(\frac1{|x|^{4}}\ast|u|^2\right)u$, also regardless of the dimension. Note that the generalized Hartree equation \eqref{gH}, being flexible in power $p$, can be, say, energy-critical in dimensions less than 4, e.g., in 3d $\left(\frac1{|x|}\ast|u|^5\right)|u|^3 u$ or $\left(\frac1{|x|^2}\ast|u|^4 \right)|u|^2 u$, which can make analysis and methods more accessible. 

A global solution $u(t)$ to \eqref{gH} is said to scatter in $H^s(\R^N)$ 
as $t\rightarrow +\infty$, if there exists  $u^+ \in H^s(\R^N)$ such that
$$
\lim\limits_{t\rightarrow +\infty}\|u(t)-e^{it\Delta}u^+\|_{H^s(\R^N)}=0.
$$
There is a number of early works on global existence, asymptotic behavior of solutions and scattering theory for the standard Hartree equation \eqref{H}.
Studies trace back to Ginibre \& Velo \cite{GV80}, where the local wellposedness is established and the authors also prove asymptotic completeness for a repulsive potential. 
Hayashi \& Tsutsumi \cite{HT87} 
continue developing the scattering theory and obtain the asymptotic completeness of wave operators in $H^m \cap L^p(|x|^{\beta}dx)$.  
We refer the reader to Ginibre \& Ozawa \cite{GO93} for results in the case of the convolution with $|x|^{-1}$, or $N-\gamma=1$, for $N\geq 2$; to Ginibre \& Velo \cite{GV2000} for $2<N-\gamma<\min(4,N)$ when $N\geq 3$. In a sequence of papers \cite{GVI00}-\cite{GV01} Ginibre \& Velo considered the time-dependent potential $\pm t^{\mu-\gamma}|x|^{-\mu}$ and studied the asymptotic dynamics and scattering (for any data in the repulsive case or small data otherwise)  first when the convolution power is $\frac{1}{2}<N-\gamma<1$ in \cite{GVI00}, and then in the whole range $0<N-\gamma\leq 1$ in \cite{GVII00}. These two papers are written in the framework of Sobolev spaces with the assumption $\mu\leq N-2$ ($N\geq 3$). In \cite{GV01} the Hartree was treated in Gevrey spaces, which made it possible to cover the whole range $0<\mu\leq N$ with an arbitrary space dimensions $N\geq 1$. In \cite{HNO98} Hayashi, Naumkin \& Ozawa studied the Hartree equation with $N-\gamma=1$ ($N\geq 2$) and initial data in a weighted Sobolev space $H^{0,\alpha}\cap H^{\alpha,0}$ with $\frac{1}{2}<\alpha<\frac{N}{2}$.
	
Our aim is to understand global behavior and dynamics of solutions to the generalized Hartree \eqref{gH}, in particular, how the nonlocal potential with the flexibility of different powers in nonlinearity may influence the global behavior and dynamics of solutions either with infinite or finite time of existence. We are also curious whether solutions behave in a manner similar to local potentials as, for example, in the standard semilinear Schr\"odinger equation with $|u|^{p-1}\,u$ nonlinearity, or if nonlocality creates significant differences in solutions behavior. In addition, we want to develop methods needed to study such solutions. 

In this work we describe the global behavior of solutions to \eqref{gH} with $H^1$ initial data in the inter-critical regime ($0<s_c<1$), provided that $p \geq 2$, that is,
\begin{equation}\label{E:range1}
1+\frac{\gamma+2}{N} < p < 1+\frac{\gamma+2}{N-2}, ~0<\gamma<N ~~\mbox{and}~~~p \geq 2,
\end{equation}
with the appropriate modification of the right-hand side for $N=1, 2$ ($ p <\infty)$. (As a byproduct, we also obtain local wellposedness for any energy-critical and subcritical cases, $s \leq 1$, and small data theory in the energy-subcritical setting, $s<1$.) 
We establish a dichotomy for global vs. finite time solutions under the  mass-energy threshold and show $H^1$ scattering for the global solutions, following the concentration-compactness approach of Kenig \& Merle \cite{KM06}, and divergence along a time sequence for nonradial infinite variance data (also via concentration-compactness method). This is in the spirit of \cite{HR08}, \cite{DHR08}, \cite{G}, \cite{HR10} for the focusing NLS. We emphasize that while the concentration-compactness approach is standard in the field by now, it is important first, to understand the behavior of solutions and describe their asymptotic dynamics and thresholds if possible; secondly, to demonstrate that this method works in the general nonlocal setting while showing modifications needed to handle a general convolution term, and finally, simply to make well-posedness available in a complete general nonlocal setting of the intercritical range, which is needed for future investigations. Some of the immediate questions we investigate in subsequent papers, for example, in \cite{AKA1} we explore the approach of Dodson \& Murphy \cite{DM17} and with their method prove scattering for globally existing in time solutions in this inter-critical regime; in \cite{AKAR2} we investigate the local well-posedness at the (non-conserved) critical regularity $\dot{H}^{s_c}$ for $s_c \geq 0$ (including energy-supercritical regime) and extend the local existence to global for small $\dot{H}^{s_c}$ data while for certain large data with positive energy we show blow-up in finite time. One of the interesting questions for this equation compared to the standard NLS is to investigate the blow-up dynamics and how it is influenced by the nonlocal potential. We will address this in \cite{AKAR3} and \cite{AKARY}.  

In order to characterize the sharp threshold for the dichotomy, one needs a notion of a ground state. The equation \eqref{gH} admits solitary waves solutions of the form $u(x, t) = e^{it}Q(x)$, where $Q$ solves the nonlinear nonlocal elliptic equation
\begin{equation}\label{nonlinell}
-Q + \Delta Q + \left(|x|^{-(N-\gamma)}\ast|Q|^p\right)|Q|^{p-2}Q=0.
\end{equation}
The equation \eqref{nonlinell} is known as the nonlinear Choquard or Choquard-Pekar equation. A special case of \eqref{nonlinell} when $N=3$, $p=2$, and $\gamma=2$,
\begin{equation}\label{original}
\Delta Q-Q+\left(|x|^{-1}\ast|Q|^2\right)Q=0
\end{equation} 
appeared back in 1954 in the work of S. I. Pekar \cite{Pekar54} describing the quantum mechanics of a polaron at rest. Lieb in \cite{Lieb77} mentions it in the context of the Hartree-Fock theory of plasma, pointing out that  
P. Choquard proposed investigating minimization of the corresponding  functional in 1976. In 1996 R. Penrose proposed equation \eqref{original} as a model of self-gravitating matter, in which quantum state reduction is understood as a gravitational phenomenon, see \cite{Penrose98}.

The existence of positive solutions to \eqref{original} was first proved by Lieb \cite{Lieb77}, see also Lions \cite{Lions80}, \cite{Lions84I}. The general existence result of positive solutions along with the regularity and radial symmetry of solutions to \eqref{nonlinell} for $\frac{N+\gamma}{N}<p<\frac{N+\gamma}{N-2}$ with $0<\gamma<N$ was shown by Moroz \& Schaftingen \cite{MS13} (see also a review by Moroz \& Schaftingen \cite{MS17} and references therein). 
	
The uniqueness proof\footnote{In certain existing literature there seem to be a misconception about the uniqueness of the ground state even in the standard ($p=2$) Hartree equation: statements such as ``take the positive unique ground state solution $Q$ of the equation $\Delta Q-Q+\left(|x|^{-b}\ast|Q|^2\right)Q=0$" are not justified for any $0 < b < N$ as the uniqueness of the ground state is only proved when $b=N-2$, $2<N<6$, see Appendix.}
for $p=2$ with $\gamma=2$ in dimension $N=3$ dates back to 1976-77 work of 
Lieb \cite{Lieb77} and later in 2009 was extended to the dimension $N=4$ by Krieger, Lenzmann \& Rapha\"el in \cite{KLR09}; the uniqueness in the pseudo-relativistic 3d version of \eqref{original} was established by Lenzmann \cite{L09}. We review the proof of uniqueness for any (reasonable) $N$ (and $p=2, \gamma=2$) in the Appendix.
For other cases of $\gamma$ and $p$, it is an intricate issue, and while several authors made attempts to obtain uniqueness, it is still an open question. A recent work \cite{X16} shows uniqueness and nondegeneracy of the ground state for $p = 2 + \epsilon$, i.e., when $p$ is sufficiently close to $2$ in dimension $N=3$ and $\gamma = 2$ via perturbation methods. We note that the proof of uniqueness for the nonlinear elliptic equation with convolution \eqref{nonlinell} differs from the corresponding results for the NLS-type equations (e.g., with $|u|^{p-1}u$ type nonlinearity), for which it is given, for example, by Kwong \cite{K89} and Berestycki \& Lions \cite{BL83I}-\cite{BL83II}. The proof in the Hartree case uses Newton's theorem for the convolution in \eqref{original} and linearity in $Q$ outside of the convolution ($p=2$), see more on this in Section \ref{propQ} and Appendix. In this work, we do not need the uniqueness, it suffices to use minimizing properties of the Weinstein-type functional and the value of the sharp constant in the Gagliardo-Nirenberg convolution type inequality via ground state solutions as that value will be unique. Thus, we denote by $Q$ any ground state solution of \eqref{nonlinell} and use such quantities as $M[Q]$, $\|\nabla Q\|_{L^2}$ and $E[Q]$, which are obtained from the sharp constant.

As in \cite{HR07} and \cite{HR08} for the NLS equation, we observe that the quantities $\|u_0\|_{L^2(\R^N)}^{1-s_c} \,  \|\nabla u_0 \|_{L^2(\R^N)}^{s_c}$ and $ M[u_0]^{1-s_c} \, E[u_0]^{s_c}$ are also scale-invariant in the generalized Hartree equation, and for $s_c>0$ with $\theta=\frac{1-s_c}{s_c}$ we define
\begin{itemize}
\item 
renormalized mass-energy: 
$\ds {\mathcal{ME}[u]=\frac{M[u]^{\theta}E[u]}{M[Q]^{\theta}E[Q]}}$,

\item 
renormalized gradient (dependent on $t$): 
$\ds {\mathcal{G}[u(t)]=\frac{\|u\|^{\theta}_{L^2(\R^N)}\|\nabla u(t)\|_{L^2(\R^N)}}{\|Q\|^{\theta}_{L^2(\R^N)}\|\nabla Q\|_{L^2(\R^N)}}}$, \quad \mbox{and}
		
\item 
renormalized momentum: 
$\ds {\mathcal{P}[u]=\frac{\|u\|^{\theta-1}_{L^2(\R^N)}P[u]}{\|Q\|^{\theta}_{L^2(\R^N)}\|\nabla Q\|_{L^2(\R^N)}}}$.
\end{itemize}

We now state the main result of this paper about solutions behavior under the mass-energy threshold. We consider \eqref{gH} with given $N, \gamma$, and $p \geq 2$ so that  $s_c$ defined by \eqref{E:s} is $0<s_c<1$. We first consider solutions with zero momentum.

\begin{theorem}[Zero momentum]\label{mainP0} 
Let $u_0\in H^1(\R^N)$ with $P[u_0]=0$ and let $u(t)$ be the corresponding solution to \eqref{gH} with the maximal time interval of existence $(T_*, T^*)$. 
Suppose that $\mathcal{ME}[u_0]  < 1$.
\begin{enumerate}
\item 
If $\mathcal{G}[u_0] < 1$, then
\begin{enumerate}
\item 
the solution exists globally in time with $\mathcal{G}[u(t)] < 1$ for all $t\in \R$, and

\item 
$u(t)$ scatters in $H^1$, in other words, 
there exists $u_{\pm}\in H^1$ such that 
$$
\lim\limits_{t\rightarrow\pm\infty}\|u(t)-e^{it\Delta}u_{\pm}\|_{H^1(\R^N)}=0.
$$
\end{enumerate}

\item 
If $\mathcal{G}[u_0] > 1,$ then $\mathcal{G}[u(t)] > 1$ for all $t\in(T_*, T^*)$. Moreover, if
\begin{enumerate}
\item 
$|x|u_0\in L^2(\R^N)$ (finite variance) or $u_0$ is radial, then the solution blows up in finite time,

\item 
$u_0$ is of infinite variance and nonradial, then either the solution blows up in finite time or there exits a sequence of times $t_n\rightarrow +\infty$ (or $t_n\rightarrow -\infty$) such that $\|\nabla u(t_n)\|_{L^2(\R^N)}\rightarrow \infty$.
\end{enumerate}
\end{enumerate}
\end{theorem}

The general case when $P[u_0]\neq 0$ is given by the following	
\begin{theorem}\label{main}
Let $u_0\in H^1(\R^N)$ and $u(t)$ be the corresponding solution to \eqref{gH} with the maximal time interval of existence $(T_*,T^*)$. Assume that
\begin{equation}\label{main1}
\mathcal{ME}[u_0] - \frac{N(p-1)-\gamma}{N(p-1)-\gamma-2}\mathcal{P}[u_0]^2 < 1.
\end{equation}

\begin{enumerate}
\item 
If
\begin{equation}\label{main2}
\mathcal{G}[u_0]^2-\mathcal{P}[u_0]^2 < 1,
\end{equation}
then
\begin{enumerate}
\item 
the solution exists globally in time with $\mathcal{G}[u(t)]^2-\mathcal{P}[u_0]^2 < 1$ for all $t\in \R$, and

\item 
$u(t)$ scatters in $H^1$, i.e., 
there exists $u_{\pm}\in H^1$ such that 
$$
\lim\limits_{t\rightarrow\pm\infty}\|u(t)-e^{it\Delta}u_{\pm}\|_{H^1(\R^N)}=0.
$$
\end{enumerate}

\item 
If
\begin{align}\label{main3}
\mathcal{G}[u_0]^2-\mathcal{P}[u_0]^2 > 1,
\end{align}
then $\mathcal{G}[u(t)]^2-\mathcal{P}[u_0]^2 > 1$ for all $t\in(T_*,T^*)$. 
Moreover, if
\begin{enumerate}
\item 
$|x|u_0\in L^2(\R^N)$ or $u_0$ is radial, then the solution blows up in finite time,

\item 
$u_0$ is of infinite variance and nonradial, then either the solution blows up in finite time or there exits a sequence of times $t_n\rightarrow +\infty$ (or $t_n\rightarrow -\infty$)  such that $\|\nabla u(t_n)\|_{L^2(\R^N)}\rightarrow \infty$.
\end{enumerate}
\end{enumerate}
\end{theorem}
While we follow the strategy of \cite{HR08}, \cite{G}, \cite{DHR08} and \cite{HR10}, the fundamental difference is in the nonlocal potential, and control of convolution terms arising in various steps of this work. For example, to obtain local well-posedness and small data theory in $H^1$ we do not get the contraction automatically as the difference produces extra terms due to convolution. We use Lemma \ref{HLS} to estimate the inhomogeneous term in Duhamel's  formula via Strichartz estimate in Proposition \ref{lwp}, Theorem \ref{smalldata}, Theorem \ref{H1scatter}, Theorem \ref{LTP} and in Theorem \ref{crit.elem} (Claim \ref{LTPclaim2}). Also note that to control the potential energy 
in Proposition \ref{energypythadecomp} and in Lemma \ref{GNineq}, we rely on $L_x^{\frac{2Np}{N+\gamma}}$ norm (using the assumption that $s_c<1$) along with the Lemma \ref{HLS}. Moreover, the local virial identity \eqref{dichotomy9}, \eqref{dichotomy10}, \eqref{dichotomy11} in Theorem \ref{Dichotomy} and Theorem \ref{rigidity} has some extra terms involving convolution  which demands a careful study and application of convolution properties, Lemma \ref{HLS} and Lemma \ref{rsob}. We also have to review the sharp constant coming from the convolution-type Gagliardo-Nirenberg inequality and discuss the values coming from the minimization process as there is no uniqueness.

The paper is organized as follows: in Section \ref{pre}, we review the necessary preliminaries such as Strichartz estimates, embeddings and other useful inequalities. There, we also discuss the local well-posedness in the energy-critical and subcritical cases and $p\geq 2$. It would be interesting to investigate well-posedness for $p<2$. 
In Section \ref{small}, we prove the small data theory in the energy-subcritical setting as well as the $H^1$ scattering along with the long-time perturbation lemma. In Section \ref{propQ}, we introduce a generalized convolution type Gagliardo - Nirenberg inequality and show that the minimizer is given by a positive minimizer (a ground state) $Q$ and identify the sharp constant.  In Section \ref{dich}, we prove Theorem \ref{mainP0}, the dichotomy result: global existence vs. blow-up; we also include several Lemmas needed to prove scattering later. In section \ref{comp}, we prove Theorem \ref{mainP0}(1), part (b), the scattering, using the concentration-compactness and rigidity approach of Kenig \& Merle \cite{KM06}, and the adaptation of Holmer \& Roudenko \cite{HR08}.
In Section \ref{rigid} we exclude the existence of the critical element 
using the rigidity argument applied to the corresponding localized virial identity. In the last Section \ref{S-last} we consider the case of nonradial solutions with infinite variance and larger than 1 renormalized gradient (part 2(b) of both theorems), and discuss either the divergence to infinity along a time sequence or finite time existence of solution in a spirit of \cite{HR10}. In appendix we review the uniqueness argument.
	
\textbf{Acknowledgments.} A.K.A. would like to thank MSRI, where this work was initiated, for providing excellent working conditions during the Fall 2015 semester program ``New Challenges in PDE : Deterministic Dynamics and Randomness in High and Infinite Dimensional Systems". A.K.A. would also like to thank Luiz G. Farah, Cristi D. Guevara and Justin Holmer for fruitful and productive discussions on the subject.
S.R. was partially supported by the NSF CAREER grant DMS-1151618 and NSF grant DMS-1815873.  A.K.A.'s graduate research support was in part funded by the grant DMS-1151618 (PI: Roudenko).

\section{Preliminaries}\label{pre}
	
\subsection{Strichartz estimates and admissible pairs}

For $s>0$, the pair $(q,r)$ is referred to as an $\dot{H}^s$-admissible (for $s=0$, it is called an $L^2$-admissible), if
\begin{equation}\label{admpair}
\frac{2}{q} + \frac{N}{r}=\frac{N}{2}-s\quad \text{with}\quad 2\leq q,r\leq \infty\quad \text{and}\quad (q,r,N)\neq(2,\infty,2).	
\end{equation}

To control the constants uniformly in Strichartz estimates below, we restrict the range for the pair $(q, r)$, defined in \eqref{admpair}, depending on the dimension $N$ (as in \cite{G}): 
\begin{align}\label{range}
\begin{cases}
	\left(\frac{2}{1-s}\right)^+\leq q\leq \infty,\quad \frac{2N}{N-2s}\leq r \leq \left(\frac{2N}{N-2}\right)^-,\,\,\text{if}\,\, N\geq 3\\
	\left(\frac{2}{1-s}\right)^+\leq q\leq \infty,\quad \frac{2}{1-s}\leq r\leq \left(\left(\frac{2}{1-s}\right)^+\right)',\,\,\text{if}\,\, N=2\\
	\frac{4}{1-2s}\leq q \leq \infty,\quad \frac{2}{1-2s}\leq r\leq \infty,\,\,\text{if}\,\, N=1.
\end{cases}
\end{align}
Here, $n^+$ is a fixed number (slightly) greater than $n$ such that $\frac{1}{n}=\frac{1}{n^+}+\frac{1}{(n^+)'}$. Respectively, $n^-$ is a fixed number (slightly) less than $n$. 

Following \cite{HR08}, we introduce the $S(\dot{H}^s)$ notation: 
\begin{equation}\label{SH}
\|u\|_{S(\dot{H}^s)}=\sup\{ \|u\|_{L_t^q \, L_x^r}: (q,r) ~\mbox{as ~ in ~} \eqref{admpair} ~\mbox{and}~ \eqref{range} \}.
\end{equation}

Similarly, in order to define the dual Strichartz norm, we set the following restrictions:  
\begin{align}\label{dualrange}
\begin{cases}
	\left(\frac{2}{1+s}\right)^+\leq q\leq \left(\frac{1}{s}\right)^-,\quad \left(\frac{2N}{N-2s}\right)^+\leq r \leq \left(\frac{2N}{N-2}\right)^-,\,\,\text{if}\,\, N\geq 3\\
	\left(\frac{2}{1+s}\right)^+\leq q\leq \left(\frac{1}{s}\right)^-,\quad \left(\frac{2}{1-s}\right)^+\leq r\leq \left(\left(\frac{2}{1+s}\right)^+\right)',\,\,\text{if}\,\, N=2\\
	\frac{2}{1+2s}\leq q \leq \left(\frac{1}{s}\right)^-,\quad \left(\frac{2}{1-s}\right)^+\leq r\leq \infty,\,\,\text{if}\,\, N=1,
\end{cases}
\end{align}
and define the dual Strichartz norm as
\begin{equation}\label{S-H}
\|u\|_{S'(\dot{H}^{-s})}=\inf\{ \|u\|_{L_t^{q'}L_x^{r'}}: \tfrac1{q'}+\tfrac1{q}=1, \tfrac1{r'}+\tfrac1{r}=1 ~~\mbox{with}~~(q,r) ~\mbox{as~in~}\eqref{admpair} ~\mbox{and}~ \eqref{dualrange} \}.
\end{equation}

In the sequel, for given $N$, $p$, $\gamma$, and hence, a fixed $0<s_c<1$, we use the following $L^2$-admissible pairs : 
\begin{equation}\label{L2adm1}
(q_1,r_1)=\left(\frac{2p}{1+s_c(p-1)},\frac{2Np}{N+\gamma}\right)
\end{equation}
and
\begin{equation}\label{L2adm2}
(q_2,r_2)=\left(\frac{2p}{1-s_c},\frac{2Np}{N+\gamma+2s_cp}\right).
\end{equation}
Observe that $s_c<1$ implies $\frac{2p}{1+s_c(p-1) }>2$. As an $L^2$-dual admissible pair we take
\begin{equation}\label{L2dadm}
(q_1^{\prime},r_1^{\prime})=\left(\frac{2p}{2p-1-s_c(p-1)},\frac{2Np}{2Np-N-\gamma}\right).
\end{equation}
The specific $\dot{H}^{s_c}$-admissible pair we use is	
\begin{align}\label{Hscadm}
(q_2,r_1)=\left(\frac{2p}{1-s_c},\frac{2Np}{N+\gamma}\right),
\end{align}
and the $\dot{H}^{-s_c}$ dual admissible pair is given by 
\begin{align}\label{Hscdadm}
(q_3^{\prime},r_1^{\prime})=\left(\frac{2p}{(2p-1)(1-s_c)},\frac{2Np}{2Np-N-\gamma}\right).	
\end{align}
Note that $\left(\frac{2p}{1+s_c(2p-1)},\frac{2Np}{N+\gamma}\right)$ is also an $\dot{H}^{-s_c}$ admissible pair. 
Observe that $s_c<1$ imply that both $\frac{2p}{1+s_c(2p-1) }>\frac{2}{1+s_c}$ and $\frac{2p}{1+s_c(2p-1) }<\frac{1}{s_c}$, thus, confirming to be in the range of \eqref{dualrange}. 

Using Duhamel's formula, the equation \eqref{gH} is equivalent to the integral equation
\begin{equation}\label{duhamel}
u(x,t) = e^{it\Delta}u_0 + i \int_0^t e^{i(t-t')\Delta} (|x|^{-(N-\gamma)}\ast |u|^p)|u|^{p-2}\, u(t')\,dt'.
\end{equation}
We recall the following well-known Strichartz estimates (see Cazenave \cite{C03}, Foschi \cite{F05}, and Keel-Tao \cite{KT98}). 
\begin{lemma}\label{stzest}
For the range of $p$ and $q$ as in \eqref{SH}, we have
\begin{align}\label{stri1}
\|e^{it\Delta}\phi\|_{S(L^2)} 
&\leq c\,\|\phi\|_{L^2},\\
\label{stri2}
\left\| \int_{0}^{t}e^{i(t-t')\Delta}f(\cdot\,,\,t')\,dt'\right\|_{S(L^2)}
&\leq c\, \|f\|_{S'(L^2)}.
\end{align}
\end{lemma}
Using the Sobolev embedding (since $e^{it\Delta}$ commutes with derivatives), we obtain
\begin{corollary}
For the range of $p$ and $q$ as in \eqref{SH}, we have
\begin{align}\label{sobstri1}
\qquad \|e^{it\Delta}\phi \|_{S(\dot{H}^s)} & \leq c\, \|\phi\|_{\dot{H}^s}\\
\label{sobstri2}
\|\int_{0}^{t}e^{i(t-t')\Delta}f(\cdot\,,\,t')\,dt'\|_{S(\dot{H}^s)} &\leq c\,\|D^sf\|_{S'(L^2)}.
\end{align}
\end{corollary}
We also recall a more refined than \eqref{sobstri2} Strichartz estimate, which includes a larger set of admissible indexes than \eqref{sobstri2}.
\begin{lemma}[Kato-Strichartz estimate, \cite{Kato94}]\label{Kato}
If $F\in S'(\dot{H}^{-s})$, then
\begin{align}\label{Katostri}
\|\int_{0}^{t}e^{i(t-t')\Delta}F(t')dt'\|_{S(\dot{H}^s)}\lesssim \|F\|_{S'(\dot{H}^{-s})}.
\end{align}		
\end{lemma}
Note that we can use the dual of $\dot{H}^{s}$ pair on the right side of above inequality (for example, from \eqref{S-H}), which would not follow from \eqref{sobstri2}. 
	
\subsection{Embeddings} 
In this section we state embeddings and inequalities used later.

\begin{lemma}[Hardy-Littlewood-Sobolev inequality, \cite{Lieb83}]\label{HLS}
For $0<\gamma <N$ and $1<p,q<\infty$, there exists a sharp constant $c_{N,p,\gamma}>0$ such that
$$
\left \|\int_{\R^N} \frac{u(y)}{|x-y|^{N-\gamma}} \, dy \right\|_{L^q(\R^N)}\leq c_{N,p,\gamma}\|u\|_{L^p(\R^N)},
$$
where $\frac{1}{q}=\frac{1}{p}-\frac{\gamma}{N}$ and $p<\frac{N}{\gamma}$.
\end{lemma}

\begin{remark}\label{HLSrem}
Observe that
$$
\nabla\left(|x|^{-(N-\gamma)}\right)=C_{N,\gamma}|x|^{-(N-(\gamma-1))}.
$$
\end{remark}

\begin{lemma}[Radial Sobolev inequality, \cite{S77}]\label{rsob}
Let $u\in H^1(\R^N)$ be radially symmetric. Then
$$
\||x|^{\frac{N-1}{2}}u\|_{L^{\infty}(\R^N)}\leq C\|u\|^{1/2}_{L^2(\R^N)}\|\nabla u\|^{1/2}_{L^2(\R^N)}.
$$
\end{lemma}

\subsection{Local well-posedness in $H^1$}
We end this section with the local existence result in $H^1$ for the equation \eqref{gH}. We consider the integral equation \eqref{duhamel} with $u_0\in H^1(\R^N)$ and $0<\gamma<N$ with  
\begin{align}\label{prangelwp}
	\begin{cases}
	2 \leq p \leq 1+\frac{\gamma+2}{N-2}, & \text{if} \,\,N\geq 3	\\
	2 \leq  p<\infty,&\text{if} \,\,N=1,2.	
	\end{cases}
\end{align}

\begin{remark}\label{sdcontraction}
Let $f(z)=|z|^{p-2}z$. The complex derivative of $f$ is given by
$$
f_z(z)=\frac{p}{2}|z|^{p-2}\quad \text{and}\quad f_{\bar{z}}(z)=\frac{p-2}{2}|z|^{p-4}z^2.
$$
For $z_1$, $z_2\in\C$ we get
$$
f(z_1)-f(z_2)=\int_{0}^{1}\Big[f_{z_1}(z_2+\theta(z_1-z_2))(z_1-z_2)+f_{\overline{z_1}}(z_2+\theta(z_1-z_2))(\overline{z_1-z_2})\Big]\,d\theta.
$$
Hence,
\begin{align}\label{sdc1}
|f(z_1)-f(z_2)|\lesssim 
	 (|z_1|^{p-2}+|z_2|^{p-2})|z_1-z_2|.
\end{align}
Also, observe that for $p\geq 1$ (e.g., see \cite{CFH11})
\begin{align}\label{sdc5}
| |z_1|^p-|z_2|^p|\lesssim (|z_1|^{p-1}+|z_2|^{p-1})|z_1-z_2|.
\end{align}
\end{remark}
	
\begin{proposition}\label{lwp}
If $p$ satisfies \eqref{prangelwp}, then for $u_0\in H^1(\R^N)$ there exists $T>0$ and a unique solution $u(x,t)$ of the integral equation \eqref{duhamel} in the time interval $[0,T]$ with 
		\begin{align}\label{lwpspace}
		u\in C([0,T];H^1(\R^N))\cap L^{q_1}([0,T];W^{1,r_1}(\R^N)),	
		\end{align}
where $(q_1,r_1)$ is given by \eqref{L2adm1}. In the energy-critical case $p=1+\frac{\gamma+2}{N-2}$ (or $s_c=1$) we require an additional assumption of smallness of $\|u_0\|_{H^1_x}$. In any energy-subcritical case $p <1+\frac{\gamma+2}{N-2}$ the time $T=T(\|u_0\|_{H^1},N,p,\gamma)>0$. 
\end{proposition}
	
\begin{proof} For $T>0$, specified later, define $\ds \nu(u)=\max\big\{\sup\limits_{t\in[0,T]}\|u\|_{H^1_x},\,\,\|u\|_{L_t^{q_1}W_x^{1,r_1}}\big\}$ and for an appropriately defined constant $M>0$, also specified later, let
		\begin{align}\label{contractspace}
		\mathcal{S}=\{u\in C([0,T]);H^1_x(\R^N)\cap L_t^{q_1}([0,T]);W_x^{1,r_1}(\R^N)\,:\nu(u)\leq M \}.
		\end{align}
We prove that the following operator 
		\begin{align}\label{Dope}
			\Phi(u(t))=e^{it\Delta}u_0 + i \int_0^t e^{i(t-t')\Delta}N(u(t')) \,dt'
		\end{align} 
is a contraction on the set $\mathcal{S}$, where $N(u(t'))=(|x|^{-(N-\gamma)}\ast |u|^p)|u|^{p-2}u(t')$. Using \eqref{stri1} and \eqref{stri2}, we obtain
		\begin{align}\label{Lwp1}
		\|\Phi(u(t))\|_{L_t^{q_1}L_x^{r_1}}\lesssim \| u_0\|_{L_x^2}+\| N(u)\|_{L_t^{q_1^{\prime}}L_x^{r_1^{\prime}}}
		\end{align}
		and
		\begin{align}\label{Lwp2}
		\|\nabla\Phi(u(t))\|_{L_t^{q_1}L_x^{r_1}}\lesssim \|\nabla u_0\|_{L_x^2}+\|\nabla N(u)\|_{L_t^{q_1^{\prime}}L_x^{r_1^{\prime}}}.
		\end{align} 
Using H\"older's in time on the second term in \eqref{Lwp1} and \eqref{Lwp2}, we have
\begin{align*}
	\|N(u)\|_{L_t^{q_1^{\prime}}L_x^{r_1^{\prime}}}\lesssim T^{\theta}\|N(u)\|_{L_t^{q_1}L_x^{r_1^{\prime}}}\,\,\,\text{and}\,\,\,	\|\nabla N(u)\|_{L_t^{q_1^{\prime}}L_x^{r_1^{\prime}}}\lesssim T^{\theta}\|\nabla N(u)\|_{L_t^{q_1}L_x^{r_1^{\prime}}},
\end{align*}
where $\theta =\frac{(1-s_c)(p-1)}{p}$. Using H\"older's inequality, Lemma \ref{HLS} and Sobolev inequality, we estimate
		\begin{align}\label{Lwp3}\notag
		\|N(u)\|_{L_t^{q_1}L_x^{r_1^{\prime}}} &\lesssim \|(|x|^{-(N-\gamma)}\ast|u|^p)\|_{L_t^{q_1}L_x^{\frac{2N}{N-\gamma}}}\||u|^{p-2} u\|_{L_t^{\infty}L_x^{\frac{2Np}{(N+\gamma)(p-1)}}}\\\notag
		&\lesssim \|u\|^p_{L_t^{q_1 p}L_x^{r_1}}\|u\|^{p-1}_{L_t^{\infty}L_x^{r_1}}\\
		&\lesssim\|u\|^{2(p-1)}_{L_t^{\infty}L_x^{r_1}}\|u\|_{L_t^{q_1}L_x^{r_1}}\lesssim\|u\|^{2(p-1)}_{L_t^{\infty}H_x^{1}}\|u\|_{L_t^{q_1}L_x^{r_1}}
		\end{align}
		and (noting that the gradient lands on two different terms)
		\begin{align}\label{Lwp4}\notag
		\|\nabla N(u)\|_{L_t^{q_1}L_x^{r_1^{\prime}}} \lesssim&\,\, \|(|x|^{-(N-(\gamma-1))}\ast|u|^p)\|_{L_t^{q_1}L_x^{\frac{2N}{N-\gamma}}}\||u|^{p-2}u\|_{L_t^{\infty}L_x^{\frac{2Np}{(N+\gamma)(p-1)}}}\\\notag
		&+\|(|x|^{-(N-\gamma)}\ast|u|^p)\|_{L_t^{\infty}L_x^{\frac{2N}{N-\gamma}}}\||u|^{p-2}\nabla u\|_{L_t^{q_1}L_x^{\frac{2Np}{(N+\gamma)(p-1)}}}\\\notag
		\lesssim&\,\|u\|^p_{L_t^{q_1p}L_x^{\frac{2Np}{N+\gamma-2}}}\|u\|^{p-1}_{L_t^{\infty}L_x^{r_1}}+\|u\|^p_{L_t^{\infty}L_x^{r_1}}\|u\|^{p-2}_{L_t^{\infty}L_x^{r_1}}\|\nabla u\|_{L_t^{q_1}L_x^{r_1}}\\\notag
		\lesssim&\,\|u\|_{L_t^{q_1}L_x^{\frac{2Np}{N+\gamma-2p}}}\| u\|^{2(p-1)}_{L_t^{\infty}L_x^{r_1}}+\|u\|^{2(p-1)}_{L_t^{\infty}L_x^{r_1}}\|\nabla u\|_{L_t^{q_1}L_x^{r_1}}\\
		\lesssim&\,\|u\|^{2(p-1)}_{L_t^{\infty}H_x^1}\|\nabla u\|_{L_t^{q_1}L_x^{r_1}}.
		\end{align}
Combining \eqref{Lwp1} and \eqref{Lwp2}, respectively, with  \eqref{Lwp3} and \eqref{Lwp4}, we obtain
		\begin{align*}
		\|\Phi(u(t))\|_{L_t^{q_1}W_x^{1,r_1}}\lesssim \|u_0\|_{H^1_x}+T^{\theta}\|u\|^{2(p-1)}_{L_t^{\infty}H_x^1}\|u\|_{L_t^{q_1}W_x^{1,r_1}}.
		\end{align*}
Following a similar argument, we also have 
		\begin{align*}
		\|\Phi(u(t))\|_{L_t^{\infty}H_x^{1}}\lesssim \|u_0\|_{H^1_x}+T^{\theta}\|u\|^{2(p-1)}_{L_t^{\infty}H_x^1}\, \|u\|_{L_t^{q_1}W_x^{1,r_1}}.
		\end{align*}
Adding the last two lines, we get that  for $u\in\mathcal{S}$
\begin{align}\label{Lwp5}
\|\Phi(u(t))\|_{L_t^{q_1}W_x^{1,r_1}}+\|\Phi(u(t))\|_{L_t^{\infty}H_x^{1}}\leq C\|u_0\|_{H^1_x}+CT^{\theta}M^{2p-1}.
\end{align}
Set $M=2C\|u_0\|_{H^1_x}$ and take $T$ so that
		\begin{align}\label{Lwp6}
		CT^{\theta}M^{2(p-1)}\leq \frac{1}{2},
		\end{align}
		yielding that the right-hand side of \eqref{Lwp5} is bounded by $M$. Therefore, for $T\lesssim \|u_0\|^{-\frac{2p}{1-s_c}}_{H^1_x}$, we obtain $\Phi:\,\mathcal{S}\rightarrow\mathcal{S}.$ Note that the above estimate works for any $s_c<1$. In the energy-critical case, $s_c=1$, we have $\theta=0$, and thus, there is no time dependence in \eqref{Lwp5},
\begin{equation}\label{Lwp5a}
\|\Phi(u(t))\|_{L_t^{q_1}W_x^{1,r_1}}+\|\Phi(u(t))\|_{L_t^{\infty}H_x^{1}}\leq C\|u_0\|_{H^1_x}+CM^{2p-1}.
\end{equation} 
Hence, we can proceed only if $\|u_0\|_{H^1_x}$ is small enough, namely, if 
\begin{equation}\label{Lwp6a}
C\|u_0\|^{2(p-1)}_{H^1_x} < \frac{1}{2},
\end{equation}
which then bounds the right-hand side of \eqref{Lwp5a} by $M$: 
$C\|u_0\|_{H^1_x}+CM^{2p-1} < M$, yielding $\Phi$ mapping $\mathcal{S}$ into itself.

To complete the proof we need to show that the operator $\Phi$ is a contraction. This is achieved by running the same argument as above on the difference 
$$
d(\Phi(u(t)),\Phi(v(t))):=\|\Phi(u(t))-\Phi(v(t))\|_{L_t^{q_1}W_x^{1,r_1}}+\|\Phi(u(t))-\Phi(v(t))\|_{L_t^{\infty}H_x^{1}}
$$
for $u,v\in\mathcal{S}$. We again note that because of the convolution and also estimating at the $H^1$ level, we end up with extra terms to work unlike the proof for the mapping $\Phi$ into itself above (or which would be simply repeating that argument in the pure nonlinearity case of NLS). 

We first apply H\"older's in time to get
		\begin{align*}
			d(\Phi(u(t)),\Phi(v(t)))\lesssim&\,\, T^{\theta}\left(\|\Phi(u(t))-\Phi(v(t))\|_{L_t^{q_1}L_x^{r_1^{\prime}}}+\|\nabla(\Phi(u(t))-\Phi(v(t)))\|_{L_t^{q_1}L_x^{r_1^{\prime}}}\right),
		\end{align*}
		where
		\begin{align*}
			\|\Phi(u(t))-\Phi(v(t))\|_{L_t^{q_1}L_x^{r_1^{\prime}}}\lesssim&\,\,\|\big(|x|^{-(N-\gamma)}\ast |u|^p\big)\big(|u|^{p-2}u-|v|^{p-2}v\big)\|_{L_t^{q_1}L_x^{r_1^{\prime}}}\\
			&+\|\big(|x|^{-(N-\gamma)}\ast (|u|^p-|v|^p)\big)|v|^{p-2}v\|_{L_t^{q_1}L_x^{r_1^{\prime}}}\\
			=&\,\,A_1+A_2
		\end{align*}
		and
		\begin{align*}
			\|\nabla(\Phi(u(t))-\Phi(v(t)))\|_{L_t^{q_1}L_x^{r_1^{\prime}}}\lesssim&\,\,\|\nabla\big[\big(|x|^{-(N-\gamma)}\ast |u|^p\big)\big(|u|^{p-2}u-|v|^{p-2}v\big)\big]\|_{L_t^{q_1}L_x^{r_1^{\prime}}}\\
			&+\|\nabla\big[\big(|x|^{-(N-\gamma)}\ast (|u|^p-|v|^p)\big)|v|^{p-2}v\big]\|_{L_t^{q_1}L_x^{r_1^{\prime}}}\\
		=&\,\,B_1+B_2.
		\end{align*}
Here, we have added and subtracted the term $\big(|x|^{-(N-\gamma)}\ast |u|^p\big)|v|^{p-2}v$.
For $A_1$, we use H\"older's inequality, Lemma \ref{HLS} and \eqref{sdc1} to obtain
\begin{align}\label{Lwp11}\notag
          A_1&\lesssim\|(|x|^{-(N-\gamma)}\ast|u|^p)\|_{L_t^{q_1}L_x^{\frac{2N}{N-\gamma}}}\||u|^{p-2}u-|v|^{p-2}v\|_{L_t^{\infty}L_x^{\frac{2Np}{(N+\gamma)(p-1)}}}\\\notag
			&\lesssim \|u\|^p_{L_t^{q_1p}L_x^{r_1}}\Big(\|u\|^{p-2}_{L_t^{\infty}L_x^{r_1}}+\|v\|^{p-2}_{L_t^{\infty}L_x^{r_1}}\Big)\|u-v\|_{L_t^{\infty}L_x^{r_1}}\\\notag
			&\lesssim\|u\|^{p-1}_{L_t^{\infty}L_x^{r_1}}\|u\|_{L_t^{q_1}L_x^{r_1}}\Big(\|u\|^{p-2}_{L_t^{\infty}L_x^{r_1}}+\|v\|^{p-2}_{L_t^{\infty}L_x^{r_1}}\Big)\|u-v\|_{L_t^{\infty}L_x^{r_1}}\\
			&\lesssim\|u\|^{p-1}_{L_t^{\infty}H_x^{1}}\|u\|_{L_t^{q_1}L_x^{r_1}}\Big(\|u\|^{p-2}_{L_t^{\infty}H^1_x}+\|v\|^{p-2}_{L_t^{\infty}H^1_x}\Big)\|u-v\|_{L_t^{\infty}H^1_x}.
\end{align} 
We again use H\"older's, Lemma \ref{HLS} and \eqref{sdc5} to estimate $A_2$
\begin{align}\label{Lwp12}\notag
			A_2&\lesssim\||x|^{-(N-\gamma)}\ast(|u|^p-|v|^p)\|_{L_t^{q_1}L_x^{\frac{2N}{N-\gamma}}}\||v|^{p-2}v\|_{L_t^{\infty}L_x^{\frac{2Np}{(N+\gamma)(p-1)}}}\\\notag
			&\lesssim\||u|^p-|v|^p\|_{L_t^{q_1}L_x^{\frac{2N}{N+\gamma}}}\|v\|^{p-1}_{L_t^{\infty}L_x^{r_1}}\\\notag
			&\lesssim\Big(\|u\|^{p-1}_{L_t^{\infty}L_x^{r_1}}+\|v\|^{p-1}_{L_t^{\infty}L_x^{r_1}}\Big)\|u-v\|_{L_t^{q_1}L_x^{r_1}}\|v\|^{p-1}_{L_t^{\infty}L_x^{r_1}}\\
			&\lesssim\Big(\|u\|^{p-1}_{L_t^{\infty}H^1_x}+\|v\|^{p-1}_{L_t^{\infty}H^1_x}\Big)\|u-v\|_{L_t^{q_1}L_x^{r_1}}\|v\|^{p-1}_{L_t^{\infty}H_x^1}.
\end{align}
For $B_1$ we first use the product rule
		\begin{align*}
			B_1\lesssim&\,\,\|(|x|^{-(N-(\gamma-1))}\ast|u|^p)\|_{L_t^{q_1}L_x^{\frac{2N}{N-\gamma}}}\||u|^{p-2}u-|v|^{p-2}v\|_{L_t^{\infty}L_x^{\frac{2Np}{(N+\gamma)(p-1)}}}\\
			&+\|(|x|^{-(N-\gamma)}\ast|u|^p)\|_{L_t^{\infty}L_x^{\frac{2N}{N-\gamma}}}\|\nabla(|u|^{p-2}u-|v|^{p-2}v)\|_{L_t^{q_1}L_x^{\frac{2Np}{(N+\gamma)(p-1)}}},
		\end{align*}
	    then applying H\"older's inequlaity, Lemma \ref{HLS}, Sobolev inequality and \eqref{sdc1} yields
	    \begin{align}\label{Lwp13}\notag
	    	B_1\lesssim&\,\,\|u\|^p_{L_t^{q_1p}L_x^{\frac{2Np}{N+\gamma-2}}}\Big(\|u\|^{p-2}_{L_t^{\infty}L_x^{r_1}}+\|v\|^{p-2}_{L_t^{\infty}L_x^{r_1}}\Big)\|u-v\|_{L_t^{\infty}L_x^{r_1}}\\\notag
	    	&+\|u\|^p_{L_t^{\infty}L_x^{r_1}}\Big(\|u\|^{p-2}_{L_t^{\infty}L_x^{r_1}}+\|v\|^{p-2}_{L_t^{\infty}L_x^{r_1}}\Big)\|\nabla(u-v)\|_{L_t^{q_1}L_x^{r_1}}\\\notag
	    	\lesssim&\,\,\|u\|_{L_t^{q_1}L_x^{\frac{2Np}{N+\gamma-2p}}}\|u\|^{p-1}_{L_t^{\infty}L_x^{r_1}}\Big(\|u\|^{p-2}_{L_t^{\infty}L_x^{r_1}}+\|v\|^{p-2}_{L_t^{\infty}L_x^{r_1}}\Big)\|u-v\|_{L_t^{\infty}L_x^{r_1}}\\\notag
	    	&+\|u\|^p_{L_t^{\infty}L_x^{r_1}}\Big(\|u\|^{p-2}_{L_t^{\infty}L_x^{r_1}}+\|v\|^{p-2}_{L_t^{\infty}L_x^{r_1}}\Big)\|\nabla(u-v)\|_{L_t^{q_1}L_x^{r_1}}\\\notag
	    	\lesssim&\,\,\|\nabla u\|_{L_t^{q_1}L_x^{r_1}}\|u\|^{p-1}_{L_t^{\infty}H^1_x}\Big(\|u\|^{p-2}_{L_t^{\infty}H^1_x}+\|v\|^{p-2}_{L_t^{\infty}H^1_x}\Big)\|u-v\|_{L_t^{\infty}H^1_x}\\
	    	&+\|u\|^p_{L_t^{\infty}H_x^1}\Big(\|u\|^{p-2}_{L_t^{\infty}H_x^1}+\|v\|^{p-2}_{L_t^{\infty}H^1_x}\Big)\|\nabla(u-v)\|_{L_t^{q_1}L_x^{r_1}}.
	    \end{align}
Again using the product rule and Lemma \ref{HLS} to estimate $B_2$, we get
	    \begin{align*}
	    	B_2\lesssim&\,\,\||x|^{-(N-(\gamma-1))}\ast(|u|^p-|v|^p)\|_{L_t^{q_1}L_x^{\frac{2N}{N-\gamma}}}\||v|^{p-2}v\|_{L_t^{\infty}L_x^{\frac{2Np}{(N+\gamma)(p-1)}}}\\
	    	&+\||x|^{-(N-\gamma)}\ast(|u|^p-|v|^p)\|_{L_t^{\infty}L_x^{\frac{2N}{N-\gamma}}}\|\nabla(|v|^{p-2}v)\|_{L_t^{q_1}L_x^{\frac{2Np}{(N+\gamma)(p-1)}}}\\
	    	\lesssim&\,\,\||u|^p-|v|^p\|_{L_t^{q_1}L_x^{\frac{2N}{N+\gamma-2}}}\|v\|^{p-1}_{L_t^{\infty}L_x^{r_1}}+\||u|^p-|v|^p\|_{L_t^{\infty}L_x^{\frac{2N}{N+\gamma}}}\|v\|^{p-2}_{L_t^{\infty}L_x^{r_1}}\|\nabla v\|_{L_t^{q_1}L_x^{r_1}}.
	    \end{align*}
Using \eqref{sdc5} and Sobolev, we obtain
	    \begin{align}\label{Lwp14}\notag
	    	B_2\lesssim&\,\,\Big(\|u\|^{p-1}_{L_t^{\infty}L_x^{r_1}}+\|v\|^{p-1}_{L_t^{\infty}L_x^{r_1}}\Big)\|u-v\|_{L_t^{q_1}L_x^{\frac{2Np}{N+\gamma-2p}}}\|v\|^{p-1}_{L_t^{\infty}L_x^{r_1}}\\\notag
	    	&+\Big(\|u\|^{p-1}_{L_t^{\infty}L_x^{r_1}}+\|v\|^{p-1}_{L_t^{\infty}L_x^{r_1}}\Big)\|u-v\|_{L_t^{\infty}L_x^{r_1}}\|v\|^{p-2}_{L_t^{\infty}L_x^{r_1}}\|\nabla v\|_{L_t^{q_1}L_x^{r_1}}\\\notag
	    	\lesssim&\,\,\Big(\|u\|^{p-1}_{L_t^{\infty}H^1_x}+\|u\|^{p-1}_{L_t^{\infty}H^1_x}\Big)\Big(\|\nabla(u-v)\|_{L_t^{q_1}L_x^{r_1}}\|v\|^{p-1}_{L_t^{\infty}H^1_x}\\
	    	&+\|u-v\|_{L_t^{\infty}H^1_x}\|v\|^{p-2}_{L_t^{\infty}H^1_x}\|\nabla v\|_{L_t^{q_1}L_x^{r_1}}\Big).
	    \end{align}
Combining \eqref{Lwp11}, \eqref{Lwp12}, \eqref{Lwp13} and \eqref{Lwp14}, we obtain that for $u,v\in \mathcal{S}$
	    \begin{align*}
	    	d(\Phi(u(t)),\Phi(v(t)))\lesssim T^{\theta}M^{2(p-1)}d(u,v).
	    \end{align*}
This together with \eqref{Lwp6}, the bound on time $T$, implies that $\Phi$ is a contraction on $\mathcal{S}$ for the energy-subcritical case. Similarly, for the energy-critical case, we have that for $u,v\in \mathcal{S}$
\begin{align*}
d(\Phi(u(t)),\Phi(v(t)))\lesssim M^{2(p-1)}d(u,v),
\end{align*} 
which with the smallness of \eqref{Lwp6a} implies that $\Phi$ is again a contraction on $\mathcal{S}$. To prove the continuous dependence with respect to $u_0$, we note that if $u$ and $v$ are the corresponding solutions of \eqref{duhamel} with initial data $u_0$ and $v_0$, respectively, then
$$
u(t)-v(t) = e^{it\Delta}(u_0-v_0) + i\int_{0}^{t}e^{i(t-t')\Delta}(N(u)-N(v))(t') \,dt'.
$$
Thus, the same argument as in \eqref{Lwp11}, \eqref{Lwp12}, \eqref{Lwp13} and \eqref{Lwp14} (and the appropriate modifications when $s_c=1$) yields 
\begin{align*}
d(u(t),v(t))&:=\|u(t)-v(t)\|_{L_t^{q_1}W_x^{1,r_1}}+\|u(t)-v(t)\|_{L_t^{\infty}H_x^{1}}\\
	&\lesssim \|u_0-v_0\|_{H^1}+C_{N,p\gamma}T^{\theta}M^{2(p-1)}d(u(t),v(t)).
\end{align*}
This implies that if $\|u_0-v_0\|_{H^1}$ is small enough (see \eqref{Lwp6} or \eqref{Lwp6a}), we have that
	\begin{align*}
		d(u(t),v(t)\leq \widetilde{C}\|u_0-v_0\|_{H^1},
	\end{align*} 
which completes the proof.
\end{proof}
	
	
\section{Small data theory}\label{small}
As we now have the $H^1$ local well-posedness, we investigate the global existence of small data and scattering in $H^1$. At the end of this section we also include the long-time perturbation argument. This may appear to be standard, however, we give a careful and detailed proof demonstrating how we tackle the nonlocal potential term. 
In this section we consider the integral equation \eqref{duhamel} with $u_0\in H^1(\R^N)$ and $0<\gamma<N$ with $p\geq 2$ satisfying 
\begin{align}\label{prange}
\begin{cases}
1+\frac{\gamma+2}{N} \leq p <1+\frac{\gamma+2}{N-2},&\text{if} ~N\geq 3\\
1+\frac{\gamma+2}{N} \leq p <\infty,& \text{if} ~N=1,2.	
\end{cases}
\end{align} 

In the energy-subcritical case ($s_c<1$) it is possible to obtain $\dot{H}^{s_c}$ small data theory, replacing the right-hand side bound below in \eqref{E:H^s-level}	with the $\dot{H}^{s_c}$ norm (instead of $H^1$ norm) as done in \cite{HR07}, \cite{G}. This requires fractional derivatives, introduction of different Strichartz pairs and considering different cases of smoothness, depending on $p$ and $s_c$; it is done in \cite{AKA2}. For the purpose of this paper, it suffices to have $H^1$ small data, and thus, we consider the bound on the right-hand side of \eqref{E:H^s-level} by the full $H^1$ norm. Also note that while the norm on the left-hand side of \eqref{E:H^s-level} is at the $H^{s_c}$ level, it can be replaced 
with the norms at the $H^1$ level, that is by $\|u\|_{S(L^2)} +\|\nabla u\|_{S(L^2)}$ (by the interpolation and then separating it into the sum by Peter-Paul), which we will do in the proof. For brevity, we chose to state \eqref{E:H^s-level} at the $H^{s_c}$ level. Furthermore, we note that the Proposition \ref{smalldata} also holds true for the $L^2$-critical equations ($s_c=0$) with $u_0\in H^1(\R^N)$ and \eqref{E:H^s-level} reduces just to one condition \eqref{base-level}. We also mention that one would need to use different Strichartz pairs to obtain small data theory for the energy-critical case ($s_c=1$), which is possible but beyond the scope of this paper.
   
\begin{proposition}[Small data theory in $H^1$]\label{smalldata}
Let $p\geq 2$ satisfy \eqref{prange} with $0<\gamma<N$ and $u_0\in H^1(\R^N)$. Suppose $\|u_0\|_{H^1} \leq A$. There exists $\delta = \delta(A)>0$ such that if $\|e^{it\Delta}u_0\|_{S(\dot{H}^{s_c})} \leq \delta$, then there exists a unique global solution $u$ of \eqref{gH} in $H^1(\R^N)$ such that
\begin{equation}\label{base-level}
\|u\|_{S(\dot{H}^{s_c})} \leq 2\|e^{it\Delta}u_0\|_{S(\dot{H}^{s_c})}
\end{equation}
and
\begin{equation}\label{E:H^s-level}
\| |\nabla|^{s_c}u\|_{S(L^2)} \leq 2\,c\, \|u_0\|_{{H}^{1}},
\end{equation}
where $c$ depends on constants from the Gagliardo-Nirenberg interpolation estimate and the Strichartz inequality.
\end{proposition}

\begin{proof}
First, note that by Strichartz \eqref{sobstri1} and Sobolev estimates, we can track the dependence of $\delta$ on $A$ (if needed, splitting the time interval). 
Next, denote
$$
B=\left\{ u ~:\, \|u\|_{S(\dot{H}^{s_c})} \leq 2\,\|e^{it\Delta}u_0\|_{S(\dot{H}^{s_c})} \peq\text{and}\peq \||\nabla|^{s_c}u\|_{S(L^2)} \leq 2\,c\|u_0\|_{H^{1}}
\right\},
$$
and define
\begin{align}\label{sd0}
\Phi_{u_0}(u) = e^{it\Delta}u_0 + i \int_0^t e^{i(t-t')\Delta} F(u(t'))\,dt', \quad \mbox{where} \quad F(u)=(|x|^{-(N-\gamma)}\ast|u|^p)|u|^{p-2}u.
\end{align}
Applying the triangle inequality and \eqref{Katostri} to \eqref{sd0}, we obtain
\begin{align}\label{sd1}
 \|\Phi_{u_0}(u)\|_{S(\dot{H}^{s_c})} \leq \|e^{it\Delta}u_0\|_{S(\dot{H}^{s_c})} + c\|F(u)\|_{S'(\dot{H}^{-s_c})}.
 \end{align}
Using the pair $(q_3^{\prime}, r_1^{\prime})$,  H\"older's inequality yields
\begin{align}\label{sd2}
	\|F(u)\|_{L_t^{q_3^{\prime}}L_x^{r_1^{\prime}}}\leq \||x|^{-(N-\gamma)}\ast|u|^p\|_{L_t^{\frac{2}{1-s_c}}L_x^{\frac{2N}{N-\gamma}}}\|u\|^{p-1}_{L_t^{q_2}L_x^{r_1}}.
\end{align}
Applying Lemma \ref{HLS} for $N>\gamma$, we estimate
\begin{align}\label{sd3}
\||x|^{-(N-\gamma)}\ast|u|^p\|_{L_t^{\frac{2}{1-s_c}}L_x^{\frac{2N}{N-\gamma}}}\leq c_{N,p,\gamma}\|u\|^p_{L_t^{q_2}L_x^{r_1}}.
\end{align}
Using \eqref{sd3}, we can write the estimate \eqref{sd2} as
\begin{align}\label{sd4}
\|F(u)\|_{S'(\dot{H}^{-s_c})} \leq c_{N,p,\gamma} \|u\|^p_{S(\dot{H}^{s_c})} \|u\|^{p-1}_{S(\dot{H}^{s_c})}.
\end{align} 
Thus, for $u\in B$, \eqref{sd4} gives
\begin{align}\label{sd5}
\|F(u)\|_{S'(\dot{H}^{-s_c})} \leq c_{N,p,\gamma} \,2^{2p-1}\, \|e^{it\Delta}u_0\|_{S(\dot{H}^{s_c})}^{2p-1}.
\end{align}
Inserting \eqref{sd5} into \eqref{sd1} and redefining the constant $c_{N,p, \gamma}c=:c_1$, we have
\begin{align*}
\|\Phi_{u_0}(u)\|_{S(\dot{H}^{s_c})} \leq \|e^{it\Delta}u_0\|_{S(\dot{H}^{s_c})}\left(1+c_1\, 2^{2p-1}\, \|e^{it\Delta}u_0\|_{S(\dot{H}^{s_c})}^{2(p-1)}\right),
\end{align*}
and thus, we need
$$
c_1 \, 2^{2p-1}\, \|e^{it\Delta}u_0\|_{S(\dot{H}^{s_c})}^{2(p-1)}\leq 1.
$$
To estimate $\| |\nabla|^{s_c}\Phi_{u_0}(u)\|_{S(L^2)}$, 
we recall the Gagliardo-Nirenberg interpolation inequality
$$
\| |\nabla|^{s_c} v \|_{L^2} \leq c_{GN}\|\nabla v \|_{L^2}^{s_c}\| v \|_{L^2}^{1-s_c},
$$
and taking $v = \Phi_{u_0}(u)$, we bound the $L^2$ and $\dot{H}^1$ norms as follows:
\begin{align}\label{sd6}
\|\Phi_{u_0}(u)\|_{S(L^2)} \leq c\,\|u_0\|_{L^2} + c\,\|F(u)\|_{S^{\prime}(L^2)}.
\end{align}
From H\"older's inequality, we get
\begin{align}\label{sd7}
\|F(u)\|_{L_t^{q_1^{\prime}}L_x^{r_1^{\prime}}} \leq \||x|^{-(N-\gamma)} \ast |u|^p \|_{L_t^{2} L_x^{\frac{2N}{N-\gamma}}}\|u\|^{p-1}_{L_t^{q_2}L_x^{r_1}}.
\end{align}
We estimate the convolution term in \eqref{sd7} again by Lemma \ref{HLS} for $N>\gamma$ and then use H\"older's to obtain 
\begin{align}\label{sd8}\notag
\|F(u)\|_{S^{\prime}(L^2)}&\leq c_{N,p, \gamma}\| \, |u|^p\|_{L_t^2L_x^\frac{2N}{N+\gamma}}\|u\|^{p-1}_{L_t^{q_2}L_x^{r_1}}\\ \notag
 &\leq c_{N,p,\gamma}\|u\|^{p-1}_{L_t^{q_2}L_x^{r_1}}\|u\|_{L_t^{q_1}L_x^{r_1}}\|u\|^{p-1}_{L_t^{q_2}L_x^{r_1}}\\
 &\leq c_{N,p,\gamma}\|u\|^{2(p-1)}_{S(\dot{H}^{s_c})}\|u\|_{S(L^2)}.
\end{align} 
Using \eqref{sobstri2} (and triangle inequality) in \eqref{sd0}, we get
\begin{align}\label{sd9}
\|\nabla\Phi_{u_0}(u)\|_{S(L^2)} \leq c\,\|\nabla u_0\|_{L^2} + c\,\|\nabla F(u)\|_{S^{\prime}(L^2)},
\end{align}
where the nonlinear term is estimated as
\begin{align}\label{sd10}\notag
\|\nabla F(u)\|_{L_t^{q_1^{\prime}}L_x^{r_1^{\prime}}}&\leq \||x|^{-(N-\gamma)}\ast|u|^p\|_{L_t^{\frac{2}{1-s_c}}L_x^{\frac{2N}{N-\gamma}}}\|\nabla(|u|^{p-2}u)\|_{L_t^{\frac{2p}{p-(1-s_c)}}L_x^{\frac{2Np}{(N+\gamma)(p-1)}}}\\\notag
 	&+\||x|^{-(N-(\gamma-1))}\ast|u|^p\|_{L_t^{2}L_x^{\frac{2N}{N-\gamma}}}\|u\|^{p-1}_{L_t^{q_2}L_x^{r_1}}\\\notag
 	\leq
 	&\,\, c_{N,p,\gamma}\|u\|^{2(p-1)}_{L_t^{q_2}L_x^{r_1}}\|\nabla u\|_{L_t^{q_1}L_x^{r_1}}+c_{N,p,\gamma}\|u\|^p_{L_t^{2p}L_x^{\frac{2Np}{N+\gamma-2}}}\|u\|^{p-1}_{L_t^{q_2}L_x^{r_1}}\\\notag
 	\leq&\,\, c_{N,p,\gamma}\|u\|^{2(p-1)}_{L_t^{q_2}L_x^{r_1}}\|\nabla u\|_{L_t^{q_1}L_x^{r_1}}+c_{N,p,\gamma}\|\nabla u\|_{L_t^{q_1}L_x^{r_1}}\|u\|^{2(p-1)}_{L_t^{q_2}L_x^{r_1}}\\
 	\leq&\,\,2c_{N,p,\gamma}\|u\|^{2(p-1)}_{S(\dot{H}^{s_c})}\|\nabla u\|_{S(L^2)}.
\end{align}
Combining \eqref{sd6} and \eqref{sd9}, and applying \eqref{sd8} and \eqref{sd10}, we obtain
\begin{align}\label{sd11}\notag
 	\|\Phi_{u_0}(u)\|_{S(L^2)} + \|\nabla\Phi_{u_0}(u)\|_{S(L^2)} \leq c&\left(\|u_0\|_{L^2}+\|\nabla u_0\|_{L^2}\right)  \\\notag
 	+ &c_1\|u\|^{2(p-1)}_{S(\dot{H}^{s_c})}\left(\|u\|_{S(L^2)}+\|\nabla u\|_{S(L^2)}\right)\\\notag
 	\leq c&\|u_0\|_{H^1}+2^{2p-1}c_1c\|e^{it\Delta}u_0\|^{2(p-1)}_{S(\dot{H}^{s_c})}\|u_0\|_{H^1} \\
 	\leq c&\|u_0\|_{H^1}\left(1+2^{2p-1}c_1\|e^{it\Delta}u_0\|^{2(p-1)}_{S(\dot{H}^{s_c})}\right),
\end{align}
where $c_{N,p, \gamma}c=:c_1$. 
Now, if we take 
$$
2^{2p-1}\,c_1\,\|e^{it\Delta}u_0\|^{2(p-1)}_{S(\dot{H}^{s_c})}\leq 1,
$$
and recalling that $\ds \|e^{it\Delta}u_0\|_{S(\dot{H}^{s_c})} <\delta$,  then \eqref{sd11} would give the required bound for the space $B$: $2c \|u_0\|_{H^1}$.
Hence, choosing $\delta < \delta_0 =\frac{1}{2}\sqrt[2(p-1)]{\frac{1}{2c_1}}$ implies that $\Phi_{u_0}\in B$. Now we show that $\Phi_{u_0}(u)$ is a contraction on $B$ with the metric
$$
d(u,v)=\|u-v\|_{S(L^2)}+\|\nabla(u-v)\|_{S(L^2)}+\|u-v\|_{S(\dot{H}^{s_c})}.
$$ 
(The last norm is included for convenience.)
For $u$, $v\in B$, by Strichartz estimates \eqref{Katostri} and \eqref{stri2}, we obtain
\begin{align}\label{sd12}
\|\Phi_{u_0}(u)-\Phi_{u_0}(v)\|_{S(\dot{H}^{s_c})}&\leq c \|F(u)-F(v)\|_{S'(\dot{H}^{-s_c})}
\end{align}
and
\begin{align}\label{sd13}
\|(1+\nabla)(\Phi_{u_0}(u)-\Phi_{u_0}(v))\|_{S(L^2)}&\leq c\|(1+\nabla)(F(u)-F(v))\|_{S'(L^2)}.
\end{align}
The triangle inequality applied to the right-hand side of \eqref{sd12} yields 
\begin{align*}
\|F(u)-F(v)\|_{S'(\dot{H}^{-s_c})}\leq&\,\, \|\big(|x|^{-(N-\gamma)}\ast |u|^p\big)\big(|u|^{p-2}u-|v|^{p-2}v\big)\|_{S^{\prime}(\dot{H}^{-s_c})}\\
&+\|\big(|x|^{-(N-\gamma)}\ast (|u|^p-|v|^p)\big)|v|^{p-2}v\|_{S^{\prime}(\dot{H}^{-s_c})},
\end{align*} 
where we have added and subtracted the term $\big(|x|^{-(N-\gamma)}\ast |u|^p\big)|v|^{p-2}v$ to the difference. Using \eqref{sdc1}, \eqref{sdc5} and calculations in \eqref{sd2}, we obtain
\begin{align*}
\|F(u)-F(v)\|_{S'(\dot{H}^{-s_c})}	\leq&\,\,c_{N,p,\gamma}\||x|^{-(N-\gamma)}\ast |u|^p\|_{L_t^{\frac{q_2}{p}}L_x^{\frac{2N}{N-\gamma}}}\||u|^{p-2}u-|v|^{p-2}v\|_{L_t^{\frac{q_2}{p-1}}L_x^{\frac{r_1}{p-1}}}\\\notag
&+c_{N,p,\gamma}\||x|^{-(N-\gamma)}\ast(|u|^p-|v|^p)\|_{L_t^{\frac{q_2}{p}}L_x^{\frac{2N}{N-\gamma}}}\|v\|^{p-1}_{S(\dot{H}^{s_c})}\\\notag
\leq&\,\,c_{N,p,\gamma}\|u\|^p_{L_t^{q_2}L_x^{r_1}}\Big(\|u\|^{p-2}_{L_t^{q_2}L_x^{r_1}}+\|v\|_{L_t^{q_2}L_x^{r_1}}\Big)\|u-v\|_{L_t^{q_2}L_x^{r_1}}\\\notag
&+c_{N,p,\gamma}\||u|^p-|v|^p\|_{L_t^{\frac{q_2}{p}}L_x^{\frac{r_1}{p}}}\|v\|^{p-1}_{S(\dot{H}^{s_c})}\\\notag
\leq&\,\,c_{N,p,\gamma}\|u\|^p_{S(\dot{H}^{s_c})}\Big(\|u\|^{p-2}_{S(\dot{H}^{s_c})}+\|v\|^{p-2}_{S(\dot{H}^{s_c})}\Big)\|u-v\|_{S(\dot{H}^{s_c})}\\\notag
&+c_{N,p,\gamma}\Big(\|u\|^{p-1}_{S(\dot{H}^{s_c})}+\|v\|^{p-1}_{S(\dot{H}^{s_c})}\Big)\|u-v\|_{S(\dot{H}^{s_c})}\|v\|^{p-1}_{S(\dot{H}^{s_c})}.
\end{align*}
For $u$, $v\in B$, we have that
\begin{align}\label{sd14}
\|F(u)-F(v)\|_{S'(\dot{H}^{-s_c})}\leq 2^{2p}c_{N,p,\gamma}\|e^{it\Delta}u_0\|^{2(p-1)}_{S(\dot{H}^{s_c})}\|u-v\|_{S(\dot{H}^{s_c})}.
\end{align}
Combining \eqref{sd12} with \eqref{sd14}, we obtain
\begin{align}\label{sd15}
\|\Phi_{u_0}(u)-\Phi_{u_0}(v)\|_{S(\dot{H}^{s_c})}&\leq2^{2p}c_1\|e^{it\Delta}u_0\|^{2(p-1)}_{S(\dot{H}^{s_c})}\|u-v\|_{S(\dot{H}^{s_c})}.
\end{align}
Next, we estimate the difference from \eqref{sd13} using again the triangle inequality  and H\"older's 
\begin{align*}
\|F(u)-F(v)\|_{S'(L^2)}\leq&\,\, \|\big(|x|^{-(N-\gamma)}\ast |u|^p\big)\big(|u|^{p-2}u-|v|^{p-2}v\big)\|_{S'(L^2)}\\\notag
&+\|\big(|x|^{-(N-\gamma)}\ast (|u|^p-|v|^p)\big)|v|^{p-2}v\|_{S'(L^2)}\\
\leq&\,\,c_{N,p,\gamma}\||x|^{-(N-\gamma)}\ast|u|^p\|_{L_t^{\frac{q_2}{p}}L_x^{\frac{2N}{N-\gamma}}}\||u|^{p-2}u-|v|^{p-2}v\|_{L_t^{\frac{2p}{p-1+s_c}}L_x^{\frac{r_1}{p-1}}}\\\notag
&+c_{N,p,\gamma}\||x|^{-(N-\gamma)}\ast(|u|^p-|v|^p)\|_{L_t^{2}L_x^{\frac{2N}{N-\gamma}}}\|v\|^{p-1}_{L_t^{q_2}L_x^{r_1}}.
\end{align*}
Apply \eqref{sdc1}, \eqref{sdc5} and calculations in \eqref{sd7}, \eqref{sd8} on the right-hand side of above estimate to obtain
\begin{align*}
\|F(u)-F(v)\|_{S'(L^2)}\leq&\,\,c_{N,p,\gamma}\|u\|^p_{L_t^{q_2}L_x^{r_1}}\Big(\|u\|^{p-2}_{L_t^{q_2}L_x^{r_1}}+\|v\|^{p-2}_{L_t^{q_2}L_x^{r_1}}\Big)\|u-v\|_{L_t^{q_1}L_x^{r_1}}\\\notag
&+c_{N,p,\gamma}\||u|^p-|v|^p\|_{L_t^{2}L_x^{\frac{r_1}{p}}}\|v\|^{p-1}_{L_t^{q_2}L_x^{r_1}}\\\notag
\leq&\,\,c_{N,p,\gamma}\|u\|^{p}_{S(\dot{H}^{s_c})}\Big(\|u\|^{p-2}_{S(\dot{H}^{s_c})}+\|v\|^{p-2}_{S(\dot{H}^{s_c})}\Big)\|u-v\|_{S(L^2)}\\\notag
&+c_{N,p,\gamma}\Big(\|u\|^{p-1}_{S(\dot{H}^{s_c})}+\|v\|^{p-1}_{S(\dot{H}^{s_c})}\Big)\|u-v\|_{S(L^2)}\|v\|^{p-1}_{S(\dot{H}^{s_c})}.
\end{align*}
For $u$, $v\in B$, we have
\begin{align}\label{sd16}
\|F(u)-F(v)\|_{S'(L^2)}\leq2^{2p}c_{N,p,\gamma}\|e^{it\Delta}u_0\|^{2(p-1)}_{S(\dot{H}^{s_c})}\|u-v\|_{S(L^2)}.
\end{align}
Combining \eqref{sd13} with \eqref{sd16}, we obtain
\begin{align}\label{sd21}
\|\Phi_{u_0}(u)-\Phi_{u_0}(v)\|_{S(L^2)}&\leq 2^{2p}c_1\|e^{it\Delta}u_0\|^{2(p-1)}_{S(\dot{H}^{s_c})}\|u-v\|_{S(L^2)}.
\end{align}
Finally, estimating the difference in \eqref{sd13} with the gradient, we obtain 
\begin{align}\label{sd17}
\|\nabla(F(u)-F(v))\|_{S'(L^2)}\leq&\,\,\|\nabla\big[\big(|x|^{-(N-\gamma)}\ast |u|^p\big)\big(|u|^{p-2}u-|v|^{p-2}v\big)\big]\|_{S'(L^2)}\\\label{sd18}
&+\|\nabla\big[\big(|x|^{-(N-\gamma)}\ast (|u|^p-|v|^p)\big)|v|^{p-2}v\big]\|_{S'(L^2)}.
\end{align}
Using \eqref{sd1} along with the calculations for \eqref{sd11} and embedding $\dot{W}^{1,\frac{2Np}{N+\gamma}}\hookrightarrow L^{\frac{2Np}{N+\gamma-2p}}$, we get 
\begin{align*}
	\eqref{sd17}&\leq 2c_{N,p,\gamma}\|u\|^{p}_{S(\dot{H}^{s_c})}\Big(\|u\|^{p-2}_{S(\dot{H}^{s_c})}+\|v\|^{p-2}_{S(\dot{H}^{s_c})}\Big)\|\nabla(u-v)\|_{S(L^2)}.
\end{align*}
Similarly, we obtain
\begin{align*}
	\eqref{sd18}&\leq c_{N,p,\gamma}\Big(\|u\|^{p-1}_{S(\dot{H}^{s_c})}+\|v\|^{p-1}_{S(\dot{H}^{s_c})}\Big)\|\nabla(u-v)\|_{S(L^2)}\|v\|^{p-1}_{S(\dot{H}^{s_c})}\\
	&+c_{N,p,\gamma}\Big(\|u\|^{p-1}_{S(\dot{H}^{s_c})}+\|v\|^{p-1}_{S(\dot{H}^{s_c})}\Big)\|u-v\|_{S(\dot{H}^{s_c})}\|\nabla v\|_{S(L^2)} \|v\|^{p-2}_{S(\dot{H}^{s_c})}.
\end{align*}
Then for $u$, $v\in B$, we have
\begin{align}\label{sd19}\notag
	\|\nabla(\Phi_{u_0}&(u)-\Phi_{u_0}(v))\|_{S(L^2)}\leq c\|\nabla(F(u)-F(v))\|_{S'(L^2)}\\
	\leq 2^{2p}&c_1\|e^{it\Delta}u_0\|^{2(p-1)}_{S(\dot{H}^{s_c})}\|\nabla(u-v)\|_{S(L^2)}+2^{2p-1}c_1\|u_0\|_{H^1}\|e^{it\Delta}u_0\|^{2(p-1)}_{S(\dot{H}^{s_c})}\|u-v\|_{S(\dot{H}^{s_c})}.
\end{align} 
From \eqref{sd15}, \eqref{sd16} and \eqref{sd19}, we get
$$
d(\Phi_{u_0}(u),\Phi_{u_0}(v))\leq 2^{2p-1}c_1\|u_0\|_{H^1}\|e^{it\Delta}u_0\|^{2(p-1)}_{S(\dot{H}^{s_c})}d(u,v)\leq \frac{1}{2}d(u,v)
$$
for $\delta_1\leq \frac{1}{2}\sqrt[2(p-1)]{\frac{1}{2c_1A}}.$ 
Finally, taking $\delta \leq \min(\delta_0,\delta_1)$ concludes that $\Phi_{u_0}$ is a contraction.
\end{proof}

Next we establish the scattering in $H^1(\R^N)$.

\begin{theorem}[$H^1$ scattering]\label{H1scatter}
Let $u(t)$ be a global solution to \eqref{gH} with initial data $u_0\in H^1(\R^N)$. If $\|u\|_{S(\dot{H}^{s_c})} < +\infty$ (globally finite $\dot{H}^{s_c}$ Strichartz norm) and $\sup_{t\in \R^+}\|u(t)\|_{H^1}\leq B$ (uniformly bounded $H^1(\R^N)$ norm). Then $u(t)$ scatters in $H^1(\R^N)$ as $t\rightarrow +\infty$, i.e., there exists $\peq u^+ \in H^1(\R^N)$ such that
$$
\lim_{t\rightarrow +\infty} \|u(t)-e^{it\Delta}u^+\|_{H^1}=0.
$$
\end{theorem}

\begin{proof}
The assumption $\|u\|_{S(\dot{H}^{s_c})} < +\infty$ implies that there exists $M$ such that
$$
M = \|u\|_{L_t^{\frac{2p}{1-s_c}}L_x^{\frac{2Np}{N+\gamma}}} < +\infty.
$$
Recall that $\left(\frac{2p}{1-s_c},\frac{2Np}{N+\gamma}\right)$ is an $\dot{H}^{s_c}$-admissible pair. Let $\widetilde{M} = M^{\frac{2p}{1-s_c}}$. Given $\delta>0$ we can decompose $[0,+\infty) = \cup_{j=1}^{\widetilde{M}}I_j$, where $I_j=[t_j,t_{j+1})$ such that for each $j$, we have
$$
\|u\|_{L_{I_j}^{\frac{2p}{1-s_c}}L_x^{\frac{2Np}{N+\gamma}}} < \delta.
$$
Hence, by the triangle inequality and Strichartz estimates \eqref{stri1} and \eqref{stri2} applied to the integral equation \eqref{duhamel} on $I_j$, we have
\begin{align}\label{scat1}
\|u\|_{S(L^2;I_j)} &\leq c \|u(t_j)\|_{L^2} +c\|(|x|^{-(N-\gamma)}\ast|u|^p)|u|^{p-2}u\|_{S^{\prime}(L^2;I_j)}.
\end{align}
From \eqref{sd8}, we have
\begin{align}\label{scat2}
\|(|x|^{-(N-\gamma)}\ast|u|^p)|u|^{p-2}u\|_{S^{\prime}(L^2;I_j)}\leq c_{N,p,\gamma}\|u\|^{2(p-1)}_{S(\dot{H}^{s_c};I_j)}\|u\|_{S(L^2;I_j)}. 
\end{align}
Thus, \eqref{scat1} combined with \eqref{scat2} and the assumption $\sup_{t\in \R^+}\|u(t)\|_{H^1}\leq B$ implies 
\begin{align}\label{scat3}
	\|u\|_{S(L^2;I_j)} \leq cB + c_1\delta^{2(p-1)}\|u\|_{S(L^2;I_j)}.
\end{align}
Similarly, using Strichartz estimates \eqref{sobstri1} and \eqref{sobstri2} for $s=1$ along with \eqref{sd10} yields
\begin{align}\label{scat4}\notag
\|\nabla u\|_{S(L^2;I_j)}&\leq  c\|\nabla u(t_j)||_L^2 + c\|\nabla\big((|x|^{-(N-\gamma)}\ast|u|^p)|u|^{p-2}u\big)\|_{S'(L^2;I_j)}\\\notag
&\lesssim cB + 2c_1\|u\|^{2(p-1)}_{S(\dot{H}^{s_c};I_j)}\|\nabla u\|_{S(L^2;I_j)}\\
&\lesssim cB + 2c_1\delta^{2(p-1)}\|\nabla u\|_{S(L^2;I_j)}.
\end{align}
Combining \eqref{scat3} and \eqref{scat4}, we get
$$
\|u\|_{S(L^2;I_j)}+\|\nabla u\|_{S(L^2;I_j)} \leq 2\,c\,B + 2\,c_1\, \delta^{2(p-1)}\left(\|u\|_{S(L^2;I_j)}+\|\nabla u\|_{S(L^2;I_j)}\right).
$$
Performing the summation over $I_j$, we obtain
\begin{align*}
\|u\|_{S(L^2)}+\|\nabla u\|_{S(L^2)} &\leq 2\,c\,BM^{\frac{2p}{1-s_c}} + 2\,c_1\,\delta^{2(p-1)}\left(\|u\|_{S(L^2)}+\|\nabla u\|_{S(L^2)}\right),
\end{align*}
which implies that
$$
\left(1-2c_1\delta^{2(p-1)}\right)\left(\|u\|_{S(L^2)}+\|\nabla u\|_{S(L^2)}\right) \lesssim 2\,c\,BM^{\frac{2p}{1-s_c}}.
$$
Thus, for small $\delta$, we require that
$
1-2\delta^{2(p-1)} \leq \frac{1}{2},
$
so that
\begin{align}\label{bound}
\|u\|_{S(L^2)}+\|\nabla u\|_{S(L^2)} \leq 4\,c\,BM^{\frac{2p}{1-s_c}}.
\end{align}
Now, we define the wave operator
\begin{align}\label{eq:waveop}
u^+ = u_0 + i\int_0^{+\infty}e^{-it'\Delta}(|x|^{-(N-\gamma)}\ast|u|^p)|u|^{p-2}u(t')\,dt'.
\end{align}
By the same arguments as before, we have that
\begin{align*}
\|u^+\|_{L^2}&\leq c\|u_0\|_{L^2} + c_1\|u\|^{2(p-1)}_{S(\dot{H}^{s_c})}\|u\|_{S(L^2)},
\end{align*}
and
\begin{align*}
\|\nabla u^+\|_{L^2}&\leq c\|\nabla u_0\|_{L^2} + 2c_1\|u\|^{2(p-1)}_{S(\dot{H}^{s_c})}\|\nabla u\|_{S(L^2)}.
\end{align*}
Finally, by initial assumptions, we get
\begin{align*}
	\|u^+\|_{L^2}+\|\nabla u^+\|_{L^2}\leq cB + 2c_1M^{2(p-1)}\left(\|u\|_{S(L^2)}+\|\nabla u\|_{S(L^2)}\right).
\end{align*}
Using \eqref{bound}, we obtain that $\|u^+\|_{H^1}\leq \text{constant}$. This implies that $u^+ \in H^1(\R^N)$. From \eqref{eq:waveop} and the integral equation \eqref{duhamel}, we have
\begin{align*}
u(t)-e^{it\Delta}u^+ = -i\int_{t}^{+\infty}e^{i(t-t')\Delta}(|x|^{-(N-\gamma)}\ast |u|^p)|u|^{p-2}u(t')\,dt'.
\end{align*}
Again using the similar computation, we obtain
\begin{align*}
\|u(t)-e^{it\Delta}u^+\|_{L^2} &\leq c_1\|u\|^{2(p-1)}_{S(\dot{H}^{s_c};[t,+\infty))}\|u\|_{S(L^2;[t,+\infty))}
\end{align*}
and
\begin{align*}
	\|\nabla\left(u(t)-e^{it\Delta}u^+\right)\|_{L^2} &\leq c_1\|u\|^{2(p-1)}_{S(\dot{H}^{s_c};[t,+\infty))}\|\nabla u\|_{S(L^2;[t,+\infty))}.
\end{align*}
While obtaining \eqref{bound}, we have observed that the Strichartz norm on $[0,+\infty)$ for the above expression is bounded, therefore, the tail has to vanish as $t\rightarrow +\infty$, and thus, $\|u\|_{S(\dot{H}^{s_c};[t,+\infty))}\rightarrow 0$ as $t\rightarrow +\infty$. Hence,
$$
\lim_{t\rightarrow +\infty}\|u(t)-e^{it\Delta}u^+\|_{\dot{H}^1}=0.
$$
\end{proof}
We note that Theorem \ref{H1scatter} with initial data $u_0\in H^1(\R^N)$ also holds in the $L^2$-critical case ($s_c=0$ or $p=1+\frac{\gamma+2}{N}\geq 2$).
One can also obtain a similar result for the energy-critical case ($s_c=1$) but with a different selection of Strichartz pairs.

We now prove the long time perturbation result in the spirit of \cite{HR08}, which is one of the necessary ingredients in the subsequent analysis, specifically, in Theorem \ref{crit.elem}.

\begin{theorem}[Long time perturbation]\label{LTP}
For each $A\gg 1$, there exists $\epsilon_0=\epsilon_0(A)\ll1$ and $c=c(A)\gg 1$ such that the following holds. Let $u=u(x,t)\in H^1(\R^N)$ for all time $t$ and solve \eqref{gH}. Let $\widetilde{u}=\widetilde{u}(x,t)\in H^1(\R^N)$ for all $t$
and define $e$ to be
$$
e \defeq i\widetilde{u}_t+\Delta\widetilde{u}+(|x|^{-(N-\gamma)}\ast |\widetilde{u}|^p)|\widetilde{u}|^{p-2}\widetilde{u}.
$$
Suppose that 
\begin{align}\label{ltp1}
\|\widetilde{u}\|_{S(\dot{H}^{s_c})}\leq A,\quad \|e\|_{S^{\prime}(\dot{H}^{-s_c})}\leq \epsilon_0
\end{align} and
\begin{align}\label{epsiloncond}
\|e^{i(t-t_0)\Delta}(u(t_0)-\widetilde{u}_0(t_0))\|_{S(\dot{H}^{s_c})}\leq \epsilon_0.
\end{align}
Then
\begin{align}\label{concl}
\|u\|_{S(\dot{H}^{s_c})}\leq c=c(A)<+\infty.
\end{align}
\end{theorem}
\begin{proof}
Denote by $w$ the perturbation of $u$: $w=u-\widetilde{u}$. For $F(u)=(|x|^{-(N-\gamma)}\ast |u|^p)|u|^{p-2}u$ set $W(\widetilde{u},w)= F(u)-F(\widetilde{u}) = F(\widetilde{u}+w)-F(\widetilde{u})$. Then, $w$ solves
$$
iw_t+\Delta w + W(\widetilde{u},w) - e =0.
$$
Since $\|\widetilde{u}\|_{S(\dot{H}^s)}\leq A$, we can partition the interval $[t_0,+\infty)$ into $K=K(A)$ intervals $I_j=[t_j,t_{j+1}]$ such that for each $j$,
$
\|\widetilde{u}\|_{S(\dot{H}^{s_c};I_j)}\leq \delta.
$
Note that the number of intervals depends only on $A$, however, the intervals themselves depend upon $\widetilde{u}$.
The integral equation of $w$ at time $t_j$ is given by
\begin{align}\label{pertinteg}
w(t)=e^{i(t-t_j)\Delta}w(t_j) +i\int_{t_j}^{t}e^{i(t-t')\Delta}(W-e)(t')dt'.
\end{align}
Applying Kato estimate \eqref{Katostri} to \eqref{pertinteg} for each $I_j$, we obtain
\begin{align}\label{pertest}
\|w\|_{S(\dot{H}^{s_c};I_j)} &\leq \|e^{i(t-t_j)\Delta}w(t_j)\|_{S(\dot{H}^{s_c};I_j)} + c\|W(\widetilde{u},w)||_{S^{\prime}(\dot{H}^{-s_c};I_j)}+c\|e\|_{S^{\prime}(\dot{H}^{-s_c};I_j)}\notag\\
&\leq \|e^{i(t-t_j)\Delta}w(t_j)\|_{S(\dot{H}^{s_c};I_j)} + c\|W(\widetilde{u},w)\|_{S^{\prime}(\dot{H}^{-s_c};I_j)}+c\epsilon_0.
\end{align}
Next we estimate
\begin{align*}
&\|W(\widetilde{u},w)\|_{S^{\prime}(\dot{H}^{-s_c};I_j)}\lesssim \|F(\widetilde{u}+w)-F(\widetilde{u})\|_{L_{I_j}^{q_3^{\prime}}L_x^{r_1^{\prime}}}.
\end{align*}
Adding and subtracting $(|x|^{-(N-\gamma)}\ast |\widetilde{u}+w|^p)|\widetilde{u}|^{p-2}\widetilde{u}$, we obtain
\begin{align*}
\|W(\widetilde{u},w)\|_{S^{\prime}(\dot{H}^{-s_c};I_j)}\lesssim &\,\,\|\big(|x|^{-(N-\gamma)}\ast |\widetilde{u}+w|^p\big)\big(|\widetilde{u}+w|^{p-2}(\widetilde{u}+w)-|\widetilde{u}|^{p-2}\widetilde{u})\|_{L_{I_j}^{q_3^{\prime}}L_x^{r_1^{\prime}}}\\
&+\|\big(|x|^{-(N-\gamma)}\ast (|\widetilde{u}+w|^p-|\widetilde{u}|^p)\big)|\widetilde{u}|^{p-2}\widetilde{u}\|_{L_{I_j}^{q_3^{\prime}}L_x^{r_1^{\prime}}}.
\end{align*}
Using the calculations similar to \eqref{sd2}, we get
\begin{align*}
\|W(\widetilde{u},w)\|_{S^{\prime}(\dot{H}^{-s_c};I_j)}\leq&\,\, c_{N,\gamma}\|\widetilde{u}+w\|^p_{L_{I_j}^{q_2}L_x^{r_1}}\||\widetilde{u}+w|^{p-2}(\widetilde{u}+w)-|\widetilde{u}|^{p-2}\widetilde{u}\|_{L_{I_j}^{\frac{q_2}{p-1}}L_x^{\frac{r_1}{p-1}}}\\
&+c_{N,\gamma}\||\widetilde{u}+w|^p-|\widetilde{u}|^p\|_{L_{I_j}^{\frac{q_2}{p}}L_x^{\frac{r_1}{p}}}\|\widetilde{u}\|^{p-1}_{L_{I_j}^{q_2}L_x^{r_1}}.
\end{align*}
Using \eqref{sdc1} and \eqref{sdc5} yields
\begin{align}\label{pert1}
\|W(\widetilde{u},w)\|_{S^{\prime}(\dot{H}^{-s_c};I_j)}\leq&\,\, c_{N,\gamma}\|\widetilde{u}+w\|^p_{L_{I_j}^{q_2}L_x^{r_1}}\|w\|_{L_{I_j}^{q_2}L_x^{r_1}}\left(\|\widetilde{u}+w\|^{p-2}_{L_{I_j}^{q_2}L_x^{r_1}}+\|\widetilde{u}\|^{p-2}_{L_{I_j}^{q_2}L_x^{r_1}}\right)\\\label{ltp2}
&+c_{N,\gamma}\|w\|_{L_{I_j}^{q_2}L_x^{r_1}}\left(\|\widetilde{u}+w\|^{p-1}_{L_{I_j}^{q_2}L_x^{r_1}}+\|\widetilde{u}\|^{p-1}_{L_{I_j}^{q_2}L_x^{r_1}}\right)\|\widetilde{u}\|^{p-1}_{L_{I_j}^{q_2}L_x^{r_1}}.
\end{align}
We use the fact that $(a+b)^p\underset{p}{\lesssim} a^p+b^p$ for the $\|\widetilde{u}+w\|_{L_{I_j}^{q_2}L_x^{r_1}}$ terms in \eqref{pert1} and \eqref{ltp2} to obtain
\begin{align}
\|W(\widetilde{u},w)\|_{S^{\prime}(\dot{H}^{-s_c};I_j)}
\lesssim&\,\,c_{N,\gamma} \left(\|\widetilde{u}\|^p_{L_{I_j}^{q_2}L_x^{r_1}}+\|w\|^p_{L_{I_j}^{q_2}L_x^{r_1}}\right)\|w\|_{L_{I_j}^{q_2}L_x^{r_1}}\left(\|w\|^{p-2}_{L_{I_j}^{q_2}L_x^{r_1}}+\|\widetilde{u}\|^{p-2}_{L_{I_j}^{q_2}L_x^{r_1}}\right)\notag\\
&+c_{N,\gamma}\|w\|_{L_{I_j}^{q_2}L_x^{r_1}}\left(\|\widetilde{u}\|^{p-1}_{L_{I_j}^{q_2}L_x^{r_1}}+\|w\|^{p-1}_{L_{I_j}^{q_2}L_x^{r_1}}\right)\|\widetilde{u}\|^{p-1}_{L_{I_j}^{q_2}L_x^{r_1}}\notag.
\end{align}
Since $(q_2,r_1)$ is a $\dot{H}^{s_c}$ admissible pair by our assumption $\|\widetilde{u}\|_{S(\dot{H}^{s_c};I_j)}\leq \delta$, we obtain
\begin{align*}
\|W(\widetilde{u},w)\|_{S^{\prime}(\dot{H}^{-s_c};I_j)}
\lesssim&\,\,c_{N,\gamma} \left(\delta^p+\|w\|^p_{L_{I_j}^{q_2}L_x^{r_1}}\right)\|w\|_{L_{I_j}^{q_2}L_x^{r_1}}\left(\|w\|^{p-2}_{L_{I_j}^{q_2}L_x^{r_1}}+\delta^{p-2}\right)\\
&+c_{N,\gamma}\|w\|_{L_{I_j}^{q_2}L_x^{r_1}}\left(\delta^{p-1}+\|w\|^{p-1}_{L_{I_j}^{q_2}L_x^{r_1}}\right)\delta^{p-1}.
\end{align*}
Substituting the above estimate in \eqref{pertest},
\begin{align}
\|w\|_{S(\dot{H}^{s_c};I_j)}\lesssim&\,\, \|e^{i(t-t_j)\Delta}w(t_j)\|_{S(\dot{H}^{s_c};I_j)}+ c_1\delta^p\|w\|^{p-1}_{S(\dot{H}^{s_c};I_j)}+2c_1\delta^{2(p-1)}\|w\|_{S(\dot{H}^{s_c};I_j)}\notag\\
&+c_1\delta^{p-2}\|w\|^{p+1}_{S(\dot{H}^{s_c};I_j)}+c_1\delta^{p-1}\|w\|^p_{S(\dot{H}^{s_c};I_j)}+c_1\|w\|^{2p-1}_{S(\dot{H}^{s_c};I_j)}+c\epsilon_0.\notag
\end{align}
Let $\|w\|_{S(\dot{H}^{s_c};I_j)}\leq\widetilde{c}\delta$. If $c_1\widetilde{c}\delta^{2(p-1)}\leq \frac{1}{12}$, by
choosing
$\delta\leq\min\left(1,\delta_1\right)$, where $\delta_1=\sqrt[2(p-1)]{\frac{1}{12c_1\widetilde{c}}}$ together with \eqref{epsiloncond}, we can make sure that at time $t_j$, $\|e^{i(t-t_j)\Delta}w(t_j)\|_{S(\dot{H}^{s_c})}\leq\epsilon_1$, where $\epsilon_1$ depends on $\epsilon_0$, thus, we take
\begin{align}\label{pertepsilon}
\|e^{i(t-t_j)\Delta}w(t_j)\|_{S(\dot{H}^{s_c};I_j)}+c\epsilon_0 \leq \min\left(1,\frac{\delta_1}{2}\right).
\end{align}
Therefore, \eqref{pertepsilon} ensures that,
\begin{align}\label{pertfinal}
\|w\|_{S(\dot{H}^{s_c};I_j)}&\leq 2\|e^{i(t-t_j)\Delta}w(t_j)\|_{S(\dot{H}^{s_c};I_j)}+2c\epsilon_0.
\end{align}
Taking $t=t_{j+1}$ in \eqref{pertinteg}, applying $e^{i(t-t_{j+1})\Delta}$ to both sides and repeating the similar argument used for \eqref{pertfinal} (since the Duhamel integral is confined to $I_j=[t_j,t_{j+1}]$), we obtain
$$
\|e^{i(t-t_{j+1})\Delta}w(t_{j+1})\|_{S(\dot{H}^{s_c})}\leq 2\|e^{i(t-t_j)\Delta}w(t_j)\|_{S(\dot{H}^{s_c})}+2c\epsilon_0.
$$
Iterating down to $j=0$ and using \eqref{epsiloncond}, we get
\begin{align}\label{useinconcl}
\|e^{i(t-t_j)\Delta}w(t_j)\|_{S(\dot{H}^{s_c})}\leq 
 2^j\|e^{i(t-t_0)\Delta}w(t_0)\|_{S(\dot{H}^{s_c})}+(2^j-1)2c\epsilon_0\leq 2^{j+2}c\epsilon_0.
\end{align}
Now to satisfy the assumption \eqref{pertepsilon} for all intervals $I_j$, $0\leq j\leq n-1$, we require that
\begin{align}\label{smallepsi}
2^{n+2}c\epsilon_0 \leq \min\left(1,\frac{\delta_1}{2}\right).
\end{align}
This quantifies $\epsilon_0$ in terms of $n$ (number of time subintervals), which is determined by $A$ (given).
Hence, substituting $w=u-\tilde{u}$ on the left-hand side of \eqref{pertinteg} and applying Kato estimate \eqref{Katostri}, we obtain
\begin{align}
\|u\|_{S(\dot{H}^{s_c})}\leq \|e^{i(t-t_j)\Delta}w(t_j)\|_{S(\dot{H}^{s_c})} + c\|W(\widetilde{u},w)\|_{S^{\prime}(\dot{H}^{-s_c};I_j)}+c\epsilon_0+\|\widetilde{u}\|_{S(\dot{H}^{s_c})}\notag.
\end{align}
Thus, by repeating the argument used to deduce \eqref{pertfinal} and using \eqref{useinconcl} \eqref{ltp1} and \eqref{smallepsi}, we can conclude that
$$
\|u\|_{S(\dot{H}^{s_c})}\leq c(A).
$$	
\end{proof}

\section{Properties of Ground State}\label{propQ}

Now that we have local existence and that it was extended to get global existence of small data and $H^1$ scattering, we would like to study how large the initial data can be taken to continue enjoying the property of global existence and scattering. As in most focusing dispersive equations, there is typically a (sharp) threshold, which can be identified via the so-called ground state. However, one would need to know that such ground state solutions exist, whether they are unique (perhaps up to certain symmetries), and if ground state solutions can be obtained as minimizers of a certain functional (as it was originally done by Weinstein for the NLS in \cite{W83}). Minimization will identify the value of the threshold via some sharp constants of inequalities from which the functional is derived. We proceed along this route: we consider an appropriate interpolation inequality, set up a functional, minimize it and identify the sharp constant. One property that we do not know is if the minimizer is unique. Nevertheless, for the purpose of this work, it is sufficient to use the value of the sharp constant. 


We start with the Gagliardo-Nirenberg type inequality of convolution type. For brevity we denote 
$$
Z(u)=\int_{\R^N}\left(|x|^{-(N-\gamma)}\ast|u|^p\right)|u|^p\,dx.
$$
\begin{lemma}\label{GNineq} 
Suppose $p\geq 2$ and $0<\gamma<N$. Then
\begin{align}\label{eq:GN}
Z(u)\leq C_{GN}\|\nabla u\|_{L^2}^{Np-(N+\gamma)}\|u\|_{L^2}^{N+\gamma-(N-2)p}.
\end{align}
Moreover, the equality is attained on ground state solutions Q, which solve\footnote{In this equation we use the normalization for $Q$ as in Weinstein \cite{W83} when $\|Q\|_{L^2} = \|\nabla Q\|_{L^2} = Z(Q)$. Below we rescale $Q$ to have as elliptic equation with unit coefficients.}
\begin{align}\label{eq:finalground}
-\left(\frac{N+\gamma}{2p}-\frac{N-2}{2}\right)Q+\left(\frac{N}{2}-\frac{N+\gamma}{2p}\right)\Delta Q+\left(|x|^{-(N-\gamma)}\ast|Q|^p\right)|Q|^{p-2}Q=0,
\end{align}
and the sharp constant for \eqref{eq:GN} is attained at (any ground state) $Q$, which may be expressed as $C_{GN} = \|Q\|_{L^2({\mathbb{R}^N})}^{-2(p-1)}$.
\end{lemma}
	
\begin{remark}
We note that the ground state solutions $Q$ are positive, vanishing at infinity solutions, which are radial (modulo translations). These and other properties are investigated in \cite{MS13}, see also early works on the Hartree case in $\mathbb{R}^3$ in \cite{Lieb77}, \cite{Lieb83}, \cite{Lions80}, \cite{Lions84I}, \cite{Lions84II}. As we mentioned in the introduction the uniqueness is only known in the standard Hartree case $p=2, \gamma=2$ and $N \geq 3$ (also for $p=2+\epsilon$, $\gamma=2$ in dimension $N=3$).
\end{remark}	
	
\begin{proof}
We consider the Weinstein-type functional for functions $u \in H^1(\mathbb{R}^N)\setminus \{0\}$
\begin{equation}\label{E:W}
J(u)=\frac{\|u\|_{L^2}^{(N+\gamma)-(N-2)p} \, \|\nabla u\|_{L^2}^{Np-(N+\gamma)}}{Z(u)}.
\end{equation}
We mention that since we are interested in minimizing the value of $J$, replacing $u$ with its symmetric decreasing rearrangement will decrease both the $L^2$ norm and the $H^1$ norm (by Hardy-Littlewood and P\'olya-Szeg\"o inequalities). On the other hand, 
the symmetric decreasing rearrangement will increase the value of $Z(u)$ by Riesz's inequality, and thus, also will decrease the value of $J$. Hence, we can consider only radially symmetric functions $u = u(r)$, which are radially non-increasing (this is up to translations).

We proceed as in Weinstein \cite{W83} by defining 
$$
\eta = \inf \{ J(u): ~ u \in H^1_{rad}\setminus \{0\} \}.
$$ 
Since $J(u)>0$, there exists a minimizing sequence $\{u_k\}$ such that $\eta = 
\lim\limits_{k\to\infty}J(u_k)<\infty$. Note that if we set $u_{\lambda,\mu}=\mu u(\lambda x)$, then
$\|u_{\lambda,\mu}\|_{L^2}^2=\lambda^{-N}\mu^2 \|u\|_{L^2}^2 \quad \text{and} \quad 
\|\nabla u_{\lambda,\mu}\|_{L^2}^2 = \lambda^{2-N}\mu^2 \|\nabla u \|_{L^2}^2.$
By choosing $\lambda_k = {\|u_k \|_{L^2}}/{\| \nabla u_k \|_{L^2}}$ and 
$\mu_k = { \|u_k\|_{L^2}^{\frac{N}{2}-1}}/{ \| \nabla u_k \|_{L^2}^{\frac{N}{2}} } $, we obtain the sequence $\{ u_{\lambda_k,\mu_k} \}$, denoting it also by $\{u_k\}$, with $\|\nabla u_k \|_{L^2} = \|u_k \|_{L^2} = 1.$
Thus, $\{ u_k \}$ is a bounded non-negative sequence in $H^1$. Therefore, there exists $u^* \in H^1 \setminus \{0\}$, radial, nonnegative and non-increasing,  such that a subsequence of $\{u_k\}$ converges weakly in $H^1$ to $u^*$ with $\|u^*\|_{L^2} \leq 1$ and $\|\nabla u^*\|_{L^2}\leq 1$.

We next claim that $Z(u^*) = \lim\limits_{k \to \infty} Z(u_k)$, which is justified as follows: since $\{u_k\}$ is uniformly bounded in $\dot{H}^1_{rad}$, we have $u_k\rightarrow u^{*}$ in $L^{\frac{2Np}{N+\gamma}}$ (note that $2<\frac{2Np}{N+\gamma}<\frac{2N}{N-2}$). 
Now evaluating the difference, we obtain
\begin{align*}
Z(u_k)-Z(u^{*})=&\int_{\R^N}\left(|\cdot|^{-(N-\gamma)}\ast|u_k|^p\right)\big(|u_k|^p-|u^{*}|^p\big)\,dx\\
	&+\int_{\R^N}\left(|\cdot|^{-(N-\gamma)}\ast\big(|u_k|^p-|u^{*}|^p\big)\right)|u^{*}|^p\,dx\\
	\lesssim&\,\,\|u_k\|^p_{L^{\frac{2Np}{N+\gamma}}}\||u_k|^p-|u^{*}|^p\|_{L^{\frac{2N}{N+\gamma}}}+\||u_k|^p-|u^{*}|^p\|_{L^{\frac{2N}{N+\gamma}}}\|u^{*}\|^p_{\frac{2Np}{N+\gamma}} \xrightarrow[k \to \infty] {} 0.
\end{align*}
We can now conclude 
\begin{equation}\label{E:eta}
\eta \leq J(u^*) \leq \frac{1}{Z(u^*)} =\lim\limits_{k \to \infty}J(u_k)=\eta.
\end{equation}
This implies that 
$\|u^*\|_{L^2}= \|\nabla u^*\|_{L^2}=1$, and also $u_k\rightarrow u^*$ strongly in $H^1$. Therefore, $u^*$ is indeed a minimizer of $J$. 

Next we note that a minimizer $u^*$ satisfies the Euler - Lagrange equation
$$
\frac{d}{d\epsilon}\bigg|_{\epsilon=0}J(u^*+\epsilon h)=0\quad\text{for all}\quad h\in C^{\infty}_0,
$$
which, with $\|u^*\|_L^2= 1$ and $\|\nabla u^*\|_{L^2}= 1$, can be written as 
\begin{equation}\label{E:u}
-\left(\frac{N+\gamma}{2p}-\frac{N-2}{2}\right)u^*+\left(\frac{N}{2}-\frac{N+\gamma}{2p}\right)\Delta u^*+\eta\left(|x|^{-(N-\gamma)}\ast|u^*|^p\right)|u^*|^{p-1}=0.
\end{equation}
With equality in \eqref{E:eta}, we have $C_{GN} = \frac1{\eta} = Z(u^*)$.
Recall that $u^*$ is a positive, vanishing at infinity function, satisfying the above equation, thus, it is a ground state solution of \eqref{E:u} with the normalization $\|u^*\|_{L^2} = \|\nabla u^*\|_{L^2} = 1$. 
 
Setting $Q= {\eta}^{\frac{1}{2(p-1)}} u^*$, we obtain that $Q$ satisfies \eqref{eq:finalground}. With this rescaling, we have  $\|Q\|^2_{L^2} = \|\nabla Q\|^2_{L^2} = Z(Q) = \eta^{\frac1{2(p-1)}}$, and the sharp constant $C_{GN}=\frac1{\eta} \equiv 1/\|Q\|_{L^2}^{2(p-1)}$. Note that $\eta$ is the infimum, it uniquely determines $C_{GN}$ or such a quantity as $\|Q\|_{L^2}$.

One can also use another approach to find $C_{GN}$ and compute Pohozhaev identities for the equation \eqref{eq:finalground}: first, multiplying \eqref{eq:finalground} by $Q$ and integrating to obtain
\begin{equation}\label{poh1a}
\left(\frac{N+\gamma}{2p}-\frac{N-2}{2}\right)\|Q\|_{L^2}^2+\left(\frac{N}{2}-\frac{N+\gamma}{2p}\right)\|\nabla Q\|_{L^2}^2=Z(Q).
\end{equation}
Secondly, multiplying \eqref{eq:finalground} by $x\cdot\nabla Q$ and integrating, yields
\begin{equation}\label{poh1b}
\frac{N}{2}\left(\frac{N+\gamma}{2p}-\frac{N-2}{2}\right)\|Q\|_{L^2}^2+\frac{N-2}{2}\left(\frac{N}{2}-\frac{N+\gamma}{2p}\right)\|\nabla Q\|_{L^2}^2=\frac{N+\gamma}{2p}Z(Q),
\end{equation}
which also gives 
\begin{equation}\label{E:QZ}
Z(Q) = \|Q\|^2_{L^2}  = \| \nabla Q \|^2_{L^2}, 
\end{equation}
and substituting these values into \eqref{eq:GN}, we obtain
$\eta \equiv C_{GN, sharp} = \|Q\|_{L^2}^{-2(p-1)}$.
\end{proof}

\begin{remark}
It is convenient to rescale $Q$ as $Q(x) = \beta^{\frac{1}{2(p-1)}}\widetilde{Q}\left(\frac{\sqrt{\beta}}{\alpha} \, x\right)$, which gives the equation \eqref{nonlinell} (with all unit coefficients) for $Q$ instead of \eqref{eq:finalground} for $\tilde{Q}$. Here, $\alpha^2=\frac{N(p-1)-\gamma}{2p}$ and $\beta=\frac{N+\gamma-(N-2)p}{2p}$. From now on we only use $\widetilde{Q}$ (denoting it again by $Q$), solving \eqref{nonlinell} and the sharp constant  
\begin{equation}\label{poh1}
C_{GN} = \frac{2p}{N(p-1)-\gamma}\left(\frac{N+\gamma-(N-2)p}{N(p-1)-\gamma}\right)^{\frac{N(p-1)-\gamma}{2}-1} \, \frac1{\|Q\|_{L^2}^{2(p-1)}}.
\end{equation}
For future reference we also compute, 
\begin{equation}\label{poh2}
	M[Q]^{\theta}E[Q] = \frac{s_c(p-1)}{2s_c(p-1)+2}\|Q\|_{L^2}^{2\theta}\|\nabla Q\|_{L^2}^2
\end{equation}
and 
\begin{align}\label{poh3}
\|Q\|_{L^2}^{1-s_c}\|\nabla Q\|_{L^2}^{s_c}= \left(\frac{p\left(C_{GN}\right)^{-1}}{s_c(p-1)+1}\right)^{\frac{1}{2(P-1)}}.
\end{align}
	
\end{remark}

\section{Dichotomy: Global vs blow up solutions}\label{dich}
In this section we obtain the proof of Theorem \ref{main} part (1)(a) and part (2). We show that the condition in Theorem \ref{main} is sharp.  

\begin{theorem}\label{Dichotomy}
	Consider \eqref{gH} with $u_0 \in H^1(\R^N)$ and $0<s_c<1$. Assume that
	\begin{align}\label{eq:dichotomy1}
	\mathcal{M}\mathcal{E}[u_0]<1.
	\end{align}
	If
	\begin{align}\label{eq:dichotomy2}
	\mathcal{G}[u_0]<1,
	\end{align}
	then the solution $u(t)$ exists  for all $t\in\R$ (i.e., $I=\R$), and
	\begin{align}\label{eq:dichotomy3}
	\mathcal{G}[u(t)]<1.
	\end{align}
	If
	\begin{align}\label{eq:dichotomy4}
	\mathcal{G}[u_0]>1,
	\end{align}
	then for $t\in I = (-T,T)$
	\begin{align}\label{eq:dichotomy5}
	\mathcal{G}[u(t)]>1.
	\end{align}
	Moreover, if either $x|u_0|\in L^2(\R^N)$ or $u_0$ is radial, then $I$ is finite, and thus, the solution blows up in finite time.
\end{theorem}

The proof of this theorem goes along the established convexity arguments and the relevant Gagliardo-Nirenberg inequality with its sharp constant, we include it partially for completeness and also since the constants and coefficients are specific for the generalized Hartree case. The localized virial part deals with the convolution term, and thus, is new. 

\begin{proof} 
Using the energy conservation and \eqref{eq:GN}, we have
\begin{align}\label{eq:dichotomy6}
\mathcal{M}\mathcal{E}&[u]=\left(\frac{1}{2}\|\nabla u\|_{L^2(\R^N)}^2\|u_0\|_{L^2(\R^N)}^{2\theta}-\frac{1}{2p}Z(u)\|u_0\|_{L^2(\R^N)}^{2\theta	}\right)\frac{1}{M[Q]^{\theta}E[Q]}\notag \\
\geq &\left(\frac{1}{2}\|\nabla u\|_{L^2(\R^N)}^2\|u_0\|_{L^2(\R^N)}^{2\theta}-\frac{C_{GN}}{2p}\left(\|\nabla u\|_{L^2}\|u_0\|_{L^2}^{\theta}\right)^{2s_c(p-1)+2}\right)\frac{1}{M[Q]^{\theta}E[Q]}.
\end{align}
Using \eqref{poh2} and \eqref{poh3} and the value of $C_{GN}$, we get
\begin{align*}
\mathcal{M}\mathcal{E}[u] \geq\frac{s_c(p-1)+1}{s_c(p-1)}\mathcal{G}[u(t)]^2 - \frac{1}{s_c(p-1)}\left(\mathcal{G}[u(t)]\right)^{2s_c(p-1)+2}.
\end{align*}
Now the proof of \eqref{eq:dichotomy3} and \eqref{eq:dichotomy5} follows the same argument as in \cite{HR08}, \cite{DHR08} (see \cite{AKA2} for details).

Next if, $x u_0 \in L^2(\R^N)$, we write the virial identity as
\begin{align}\label{virial}
V_{tt} = 16(s_c(p-1)+1)E[u_0]-8s_c(p-1)\|\nabla u\|^2_{L^2(\R^N)}.
\end{align}
Multiplying the virial identity by $M[u_0]^{\theta}$ and proceeding as in \cite{HR08}, \cite{G}, we get
\begin{align*}
M[u_0]^{\theta}V_{tt} <-8s_c(p-1)\delta M[Q]^{\theta}\|\nabla Q\|^2_{L^2}<0,
\end{align*}
which by the convexity argument implies that the time interval $I$ must be finite, thus, blow-up occurs in finite time. 

If $u_0$ is radial, define $\phi\in C^{\infty}(\R)$, 
$$
\phi(|x|)=\begin{cases}
\frac{|x|^2}{2}&\quad 0\leq |x| \leq 2\\
1&\quad\quad r\geq 3
\end{cases}
$$
such that $\phi$ is smooth for $2<r<3$ and $\partial_r^2\phi(r)\leq 1$ for all $r\geq 0$. Now, for $R>0$ large, let $\phi_R=R^2\phi\left(\frac{|x|}{R}\right) $. Define the localized variance
$$
V_{loc}(t)=\int_{}^{}\phi_R(x)|u(x,t)|^2\,dx
$$
and compute the second derivative to obtain
\begin{align}\label{dichotomy9}
	\partial_t^2V_{loc}(t)
	=&\,4\int_{\R^N}\phi_R''|\nabla u|^2\,dx - \int_{\R^N}\Delta^2\phi_R|u|^2\,dx\\\label{dichotomy10}
	&-\frac{2(p-2)}{p}\int_{\R^N}\Delta\phi_R\frac{|u(x)|^p|u(y)|^p}{|x-y|^{N-\gamma}}\,dxdy\\\label{dichotomy11}
	&-\frac{4(N-\gamma)}{p}\int_{\R^N}\int_{\R^N}\nabla\phi_R\frac{(x-y)|u(x)|^p|u(y)|^p}{|x-y|^{N-\gamma+2}}\,dxdy.
\end{align}
We bound the two terms in \eqref{dichotomy9} using $\Delta\phi_R=N$  and $\Delta^2\phi_R=0$ for $|x|\leq 2R$ as follows
\begin{align}\label{dichotomy12}
	4\int\phi_R''|\nabla u|^2\,dx\leq&\, 4\int_{\R^N}|\nabla u|^2\,dx,\\\label{dichotomy13}
	- \int\Delta^2\phi_R|u|^2\,dx\leq&\, \frac{c}{R^2}\int_{2R<|x|<3R}^{}|u|^2\,dx.
	\end{align}
	Estimate \eqref{dichotomy10} using again the fact that $\Delta\phi_R(r)=N$ 
	\begin{align}\notag
	&-\frac{2(p-2)}{p}\int_{\R^N}\Delta\phi_R(|x|^{-(N-\gamma)}\ast|u|^p)|u|^p\,dx\\\notag
	\leq &-\frac{2N(p-2)}{p}\int_{|x|\leq 2R}(|x|^{-(N-\gamma)}\ast|u|^p)|u|^p\,dx+
	\frac{2c(p-2)}{p}\int_{2R<|x|<3R}(|x|^{-(N-\gamma)}\ast|u|^p)|u|^pdx\\\label{dichotomy14}
	\leq&-\frac{2N(p-2)}{p}\int_{\R^N}(|x|^{-(N-\gamma)}\ast|u|^p)|u|^pdx+c_1\int_{|x|>2R}(|x|^{-(N-\gamma)}\ast|u|^p)|u|^p\,dx.
	\end{align}
	Next we turn our attention to the term in \eqref{dichotomy11}, which can be rewritten as 
	\begin{align*}\notag
	\eqref{dichotomy11}=-&\frac{4(N-\gamma)}{p}\int_{\R^N}\int_{\R^N}\frac{R}{|x|}\phi'\left(\frac{|x|}{R}\right)\frac{x(x-y)|u(x)|^p|u(y)|^p}{|x-y|^{N-\gamma+2}}\,dxdy\\\notag
	=-&\frac{4(N-\gamma)}{p}\int_{\R^N}\int_{\R^N}\frac{x(x-y)|u(x)|^p|u(y)|^p}{|x-y|^{N-\gamma+2}}\,dxdy\\\notag
	&+\frac{4(N-\gamma)}{p}\int_{\R^N}\int_{\R^N}\left(1-\frac{R}{|x|}\phi'\left(\frac{|x|}{R}\right)\right)\frac{x(x-y)|u(x)|^p|u(y)|^p}{|x-y|^{N-\gamma+2}}\,dxdy\\
	=-&\frac{2(N-\gamma)}{p}\int_{\R^N}\int_{\R^N}\frac{|u(x)|^p|u(y)|^p}{|x-y|^{N-\gamma}}\,dxdy\\
	&+\frac{4(N-\gamma)}{p}\int_{\R^N}\int_{\R^N}\left(1-\frac{R}{|x|}\phi'\left(\frac{|x|}{R}\right)\right)\frac{x(x-y)|u(x)|^p|u(y)|^p}{|x-y|^{N-\gamma+2}}\,dxdy
	\end{align*}
Combining the above expression with \eqref{dichotomy12}, \eqref{dichotomy13} and \eqref{dichotomy14}, we write
\begin{align*}
	\partial_t^2V_{loc}(t)\leq&\,\, 4\int_{\R^N}^{}|\nabla u|^2 + \frac{c}{R^2}\int_{2R<|x|<3R}^{}|u|^2+c_1\int_{|x|>2R}(|x|^{-(N-\gamma)}\ast|u|^p)|u|^pdx\\
	& -\left(\frac{2N(p-2)}{p}+\frac{2(N-\gamma)}{p}\right)\int_{\R^N}(|x|^{-(N-\gamma)}\ast|u|^p)|u|^pdx\\ 
	&+\frac{4(N-\gamma)}{p}\int_{\R^N}\int_{\R^N}\left(1-\frac{R}{|x|}\phi'\left(\frac{|x|}{R}\right)\right)\frac{x(x-y)|u(x)|^p|u(y)|^p}{|x-y|^{N-\gamma+2}}dxdy.
	\end{align*}
	Writing the above inequality in terms of energy and gradient, we get
	\begin{align}\label{dichotomy15}
	\partial_t^2V_{loc}(t)	\leq&\,\, 4(N(p-1)-\gamma)E[u_0]- (2(N(p-1)-\gamma)-4)\int_{\R^N}^{}|\nabla u|^2\,dx\\\label{dichotomy16}
	&+ \frac{c}{R^2}\int_{2R<|x|<3R}^{}|u|^2\,dx +c_1\int_{|x|>2R}(|x|^{-(N-\gamma)}\ast|u|^p)|u|^p\,dx\\\label{dichotomy17}
	&+\frac{4(N-\gamma)}{p}\int_{\R^N}\int_{\R^N}\left(1-\frac{R}{|x|}\phi'\left(\frac{|x|}{R}\right)\right)\frac{x(x-y)|u(x)|^p|u(y)|^p}{|x-y|^{N-\gamma+2}}\,dxdy.
    \end{align}
The second term in the expression \eqref{dichotomy16} can be estimated as
\begin{align}\notag
	\int_{|x|>2R}^{}\left(|x|^{-(N-\gamma)}\ast|u|^p\right)|u|^p\,dx&\lesssim  \||x|^{-(N-\gamma)}\ast|u|^p\|_{L_{|x|>2R}^{\frac{2N}{N-\gamma}}}\|u\|^p_{L_{|x|>2R}^{\frac{2Np}{N+\gamma}}}\quad \text{(H\"older's)}\\\notag
	&\lesssim \|u\|^{2p}_{L_{|x|>2R}^{\frac{2Np}{N+\gamma}}}\quad \text{(Lemma \ref{HLS})}\\\label{dichotomy18}
	&\lesssim \frac{1}{R^{\frac{(N-1)(N(p-1)-\gamma)}{N}}}\|\nabla u\|^{\frac{N(p-1)-\gamma}{N}}_{L^2}\|u\|^{\frac{N(p+1)+\gamma}{N}}_{L^2}\,\,\text{(radial Sobolev)}.
\end{align}
We rewrite the integral in \eqref{dichotomy17}, using symmetry, as follows
\begin{align}\label{dichotomy19}
	\frac{1}{2}\int_{\R^N}\int_{\R^N}\left(\left(1-\frac{R}{|x|}\phi'\left(\frac{|x|}{R}\right)\right)x-\left(1-\frac{R}{|y|}\phi'\left(\frac{|y|}{R}\right)\right)y\right)\frac{(x-y)|u(x)|^p|u(y)|^p}{|x-y|^{N-\gamma+2}}dxdy,
\end{align}
which can be broken down into the following regions (observe that the integral vanishes in the region $|x|\leq 2R$);
\begin{itemize}
	\item Region I: $|x|\approx |y|.$ In this region we have
	$$
	|x|>2R,\,\,\,|y|> 2R.
	$$
	Observe that
	$$
	\left|\left(1-\frac{R}{|x|}\phi'\left(\frac{|x|}{R}\right)\right)x-\left(1-\frac{R}{|y|}\phi'\left(\frac{|y|}{R}\right)\right)y\right|\lesssim |x-y|.
	$$
	We estimate \eqref{dichotomy19} in a similar fashion as \eqref{dichotomy18} to obtain
	\begin{align}\label{dichotomy20}
		\int\int\frac{\chi_{|y|>2R}|u(y)|^p}{|x-y|^{N-\gamma}}\chi_{|x|>2R}|u(x)|^p\,dxdy\lesssim\frac{1}{R^{\frac{(N-1)(N(p-1)-\gamma)}{N}}}||\nabla u||^{\frac{N(p-1)-\gamma}{N}}_{L^2}||u||^{\frac{N(p+1)+\gamma}{N}}_{L^2}.
	\end{align} 	
	\item Region II: $\max\{|x|,|y|\}\gg\min\{|x|,|y|\}$ and $\max\{|x|,|y|\}>2R.$ We consider two cases:
	\begin{itemize}
		\item Case (a):	$|x|\ll|y|\approx |x-y|,\quad |y|>2R$ and $|x|<2R$. In this case \eqref{dichotomy19} becomes
		$$
		\int\int\frac{1}{|x-y|^{N-\gamma}}\,\chi_{|y|>2R}|u(y)|^p\,|u(x)|^p\,dxdy,
		$$
		since using the triangle inequality and  the definition of $\phi$, we have
		\begin{align*}
			\Big|&\left(1-\frac{R}{|x|}\phi'\left(\frac{|x|}{R}\right)\right)x-\left(1-\frac{R}{|y|}\phi'\left(\frac{|y|}{R}\right)\right)y\Big|\\
			&\leq |x|\left(1-\frac{R}{|x|}\phi'\left(\frac{|x|}{R}\right)\right)+|y|\left(1-\frac{R}{|y|}\phi'\left(\frac{|y|}{R}\right)\right)\\
			&\lesssim |y|\approx|x-y|
		\end{align*}
		since $1-\frac{R}{|x|}\phi'\left(\frac{|x|}{R}\right)<1$ and $1-\frac{R}{|y|}\phi'\left(\frac{|y|}{R}\right)>\frac{1}{2}$. Again using H\"older's inequality, Lemma \ref{HLS} and  radial Sobolev as in \eqref{dichotomy18}, we bound the above integral by
		\begin{align}\label{dichotomy21}
		\frac{1}{R^{\frac{(N-1)(N(p-1)-\gamma)}{N}}}||\nabla u||^{\frac{N(p-1)-\gamma}{N}}_{L^2}||u||^{\frac{N(p+1)+\gamma}{N}}_{L^2}.
		\end{align}
		\item Case (b): $|y|\ll|x|\approx |x-y|,\quad |x|>2R$ and $|y|<2R$. This case is symmetric and treated with a similar argument as in Case (a).
	\end{itemize}
\end{itemize}
Combining \eqref{dichotomy18}, \eqref{dichotomy20} and \eqref{dichotomy21}, we get
\begin{align*}
	\partial_t^2V_{loc}(t)\leq&\,8(s_c(p-1)+1)E[u_0]- 4s_c(p-1)\int_{\R^N}^{}|\nabla u|^2+ \frac{c}{R^2}\int_{2R<|x|<3R}^{}|u|^2\\
	&+\frac{\tilde{c}}{R^{\frac{(N-1)(N(p-1)-\gamma)}{N}}}\|\nabla u\|^{\frac{N(p-1)-\gamma}{N}}_{L^2}\|u\|^{\frac{N(p+1)+\gamma}{N}}_{L^2}.
\end{align*}
Using Young's inequality to separate the $L^2$ norm and gradient term in
the last term, we obtain
\begin{align*}
	\partial_t^2V_{loc}(t)\leq&\,8(s_c(p-1)+1)E[u_0]- 4s_c(p-1)\int_{\R^N}^{}|\nabla u|^2+ \frac{c}{R^2}\int_{2R<|x|<3R}^{}|u|^2\\
	&+\epsilon\,\|\nabla u\|_{L^2}^2+\frac{c(\epsilon,N)}{R^{\frac{2(N-1)(N(p-1)-\gamma)}{N(3-p)+\gamma}}}\|u\|^\frac{2(N(p+1)+\gamma)}{N(3-p)+\gamma}_{L^2}.
\end{align*}
 Multiplying the above expression by $M[u_0]^{\theta}$ and using the similar argument as in the case of finite variance, we get
\begin{align*}
	M[u_0]^{\theta}\partial_t^2V_{loc}(t)\leq&\,8(s_c(p-1)+1)M[u_0]^{\theta}E[u_0]- (4s_c(p-1)-\epsilon)\|u\|^{2\theta}_{L^2}\|\nabla u\|^2_{L^2}\\
	&+ \frac{c}{R^2}\|u\|^{2+2\theta}_{L^2}+\frac{c(\epsilon,N)}{R^{\frac{2(N-1)(N(p-1)-\gamma)}{N(3-p)+\gamma}}}\|u\|^{\frac{2(N(p+1)+\gamma)}{N(3-p)+\gamma}+2\theta}_{L^2},
	 \end{align*}
	 which can be re-written as
	 \begin{align*}
	M[u_0]^{\theta}\partial_t^2V_{loc}(t)\leq\,4s_c(p-1)(1-\delta_1)&M[Q]^{\theta}\|\nabla Q\|_{L^2}^2- (4s_c(p-1)-\epsilon)(1+\delta_2)	M[Q]^{\theta}\|\nabla Q\|_{L^2}^2 \\
	&+ \frac{c}{R^2}\|u\|^{2+2\theta}_{L^2}+\frac{c(\epsilon,N)}{R^{\frac{2(N-1)(N(p-1)-\gamma)}{N(3-p)+\gamma}}}\|u\|^\frac{2(N(p+1)+\gamma)}{N(3-p)+\gamma}_{L^2}
	.
\end{align*}
Choose $$
0<\epsilon<\frac{4s_c(p-1)(\delta_1+\delta_2)}{1+\delta_2}
$$
and $R=R(\epsilon, \delta_1,N,p,\gamma, M[u_0])$ large enough to obtain
\begin{align*}
	M[u_0]^{\theta}\partial_t^2V_{loc}(t)\leq-c(\epsilon,N,p,\gamma),
\end{align*}
where $c(\epsilon,N,p,\gamma)>0$, implying that the maximum interval of existence $I$ is finite.
\end{proof}
The following lemmas provide some additional estimates that will be needed for the
compactness and rigidity results in Section \ref{comp}-\ref{rigid}. We state the Lemmas without proof as the arguments are similar to the ones presented in \cite{HR08}, \cite{G}. For more details, refer to \cite{AKA2}.   
\begin{lemma}[Comparison of Energy and Gradient]\label{EGcomp}
	Let $u_0\in H^1(\R^N)$ satisfy \eqref{eq:dichotomy1} and \eqref{eq:dichotomy2}. Then
	\begin{align}\label{E-G_comp}
	\frac{s_c(p-1)}{2s_c(p-1)+2}\|\nabla u\|_{L^2(\R^N)}^2 \leq E[u] \leq \frac{1}{2}\|\nabla u\|_{L^2(\R^N)}^2.
	\end{align}
\end{lemma}

\begin{lemma}[Lower bound on the convexity of variance]\label{bndconvexvar}
	Let $u_0 \in H^1(\R^N)$ satisfy \eqref{eq:dichotomy1} and \eqref{eq:dichotomy2}. Then for all $t\in\R$
	\begin{align}\label{lowbndvar}
	16E[u]\left(1-(\mathcal{ME}[u])^{s_c(p-1)}\right)\leq 8 \left(\|\nabla u\|^2_{L^2} -  \frac{s_c(p-1)+1}{p}\,Z(u)\right).
	\end{align}
\end{lemma}

\begin{lemma}[Existence of wave operators]\label{waveop}
	Suppose $\psi^+ \in H^1(\R^N)$ and
	\begin{align}\label{op1}
	\|\psi^+\|_{L^2}^{2\theta}\|\nabla\psi^+\|_{L^2}^{2}\leq \mu^2\left(\frac{2s_c(p-1)+2}{s_c(p-1)}\right) M[Q]^{\theta}E[Q]
	\end{align}
	 for some $0<\mu\leq \left(\frac{s_c(p-1)}{2s_c(p-1)+2}\right)^{\frac{1}{2}}<1$. Then there exists $v_0 \in H^1(\R^N)$	such that $v(t)$, solving \eqref{gH} with initial data $v_0$, is global in $H^1(\R^N)$ with
	$$
	\|v_0\|_{L^2}^{\theta}\|\nabla v(t)\|_{L^2}\leq \|Q\|_{L^2}^{\theta}\|\nabla Q\|_{L^2},\quad M[v]=\|\psi^+\|_{L^2}^2,\quad E[v]=\frac{1}{2}\|\nabla\psi^+\|_{L^2}^2
	$$
	and
	$$
	\qquad \|v(t)-e^{it\Delta}\psi^+\|_{H^1}\rightarrow 0\qquad\qquad\text{as}\qquad\qquad t\rightarrow\infty .
	$$
	Moreover, if $\|e^{it\Delta}\psi^+\|_{S(\dot{H}^{s_c})}\leq\delta$, then
	$$
	\|v_0\|_{\dot{H}^{s_c}}\leq 2\|\psi^+\|_{\dot{H}^{s_c}}\qquad\text{and}\qquad \|v\|_{S(\dot{H}^{s_c})}\leq 2\|e^{it\Delta}\psi^+\|_{S(\dot{H}^{s_c})}.
	$$
\end{lemma}
\begin{proof} We consider the integral equation
	\begin{align}\label{Duhform}
	v(t) = e^{it\Delta}\psi^+ - i\int_{t}^{\infty} e^{i(t-t')\Delta}\left(\left(|x|^{-(N-\gamma)}\ast |u|^p\right)|u|^{p-2}u\right)(t')\, dt',
	\end{align}
	which we would like to solve for all $t$. Note that for $T > 0$ by Theorem \ref{smalldata} (small data theory) there exists $\delta > 0$ such that
	$\|e^{it\Delta}\psi^+\|_{S(\dot{H}^{s_c};[T,\infty))}\leq \delta$. Thus, we solve the equation \eqref{Duhform} in $H^1$ for $t \geq T$ with $T$ large. Estimating \eqref{Duhform} in $S(L^2)$ for $t\geq T$, we obtain
	\begin{align*}
	\|\nabla v\|_{S(L^2;[T,\infty))}&\lesssim \|e^{it\Delta}\nabla \psi^+\|_{S(L^2;[T,\infty))} + \|\nabla[(|\cdot|^{-(N-\gamma)}\ast|v|^p)|v|^{p-2}v]\|_{S'(L^2;[T,\infty))}\\
	&\lesssim \|\psi^+\|_{\dot{H}^1} + \| v\|^{2(p-1)}_{S(\dot{H}^{s_c};[T,\infty))}\|\nabla v\|_{S(L^2;[T,\infty))}.
	\end{align*}
	Taking $T$ sufficiently large so that $\|v\|^{2(p-1)}_{S(\dot{H}^{s_c};[T,\infty))}\leq \frac{1}{2}$, we get $\|\nabla v \|_{S(L^2;[T,\infty))}\lesssim 2\|\psi^+\|_{\dot{H}^1}.$
	Using the above inequality, we obtain in a similar fashion,
	\begin{align*}
	\|\nabla\left(v-e^{it\Delta}\psi^+\right)\|_{S(L^2;[T,\infty))}&\leq \|\nabla[(|\cdot|^{-(N-\gamma)}\ast|v|^p)|v|^{p-2}v]\|_{S'(L^2;[T,\infty))}\\
	&\leq \|v\|^{2(p-1)}_{S(\dot{H}^{s_c};[T,\infty))}\|\nabla v\|_{S(L^2;[T,\infty))}\\
	&\leq c\|\psi^+\|_{\dot{H}^1},
	\end{align*}
	hence, $\|\nabla\left(v-e^{it\Delta}\psi^+\right)\|_{S(L^2;[T,\infty))}\rightarrow 0$ as $T\rightarrow \infty$. Since, by Theorem \ref{H1scatter} ($H^1$ scattering), we have $v-e^{it\Delta}\psi^+\rightarrow 0$ in $H^1$ as $t\rightarrow\infty$ and the decay estimate together with the embedding $H^1(\R^N)\hookrightarrow L^q(\R^N)$ with $q\leq \frac{2N}{N-2}$ for $N\geq3$, $q<\infty$ for $N=2$ and $q\leq \infty$ for $N=1$ implies
	$$
	Z(e^{it\Delta}\psi^+)\lesssim \|e^{it\Delta}\psi^+\|_{L^{\frac{2Np}{N+\gamma}}}\leq |t|^{-\frac{Np-N-\gamma}{2p}}\|\psi^+\|_{H^1},
	$$
	thus, $Z\left(e^{it\Delta}\psi^+\right)\rightarrow 0$ in $L^{\frac{2Np}{N+\gamma}}$ as $t\rightarrow\infty$.
	Since $\lim_{t\rightarrow +\infty}\|v(t)\|_{H^1}=\|\nabla\psi^+\|_{H^1}$, we have
	\begin{align*}
	E[v]&=\frac{1}{2}\|\nabla v\|_{L^2}^2-\frac{1}{2p}\int (|x|^{-(N-\gamma)}\ast|v|^p)|v|^{p}\,dx\\
	&=\lim\limits_{t\rightarrow\infty}\left(\frac{1}{2}\|\nabla e^{it\Delta}\psi^+\|_{L^2}^2-\frac{1}{2p}\int \left(|x|^{-(N-\gamma)}\ast|e^{it\Delta}\psi^+|^p\right)|e^{it\Delta}\psi^+|^{p}\right)\\
	&=\frac{1}{2}\|\nabla\psi^+\|_{L^2}^2
	\end{align*}
	and $	M[v]=\lim\limits_{t\rightarrow\infty}\|e^{it\Delta}\psi^+\|_{L^2}^2=\|\psi^+\|_{L^2}^2.$	Note that by \eqref{op1} we now have
	\begin{align*}
	M[v]^{\theta}E[v]&= \frac{1}{2}\|\psi^+\|_{L^2}^{2\theta}\|\nabla\psi^+\|_{L^2}^{2}\leq \mu^2\left(\frac{2s_c(p-1)+2}{s_c(p-1)}\right)M[Q]^{\theta}E[Q]
	\end{align*}
	and by our choice of $\mu$ we conclude that $M[v]^{\theta}E[v] <M[Q]^{\theta}E[Q].$ Moreover,
	\begin{align*}
	\lim\limits_{t\rightarrow\infty}\|v(t)\|_{L^2}^{2\theta}\|\nabla v(t)\|_{L^2}^{2}&=\|\psi^+\|_{L^2}^{2\theta}\|\nabla\psi^+\|_{L^2}^{2}\\
	&\leq \mu^2\left(\frac{2s_c(p-1)+2}{s_c(p-1)}\right) M[Q]^{\theta}E[Q]\\
	&=\mu^2\|Q\|_{L^2}^{2\theta}\|\nabla Q\|_{L^2}^{2},
	\end{align*}
	where the inequality is due to \eqref{op1} and last equality is from \eqref{poh3}. We can take $T > 0$ sufficiently large so that $\|v(T)\|_{L^2}^{\theta}\|\nabla v(T)\|_{L^2}<\mu\|Q\|_{L^2}^{\theta}\|\nabla Q\|_{L^2}$. And, since $\mu<1$, by Theorem \ref{Dichotomy} (global existence of solutions), we evolve $v(T)$
	from time $T$ back to time $0$ and obtain $v$ with initial data $v_0 \in H^1$  for all time $t\in[0,\infty)$ with the desired properties.
\end{proof}

\section{Compactness }\label{comp}
\subsection{Blueprint}

To characterize the behavior of global solutions to \eqref{gH}, we must show that if $\mathcal{M}\mathcal{E}[u]<1$ and $\mathcal{G}[u_0]<1$, then the global-in-time $\dot{H}^{s_c}$ Strichartz norm is finite, i.e., $\|u\|_{S(\dot{H}^{s_c})}<\infty$. This would imply that $\|\nabla u (t)\|_{L^2}\leq C$ and thus, $I=(-\infty,\infty)$. For completeness we provide the blueprint below, which is based on the works of Holmer-Roudenko \cite{HR08}, Duyckaerts-Holmer-Roudenko \cite{DHR08} for the 3d cubic nonlinear Schr\"odinger equation and Kenig-Merle \cite{KM06} for the energy-critical NLS equation.\\

\noindent
\underline{\textit{First Stage:}} \textit{Small data theory}

Using Lemma \ref{EGcomp}, we have
$$
\|u_0\|_{\dot{H}^{s_c}}^{2(p-1)}\leq \left(\|u_0\|^{\theta}_{L^2}\|\nabla u_0\|_{L^2}\right)^{2s_c(p-1)}< \left(\frac{2p}{p-1}\right)^{s_c(p-1)}\left(M[u]^{\theta}E[u]\right)^{s_c(p-1)}.
$$
If $\mathcal{G}[u_0]<1 $ and $\mathcal{ME}[u]<\left(\frac{p-1}{2p}\right)\frac{\delta_{sd}^{2/s_c}}{M[Q]^{\theta}E[Q]}$, then from the above inequality, we obtain $\|u_0\|_{\dot{H}^{s_c}}\leq\delta_{sd}$, which  by Strichartz estimates gives $\|e^{it\Delta}u_0\|_{S(\dot{H}^{s_c})}\leq c\,\delta_{sd}$. Therefore, Theorem  \ref{smalldata} (small data theory) implies that there exists a $\delta>0$ such that if $\mathcal{G}[u_0]<1$ and $\mathcal{ME}[u]<\delta$, then $T^*=+\infty$ and $\|u_0\|_{\dot{H}^{s_c}}<\infty$. This gives us the basis for induction.\\

\noindent
\underline{\textit{Second stage:}} \textit{Construction of critical solution (via induction on scattering threshold) }

Let $(\mathcal{ME})_c$ be the supremum over all $\delta>0$ for which the following is true:

 ``If $u_0\in H^1(\R^N)$ with $\mathcal{G}[u_0]<1$ and $\mathcal{ME}[u]<\delta$ such that $\delta = \delta(M[Q]^{1-s}E[Q]^s)$, then $T^*=+\infty$ and $\|u_0\|_{\dot{H}^{s_c}}<\infty$.''

 If $(\mathcal{ME})_c=1$, then we are done, since $Q$ (soliton) does not scatter. So, we assume that $(\mathcal{ME})_c<1$. This implies (by definition of $(\mathcal{ME})_c$) that there exists a sequence of solutions $\{u_n\}$ to \eqref{gH} with initial data $u_{n,0}\in H^1(\R^N)$ that approach the threshold $(\mathcal{ME})_c$ from above but do not scatter, i.e., there exists a sequence $u_{n,0}\in H^1(\R^N)$ such that
 $$
 \mathcal{G}[u_{n,0}]<1\,\,\,\text{and}\,\,\,\mathcal{ME}[u_{n,0}]\searrow (\mathcal{ME})_c \,\,\,\text{as}\,\,n\rightarrow\infty
 $$
 for which $\|u_n\|_{S(\dot{H}^{s_c})}=+\infty$. Using the profile decomposition (Theorem \ref{linprodecomp}) on the sequence of initial data $\{u_{n,0}\}$, we prove the existence of an $H^1$ solution $u_c$ to \eqref{gH} with initial data $u_{c,0}$ such that $\mathcal{G}[u_{c,0}]<1$ and $\mathcal{ME}[u_c]=(\mathcal{ME})_c$ (i.e., it lies exactly at the  threshold $(\mathcal{ME})_c$), but $u_c$ does not scatter (Theorem \ref{crit.elem}).\\


\noindent
 \underline{\textit{Third stage:}} \textit{Localization of critical solution (setting the premise for rigidity theorem)}
 
 The critical solution $u_c(t)$, constructed in the second stage, will have the property that $K=\{u_c(t)\,\,|\,\,t\in[0,\infty)\} $ is precompact in $H^1(\R^N)$ (Proposition \ref{precompflow_crit}). This will allow us to show that for a given $\epsilon>0$, there is an $R>0$ such that
 $$
 \int_{|x+x(t)|>R}^{}|\nabla u(x,t_n)|^2\,dx\leq \epsilon
 $$
 uniformly in $t$ (Lemma \ref{unilocalization}). Together with the zero momentum hypothesis (Lemma \ref{zeromoment}), this controls the growth of path $x(t)$ (Lemma \ref{pathcontrol}).\\

 \noindent
 \underline{\textit{Final Stage:}} \textit{Rigidity theorem (Theorem \ref{rigidity})}

 Appealing to this uniform localization and control of $x(t)$, we invoke the Rigidity theorem, which leads to contradiction that such compact solution in $H^1$ exists unless it is a trivial solution, which scatters. Therefore, the assumption $(\mathcal{ME})_c<1$ is not valid, concluding the proof. 
 
We now fill in the necessary details.
\subsection{Profile decomposition}
\begin{theorem}[Linear Profile decomposition]\label{linprodecomp}
Let $\phi_n(x)$ be a uniformly bounded sequence in $H^1(\R^N)$. Then for each $M \in \N$ there exists a subsequence of
	$\phi_n(x)$ (also denoted $\phi_n(x)$), such that, for each $1 \leq j \leq M$, \begin{enumerate}
		\item there exist, fixed in $n$, a profile  $\psi^j \in H^1(\R^N)$,
		\item there exists a sequence (in $n$) $t^j_n$ of time shifts,
		\item there exists a sequence (in $n$) $x^j_n$ of space shifts,
		\item there exists a sequence (in $n$) $W^M_n(x)$
		of remainders in $H^1(\R^N)$, such that
		\begin{equation}\label{NLPD_crit}
			\phi_n(x) = \sum_{j=1}^{M}e^{-it^j_n\Delta}\psi^j(x-x^j_n)+W^M_n(x)
		\end{equation}
	\end{enumerate}
	with the properties:
	\begin{itemize}
		\item Pairwise divergence for the time and space sequences. For $1 \leq k \neq j \leq M$,
		\begin{align}\label{pairdivg}
			\lim\limits_{n\rightarrow\infty} |t_n^j-t_n^k|+|x_n^j-x_n^k|=+\infty.
		\end{align}
		\item Asymptotic smallness for the remainder sequence
		\begin{align}\label{smallrem}
			\lim\limits_{M\rightarrow\infty}\left(\lim\limits_{n\rightarrow\infty}\|e^{it\Delta}W_n^M\|_{S(\dot{H}^{s_c})}\right)=0.
		\end{align}
		\item Asymptotic Pythagorean expansion. For fixed $M\in\N$ and for any $0\leq s \leq 1$, we have
		\begin{align}\label{pythagexp}
			\|\phi_n\|_{\dot{H}^s}=\sum_{j=1}^{M}\|\psi^j\|^2_{\dot{H^s}}+\|W_n^M\|^2_{\dot{H^s}}+o_n(1).
		\end{align}
	\end{itemize}
\end{theorem}
\begin{proof}
Refer \cite{AKA2}, \cite{G}, \cite{HR08} for details.
\end{proof}
\begin{proposition}
	[Energy Pythagorean expansion]\label{energypythadecomp}
	Under the assumptions of \,Theorem \ref{linprodecomp}, we have
	\begin{align}\label{energypythaexp}
		E\left[\phi_n\right]=\sum_{j=1}^{M}E\left[e^{-it_n^j\Delta}\psi^j\right]+E\left[W_n^M\right]+o_n(1).
	\end{align}
\end{proposition}
\begin{proof}
	By the definition of energy, $E[u]$, and \eqref{pythagexp} for $s=1$, it is sufficient to establish for all $M\geq 1$,
\begin{align}\label{PEdecomp}
	Z\left(\phi_n\right)=\sum_{j=1}^{M}Z\left(e^{-it_n^j\Delta}\psi^j\right)+Z\left(W_n^M\right)+o_n(1),
\end{align}
where $Z(u)=\int_{\R^N}^{}\left(|x|^{-(N-\gamma)}\ast |u|^p\right)|u|^p$. \\

\noindent
\textit{Step 1.} \emph{Pythagorean expansion of a sum of orthogonal profiles.} We show that for $M\geq 1$ fixed, the orthogonality condition \eqref{pairdivg} implies
\begin{align}\label{pairdivgexp}
	Z\left(\sum_{j=1}^{M}e^{-it_n^j\Delta}\psi^j(\cdot - x_n^j)\right)=\sum_{j=1}^{M}Z\left(e^{-it_n^j\Delta}\psi^j\right)+o_n(1).
\end{align}
By rearranging and reindexing, we can find $M_0\leq M$ such that
\begin{itemize}
	\item For $1\leq j\leq M_0$, we have that $t_n^j$ is bounded in $n$.
	\item For $M_0+1\leq j\leq M$, we have that $\left|t_n^j\right|\rightarrow\infty$ as $n\rightarrow\infty$.
\end{itemize}
Passing to a subsequence, we may assume that for each $1\leq j\leq M_0$, $t_n^j$ converges (in $n$), and by adjusting the profiles $\psi^j$'s we can take $t_n^j=0$. Note that either for $1\leq k\leq M_0$ we have $t_n^k\rightarrow 0$ or for $M_0+1\leq k\leq M$ we have $|t_n^k|\rightarrow\infty$ as $n\rightarrow\infty$. So if $t_n^k\rightarrow 0$, then from \eqref{pythagexp} we have $|x_n^j-x_n^k|\rightarrow\infty$ as $n\rightarrow\infty$, which implies
\begin{align}\label{pairdivgexp1}
	Z\left(\sum_{j=1}^{M_0}\psi^j(\cdot - x_n^j)\right)=\sum_{j=1}^{M_0}Z\left(\psi^j\right)+o_n(1).
\end{align}
Now if $|t_n^k|\rightarrow\infty$ as $n\rightarrow\infty$, for a function $\tilde{\psi}\in\dot{H}^{\frac{N(p-1)-\gamma}{2p}}\cap L^{\frac{p+1}{p}}$, by Hardy-Littlewood-Sobolev, Sobolev embedding and $L^p$ space-time decay estimate, we obtain
$$
Z\left(e^{-it_n^k\Delta}\psi^k\right)\lesssim \|e^{-it_n^k\Delta}\psi^k\|_{L^{\frac{2Np}{N+\gamma}}}^{2p} \lesssim \|\psi^k-\tilde{\psi}\|_{\dot{H}^{\frac{N(p-1)-\gamma}{2p}}}+\left|t_n^k\right|^{-\frac{N(p-1)}{2(p+1)}}\|\tilde{\psi}\|_{L^{\frac{p+1}{p}}}.$$
Approximating $\psi^k$ by $\tilde{\psi}\in C_0^{\infty}$ in $\dot{H}^{\frac{N(p-1)-\gamma}{2p}}$ and sending $n\rightarrow\infty$, we obtain
\begin{align}\label{pairdivgexp2}
	\lim\limits_{n\rightarrow+\infty}Z\left(e^{-it_n^k\Delta}\psi^k\right)\lesssim \lim\limits_{n\rightarrow+\infty}\|e^{-it_n^k\Delta}\psi^k\|^{2p}_{L^{\frac{2Np}{N+\gamma}}} =0.
\end{align}
Thus, combining \eqref{pairdivgexp1} and \eqref{pairdivgexp2} together yields,
\begin{align*}
	Z\left(\sum_{j=1}^{M}e^{-it_n^j\Delta}\psi^j\right)&=Z\left(\sum_{j=1}^{M_0}\psi^j + \sum_{j=M_0+1}^{M}e^{-it_n^j\Delta}\psi^j\right)\\
	&=Z\left(\sum_{j=1}^{M_0}\psi^j\right) + \sum_{j=M_0+1}^{M}Z\left(e^{-it_n^j\Delta}\psi^j\right)+o_n(1)=\sum_{j=1}^{M}Z\left(e^{-it_n^j\Delta}\psi^j\right)+o_n(1),
\end{align*}
which is the right-hand side of the expansion \eqref{pairdivgexp}.\\

\noindent
\textit{Step 2.} \emph{Ending the proof.} Note that
\begin{align*}
	\|W_n^{M}\|_{L^{\frac{2Np}{N+\gamma}}}^{2p}&\leq \|e^{it\Delta}W_n^{M}\|_{L_t^{\infty}L_x^{\frac{2Np}{N+\gamma}}}^{2p}\leq \|e^{it\Delta}W_n^{M}\|_{L_t^{\infty}L_x^{\frac{2N}{N-2s_c}}}^p \|e^{it\Delta}W_n^{M}\|_{L_t^{\infty}L_x^{\frac{2Np}{N+\gamma-2(1-s_c)}}}^p\\
	&\leq \|e^{it\Delta}W_n^{M}\|_{L_t^{\infty}L_x^{\frac{2N}{N-2s_c}}}^p \|e^{it\Delta}W_n^{M}\|_{L_t^{\infty}\dot{H}_x^1}^p \leq \|e^{it\Delta}W_n^{M}\|_{L_t^{\infty}L_x^{\frac{2N}{N-2s_c}}}^p  \sup_n \|\phi_n\|_{H^1}^p.
\end{align*}
Since $\dot{H}^{s_c} \hookrightarrow L^{\frac{2N}{N-2s_c}}$, i.e., $\left(\infty,\frac{2N}{N-2s_c}\right)$ is an $\dot{H}^{s_c}$ admissible pair, by \eqref{smallrem}, we get
\begin{align}\label{smallrem1}
	\lim\limits_{M\rightarrow +\infty}\left(\lim\limits_{n\rightarrow +\infty}\|W_n^{M}\|_{\frac{2Np}{N+\gamma}}^{2p}\right) = 0.
\end{align}
Let $M\geq 1$ and $\varepsilon > 0$. Note that $\{\phi_n\}_n$ is uniformly bounded in $L^{\frac{2Np}{N+\gamma}}$, as it is uniformly bounded in $H^1$ by the hypothesis. Hence, by \eqref{smallrem1} $\{W_n^M\}_n$ is also uniformly bounded in $L^{\frac{2Np}{N+\gamma}}$. Hence, we can choose $M_1 > M$ and $n_1$ such that for $n>n_1$, we have
\begin{align}\label{PEdecomp1a}
	\big|Z(\phi_n) - Z(\phi_n - W_n^{M_1})\big| + \big|Z(W_n^M - W_n^{M_1}) &- Z(W_n^M)\big|\\
	\leq C\Bigg[\|W_n^{M_1}\|_{L^{\frac{2Np}{N+\gamma}}}\left(\sup_n\|\phi_n\|_{L^{\frac{2Np}{N+\gamma}}}^{2p-1}+\sup_n\|W_n^M\|_{L^{\frac{2Np}{N+\gamma}}}^{2p-1}\right)\Bigg]&+C\|W_n^{M_1}\|_{L^{\frac{2Np}{N+\gamma}}}^{2p}\leq\varepsilon,\notag
\end{align}
where we have used the triangle inequality to estimate
$$
\big|\|W_n^M-W_n^{M_1}\|_{L^{\frac{2Np}{N+\gamma}}}^{2p}-\|W_n^{M}\|_{L^{\frac{2Np}{N+\gamma}}}^{2p}\big|\lesssim \|W_n^{M_1}\|_{L^{\frac{2Np}{N+\gamma}}}^{2p},
$$
and by observing that $a^{2p} > a(a-b)^{2p-1}$ together with the triangle inequality, we estimate
\begin{align*}
	\Big|\|\phi_n-W_n^{M_1}\|_{L^{\frac{2Np}{N+\gamma}}}^{2p}-\|\phi_n\|_{L^{\frac{2Np}{N+\gamma}}}^{2p}\Big| &\lesssim \Big(\|\phi_n-W_n^{M_1}\|_{L^{\frac{2Np}{N+\gamma}}}-\|\phi_n\|_{L^{\frac{2Np}{N+\gamma}}}\Big)\|\phi_n-W_n^{M}\|_{L^{\frac{2Np}{N+\gamma}}}^{2p-1}\\
	& \lesssim \|W_n^{M_1}\|_{L^{\frac{2Np}{N+\gamma}}}\left(\sup_n\|\phi_n\|_{L^{\frac{2Np}{N+\gamma}}}^{2p-1}+\sup_n\|W_n^M\|_{L^{\frac{2Np}{N+\gamma}}}^{2p-1}\right).
\end{align*}
Choose $n_2\geq n_1$ such that for $n\geq n_2$, by \eqref{pairdivgexp}, we get
\begin{align}\label{PEdecomp1b}
	\Big|Z(\phi_n-W_n^{M_1}) - \sum_{j=1}^{M_1}Z\left(e^{-it_n^j\Delta}\psi^j\right)\Big|
	\leq \varepsilon.
\end{align}
Using the definition of $W_n^j$, we expand $W_n^M - W_n^{M_1}$, to obtain
$$
W_n^M-W_n^{M_1}=\sum_{j=M+1}^{M_1} e^{-it_n^j\Delta}\psi^j(\cdot - x_j).
$$
By \eqref{pairdivgexp} there exists $n_3\geq n_2$ such that for $n\geq n_3$,
\begin{align}\label{PEdecomp1c}
	\Big|Z(W_n^M - W_n^{M_1}) - \sum_{j=M+1}^{M_1} Z\left(e^{-it_n^j\Delta}\psi^j\right)\Big|\leq\varepsilon.
\end{align}
Thus, for $n\geq n_3$, by \eqref{PEdecomp1a}, \eqref{PEdecomp1b} and \eqref{PEdecomp1c}, we obtain
\begin{align*}
	Z(\phi_n) &- \sum_{j=1}^{M}Z\left(e^{-it_n^j\Delta}\psi^j\right) - Z(W_n^M) \\
	=& \bigg| Z(\phi_n) - Z(\phi_n - W_n^{M_1}) + Z(\phi_n - W_n^{M_1})
	-  \sum_{j=1}^{M_1} Z\left(e^{-it_n^j\Delta}\psi^j\right)
	+ Z(W_n^M - W_n^{M_1}) \\
	&- Z(W_n^M) + \sum_{j=M+1}^{M_1} Z\left(e^{-it_n^j\Delta}\psi^j\right) - Z(W_n^M - W_n^{M_1})\bigg|\leq 3\varepsilon,
\end{align*}
which implies \eqref{PEdecomp}.

\end{proof}

\subsection{Critical solution}
In this subsection, we study a critical solution of \eqref{gH}, denoted by $u_c(t)$. The main ingredients are Theorem \ref{linprodecomp} and Proposition \ref{energypythadecomp} (proved in previous subsection) along with Theorem \ref{LTP} (long time perturbation theory). 
\begin{theorem}[Existence of critical solution]\label{crit.elem}
Let $0<s_c<1$. There exists a global solution $u_c(t)\in H^1(\R^N)$ with initial data $u_{c,0}\in H^1(\R^N)$ such that
	$
	\|u_{c,0}\|_{L^2}=1$, $(\mathcal{ME})_c<1,
	$
	where $(\mathcal{ME})_c=\frac{E[u_c]}{M[Q]^{\theta}E[Q]}$,
	$
	\mathcal{G}[{u_c}(t)]<1\,\,\, \text{for all}~~\,\, 0\leq t\leq +\infty,
	$
	and
	\begin{align}\label{nonscat}
		\|u_c\|_{S(\dot{H}^{s_c})} = +\infty.
	\end{align}
\end{theorem}
\begin{proof}
	The argument for the linear profile expansion is similar to the one presented in \cite{HR08}, \cite{DHR08}, \cite{G}. Thus, we continue for a new nonlinear profile $\widetilde{\psi}^j$ associated to each original linear profile $\psi^j$ satisfying 
	\begin{align}\label{uc9}
	\|\text{NLF}(-t_n^j)\widetilde{\psi}^j-e^{-it_n^j\Delta}\psi^j\|_{H^1}\rightarrow 0
	\end{align}
	and
	\begin{align}\label{uc10}
	\|\text{NLF}(-t)\widetilde{\psi}^j\|_{S(\dot{H}^{s_c}}<+\infty.
	\end{align} 
	The idea now is to apply a nonlinear flow to $\phi_n(x)$ and approximate it by a combination of ``nonlinear bumps", i.e.,
	\begin{align*}
		\text{NLF}(t)\phi_n(x) \approx \sum_{j=1}^{M}\text{NLF}(t-t_n^j)\,\widetilde{\psi}^j(x-x^j_n).
	\end{align*}
	To carry out this argument, we introduce the nonlinear evolution of each separate initial condition $u_{n,0}=\phi_n$:
	$$
	u_n(t)=\text{NLF}(t)\phi_n(x)=\text{NLF}(t)u_{n,0},
	$$
	the nonlinear evolution of each separate nonlinear profile (``bump"):
	$$
	v^j(t)=\text{NLF}(t)\widetilde{\psi}^j,
	$$
	and a linear sum of nonlinear evolutions of those ``bumps":
	$$
	\widetilde{u}_n(t,x)=\sum_{j=1}^{M}v^j(t-t_n^j,x-x^j_n).
	$$
	Intuitively, we think that $u_{n,0}=\phi_n$ is a sum of nonlinear bumps $\widetilde{\psi}^j$ and $u_n(t)$ is a nonlinear evolution of their entire sum. On the other
	hand, $\widetilde{u}_n(t)$ is a sum of nonlinear evolutions of each bump and we want to
	compare $u_n(t)$ with $\widetilde{u}_n(t)$.
	Also, note that if we just had the linear evolutions, then both $u_n(t)$ and $\widetilde{u}_n(t)$
	would be the same.
	
	Thus, $u_n(t)$ satisfies
	$$
	i(u_n)_t+\Delta u_n + (|x|^{-(N-\gamma)}\ast |u_n|^p)|u_n|^{p-2}u_n=0,
	$$
	and $\widetilde{u}_n(t)$ satisfies
	$$
	i(\widetilde{u}_n)_t+\Delta \widetilde{u}_n + (|x|^{-(N-\gamma)}\ast |\widetilde{u}_n|^p)|\widetilde{u}_n|^{p-2}\widetilde{u}_n=\widetilde{e}_n^M,
	$$
	where
	$$
	\widetilde{e}_n^M = \left(|x|^{-(N-\gamma)}\ast |\widetilde{u}_n|^p\right)|\widetilde{u}_n|^{p-2}\widetilde{u}_n - \sum_{j=1}^{M}\left(|x|^{-(N-\gamma)}\ast |v^j(t-t_n^j,\cdot-x^j_n)|^p\right)|v^j|^{p-2}v^j.
	$$
	We also define 
	\begin{align}\label{NLrem}
		\widetilde{W}_n^M=W_n^M+\sum_{j=1}^M\left(e^{-it_n^j\Delta}\psi^j(x-x_n^j)-\text{NLF}(-t_n^j)\widetilde{\psi}^j(x-x_n^j)\right),
	\end{align}
	and using \eqref{NLPD_crit} we write
	\begin{align}\label{uc11}
		u_{n,0}=\sum_{j=1}^M\text{NLF}(-t_n^j)\widetilde{\psi}^j(x-x_n^j)+\widetilde{W}_n^M,
	\end{align}
	such that $u_{n,0}-\widetilde{u}(0)=\widetilde{W}_n^M$. Applying triangle inequality together with the Strichartz inequality \eqref{sobstri1}, we estimate
	$$
	\|e^{it\Delta}\widetilde{W}^M_n\|_{S(\dot{H}^{s_c})}\lesssim \|e^{it\Delta}W^M_n\|_{S(\dot{H}^{s_c})} + \sum_{j=1}^{M}\|e^{-it_n^j\Delta}\psi^j - NLF(-t_n^j)\widetilde{\psi}^j\|_{H^1}.
	$$
	By \eqref{uc9} and \eqref{smallrem} we have that
	\begin{align}\label{uc12}
		\lim\limits_{M\rightarrow\infty}\left(\lim\limits_{n\rightarrow\infty}\|e^{it\Delta}\widetilde{W}_n^M\|_{S(\dot{H}^{s_c})}\right)=0.
	\end{align}
	We now approximate $u_n$ by $\widetilde{u}_n$. Then from the Theorem \ref{LTP} (long time perturbation theory) and \eqref{uc10} it follows that for $n$ large enough, $\|u_n\|_{S(\dot{H}^{s_c})}<+\infty$, which is a contradiction, since $u_n$ is non-scattering. We assume the following two claims, which we prove later.
	\begin{claim}\label{LTPclaim1}
		There exists a constant $A$ independent of $M$, and for every $M$, there
		exists $n_0 = n_0(M)$ such that if $n > n_0$, then $\|\widetilde{u}_n\|_{S(\dot{H}^{s_c})}\leq A$.
	\end{claim}
	\begin{claim}\label{LTPclaim2}
		For each $M$ and $\epsilon > 0$, there exists $n_1 = n_1(M; \epsilon)$ such that if $n > n_1$,
		then $\|\widetilde{e}^M_n\|_{S'(\dot{H}^{-s_c})}\leq\epsilon$.
	\end{claim}
	By \eqref{uc12}, for any $\epsilon >0$ there exists $M_1=M_1(\epsilon)$ sufficiently large such that for each $M>M_1$ there exists $n_2=n_2(M)$ such that $n>n_2$ implies
	$$
	\|e^{it\Delta}(\tilde{u}_n(0)-u_n(0))\|_{S(\dot{H}^{s_c})}\leq\epsilon.
	$$
	Thus, if the Claim \ref{LTPclaim1} and Claim \ref{LTPclaim2} holds true, using Theorem \ref{LTP} for sufficiently large $M$ and $n=\max(n_0,n_1,n_2)$
we obtain $\|u_n\|_{S(\dot{H}^{s_c})} < \infty$, a contradiction, since $u_n$ is non-scattering. Now there are two possible scenarios in the profile decomposition \eqref{uc11}:
	
	\underline{\textit{Scenario 1:}} More that one $\widetilde{\psi}^j\neq 0$. Observe that for $s=0$ in \eqref{pythagexp}, we have
	\begin{align}\label{massdecomp_crit}
	\sum_{j=1}^{M} \|\psi^j\|_{L^2}^2 + \lim\limits_{n\rightarrow +\infty}\|W_n^M\|_{L^2}^2\leq\lim\limits_{n\rightarrow +\infty}\|u_{n,0}\|_{L^2}^2=1.
	\end{align}
	Thus, by \eqref{massdecomp_crit}, we must have $M[e^{-it_n^j\Delta}\widetilde{\psi}^j]<1$ for each $j$, which by energy decomposition, 
for large enough $n$ yields
	$$
	\frac{M[\text{NLF}(t)\widetilde{\psi}^j]^{1-s_c}E[\text{NLF}(t)\widetilde{\psi}^j]^{s_c}}{M[Q]^{1-s_c}E[Q]^{s_c}}=\frac{M[\widetilde{\psi}^j]^{1-s_c}E[\widetilde{\psi}^j]^{s_c}}{M[Q]^{1-s_c}E[Q]^{s_c}}=\mathcal{ME}[\widetilde{\psi}^j]<(ME)_c.
	$$
	Now, since $\|\text{NLF}(t)\widetilde{\psi}^j(\cdot -x_n^j)\|_{S(\dot{H}^{s_c})}<+\infty$, the right hand side of \eqref{uc11} is bounded in $S(\dot{H}^{s_c})$. By \eqref{uc12}, we conclude that $\|\text{NLF}(t)u_{n,0}\|_{S(\dot{H}^{s_c})}<+\infty$, which is a contradiction.
	
	\underline{\textit{Scenario 2:}} Suppose $\widetilde{\psi}^1\neq 0$ and $\widetilde{\psi}^j=0$ for all $j\geq 2$. Hence, we have
	$$
	u_{n,0}=\text{NLF}(-t_n^1)\widetilde{\psi}^1(x-x^1_n)+\widetilde{W}^1_n
	$$
	with
	$$
	M[\widetilde{\psi}^1]<1,\,\,\, \mathcal{ME}[\widetilde{\psi}^1]\leq(\mathcal{ME})_c,\,\,\, \text{and} \lim\limits_{n\rightarrow +\infty}\|e^{it\Delta}(t)\widetilde{W}_n^1\|_{S(\dot{H}^{s_c})}=0.
	$$
	Let $u_c$ be the global solution to \eqref{gH} with initial data $u_{c,0}=\widetilde{\psi}^1$ i.e., $u_c(t)=\text{NLF}(t)\widetilde{\psi}^1$. Assume by contradiction that $\|u_c\|_{S(\dot{H}^{s_c})}<+\infty$. Let $\widetilde{u}_n(t)=\text{NLF}(t-t_n^1)\widetilde{\psi}^1$, then 
	\begin{align*}
		\|\widetilde{u}_n(t)\|_{S(\dot{H}^{s_c})}=\|\text{NLF}(t-t_n^1)\widetilde{\psi}^1\|_{S(\dot{H}^{s_c})}=\|u_c\|_{S(\dot{H}^{s_c})}<+\infty.
	\end{align*}   
	Therefore, using the long time perturbation theory with $e=0$, we deduce that $\|u_n\|_{S(\dot{H}^{s_c})}<+\infty$, which is a contradiction, since by construction $u_n$ is non-scattering. It only remains to establish the claims \ref{LTPclaim1} and \ref{LTPclaim2}.
	
	\textbf{Proof of Claim \ref{LTPclaim1}:} See \cite{AKA2} or the original NLS works \cite{DHR08}, \cite{HR08}, \cite{G} for details. 
	
	\textbf{Proof of Claim \ref{LTPclaim2}:} Recall that $\left(\frac{2p}{s_c(2p-1)+1},\frac{2Np}{N+\gamma}\right)$ is an $\dot{H}^{-s_c}$ admissible pair. Then
	\begin{align*}
	\|\widetilde{e}_n^M\|_{S'(\dot{H}^{-s_c})} \leq \|\widetilde{e}_n^M\|_{L_t^{\frac{2p}{(1-s_c)(2p-1)}}L_x^{\frac{2Np}{2Np-N-\gamma}}}.
	\end{align*}
	Observe that expansion of $\widetilde{e}_n^M$ consists of cross terms of the form
	$$
	\sum_{j=1}^{M}\sum_{k=1}^{M}\sum_{\substack{l=1 \\k\ne l}}^{M}\left(|x|^{-(N-\gamma)}\ast|v^j(t-t_n^j)|^p\right)|v^k(t-t_n^k)|^{p-2}v^l(t-t_n^l),
	$$
	where all of $j$, $k$ and $l$ are not same. Assume, without loss of generality, that $k\neq l$, and thus, $|t_n^k-t_n^l|\rightarrow\infty$ as $n\rightarrow +\infty$. So, we estimate
	\begin{align*}
	\|\big(|x|^{-(N-\gamma)}&\ast|v^j(t-t_n^j)|^p\big)|v^k(t-t_n^k)|^{p-2}v^l(t-t_n^l)\|_{L_t^{\frac{2p}{(1-s_c)(2p-1)}}L_x^{\frac{2Np}{2Np-N-\gamma}}} \\
	&\leq \|v^j\|_{L_t^{\frac{2p}{1-s_c}}L_x^{\frac{2Np}{N+\gamma}}}\||v^k(t-t_n^k)|^{p-2}v^l(t-t_n^l)\|_{L_t^{\frac{2p}{(1-s_c)(p-1)}}L_x^{\frac{2Np}{(N+\gamma)(p-1)}}}.
	\end{align*}
	Since both $v^k$ and $v^l$ belong to $L_t^{\frac{2p}{1-s_c}}L_x^{\frac{2Np}{N+\gamma}}$, then
	$$
	\||v^k(t-(t_n^k-t_n^l))|^{p-2}v^l(t)\|_{L_t^{\frac{2p}{(1-s_c)(p-1)}}L_x^{\frac{2Np}{(N+\gamma)(p-1)}}} \rightarrow 0.
	$$
	This gives Claim \ref{LTPclaim2}, which completes the proof of Theorem \ref{crit.elem}.
\end{proof}

For the proof of the following Proposition and Lemmas  \ref{unilocalization}, \ref{zeromoment} and \ref{pathcontrol}, see \cite{HR08}, \cite{DHR08} and \cite{G} (or refer to \cite{AKA2} for details).

\begin{proposition}[Precompactness of the flow of the critical solution]\label{precompflow_crit}
	Assume $u_c$ as in Theorem \ref{crit.elem}. Then there exists a continuous path $x(t)$ in $\R^N$ such that
	$$
	K=\{u_c(\cdot-x(t),t)\,\,|\,\,t\in [0,\infty)\}
	$$
	is precompact in $H^1$ (i.e., $\overline{K}$ is compact in $H^1$).
\end{proposition}

\begin{lemma}[Precompactness of the flow implies uniform localization]\label{unilocalization}
	Let $u$ be a solution to \eqref{gH} such that
	$$
	K=\{u(\cdot-x(t),t)\,\,|\,\,t\in [0,\infty)\}
	$$
	is precompact in $H^1$. Then for each $\epsilon >0$, there exists $R>0$ so that
	\begin{align}\label{uniloc}
	\int_{|x+x(t)|>R}^{}|\nabla u(x,t)|^2 +|u(x,t)|^2 dx < \epsilon
	\end{align}
	for all $0\leq t<\infty$.
\end{lemma}

\begin{lemma}[Zero momentum]\label{zeromoment}
	Let $u_c$ be the critical solution constructed in Theorem \ref{crit.elem} and assume $(\mathcal{ME})_c < 1$. Then $P[u_c]=\Im\int \bar{u}_c\nabla u_c\, dx = 0$.
\end{lemma}

Next, observe that
$$
\frac{\partial}{\partial t}\int x|u(x,t)|^2\,dx=2N\Im\int\bar{u}\nabla u\,dx=2NP[u].
$$
Since $P[u_c]=0$ (Lemma \ref{zeromoment}), this implies that $\int x|u_c(x,t)|^2\,dx=\,\text{constant}$, provided it is finite. We will replace this identity with a localized version adapted to a suitably large
radius $R > 0$. To envelope the entire path $x(t)$ over $[T, T_1]$ the localization $R$ should be taken large enough over the same interval $[T,T_1]$. We can use the precompactness of the translated flow $u_c(\cdot-x(t),t)$
and the zero momentum to prove that the localized center of mass is nearly
conserved. By the localization of $u_c$ in $H^1$ around $x(t)$ and the near conservation of localized
center of mass we constrain parameter $x(t)$ from going too quickly to $+\infty$.

\begin{lemma}[Control over $x(t)$]\label{pathcontrol}
	Let $u$ be a solution of \eqref{gH} defined on $[0,+\infty)$ such that $P[u]=0$ and $	K=\{u(\cdot-x(t),t)\,\,|\,\,t\in [0,\infty)\}
	$ is precompact in $H^1$, for some continuous function $x(\cdot)$. Then
	\begin{align}\label{control}
	\frac{x(t)}{t}\rightarrow 0\quad \text{as}\,\,\, t\rightarrow +\infty.
	\end{align}
\end{lemma}

\section{Rigidity Theorem} 
\label{rigid}
\begin{theorem}[Rigidity]\label{rigidity}
	Let $u$ be the global solution of \eqref{gH} with initial data $u_0\in H^1(\R^N)$ satisfying $P[u_0]=0$, $
	\mathcal{ME}[u_0]<1$ and $\mathcal{G}[u_0]<1.$ Suppose $K=\{u(\cdot-x(t),t)\,\,|\,\,t\in [0,\infty)\}$ is precompact in $H^1$. Then $u_0\equiv 0$.
\end{theorem}
\begin{proof}
	Let $\phi\in C_0^{\infty}$ be radial, with
	$\phi(x)=|x|^2$ for $|x|\leq 1$ and $0$ for $|x|\geq 2$.
	For $R>0$, let $\phi_R(x)=R^2\phi(x/R)$. Define
	\begin{align}\label{rigidity1-2}
	V_{loc}(t)=\int \phi_R(x)|u(x,t)|^2\,dx\implies 	V_{loc}^{\prime}(t)=2R\Im\int\bar{u}(t)\cdot\nabla u(t)\,(\nabla\phi)\left(\frac{x}{R}\right)\,dx.
	\end{align}
	Using H\"older's inequality, we get
	\begin{align}\label{rigidity3}
		|V_{loc}^{\prime}(t)|\leq CR\int_{|x|\leq 2R}^{}|u(t)|\,|\nabla u(t)|\,dx\leq CR\|u(t)\|^{2(1-s_c)}_{L^2}\|\nabla u(t)\|_{L^2}^{2s_c}.
	\end{align}
	The second derivative,
	using the definition of $\phi$ and symmetrization, yields 
	\begin{align*}
		V_{loc}^{\prime\prime}(t)\geq &\,8\int_{|x|\leq R}^{}|\nabla u|^2-\frac{4(N(p-1)-\gamma)}{p}\int_{|x|\leq R}^{}\left(|x|^{-(N-\gamma)}\ast|u|^p\right)|u|^p\\
		&-\frac{c}{R^2}\int_{R<|x|<2R}^{} |u|^2 +4\int_{R<|x|<2R}^{}\phi''\left(\frac{|x|}{R}\right)|\nabla u|^2\\
		&-\left(4\,c\left(\frac{1}{2}-\frac{1}{p}\right)+\frac{2(N-\gamma)}{p}\right)\int_{R<|x|<2R}^{}\left(|x|^{-(N-\gamma)}\ast|u|^p\right)|u|^p\\
		&+ \frac{2(N-\gamma)}{p}\int\int_{\Omega}\left(1-\frac{R}{|x|}\phi'\left(\frac{|x|}{R}\right)\right)\frac{x(x-y)|u(x)|^p|u(y)|^p}{|x-y|^{N-\gamma+2}}\,dxdy\\
		&-\frac{2(N-\gamma)}{p}\int\int_{\Omega}\left(1-\frac{R}{|y|}\phi'\left(\frac{|y|}{R}\right)\right)\frac{y(x-y)|u(x)|^p|u(y)|^p}{|x-y|^{N-\gamma+2}}\,dxdy.
		\end{align*}
		We re-write the above estimate as
		\begin{align}\label{rigidity3.5}
		V_{loc}^{\prime\prime}(t)\geq& \left(8\int_{|x|\leq R}^{}|\nabla u|^2-\frac{4(N(p-1)-\gamma)}{p}\int_{|x|\leq R}^{}\left(|x|^{-(N-\gamma)}\ast|u|^p\right)|u|^p\right)\\\notag
		&-c_1\left(\int_{R<|x|<2R}^{}|\nabla u|^2+\frac{|u|^2}{R^2}+\left(|x|^{-(N-\gamma)}\ast|u|^p\right)|u|^p\right)\\\label{rigidity4}
		&+ \frac{2(N-\gamma)}{p}\int\int_{\Omega}\left(1-\frac{R}{|x|}\phi'\left(\frac{|x|}{R}\right)\right)\frac{x(x-y)|u(x)|^p|u(y)|^p}{|x-y|^{N-\gamma+2}}\,dxdy\\\label{rigidity5}
		&-\frac{2(N-\gamma)}{p}\int\int_{\Omega}\left(1-\frac{R}{|y|}\phi'\left(\frac{|y|}{R}\right)\right)\frac{y(x-y)|u(x)|^p|u(y)|^p}{|x-y|^{N-\gamma+2}}\,dxdy,
	\end{align}
	where
	$$
	\Omega = \{(x,y)\in \R^N\times \R^N\,:\, |x|>R\}\cup\{(x,y)\in \R^N\times \R^N\,:\, |y|>R\}.
	$$Since $\{u(\cdot-x(t),t)\,\,|\,\,t\in [0,\infty)\}$ is precompact in $H^1(\R^N)$, by Lemma \ref{unilocalization} there exists $R_0\geq 0$ such that taking $R\geq R_0 +\sup_{t\in[T,T_1]}|x(t)|$, we obtain for all $t\in [T,T_1]$
	\begin{align}\label{rigidity6}
	\int_{|x|>R_0}^{}|\nabla u(x,t)|^2 +|u(x,t)|^2 dx < \frac{\epsilon}{8}.	
	\end{align}
	\noindent
	Using H\"older's inequality, Hardy-Littlewood-Sobolev inequality and radial Sobolev inequality  yields the existence of $R_1>0$ such that
	\begin{align}\label{rigidity7}\notag
	\int_{|x|>R_1}^{}&\left(|x|^{-(N-\gamma)}\ast|u|^p\right)|u|^p\,dx\\\notag
	&\leq c_2 \||\cdot|^{-(N-\gamma)}\ast|u|^p\|_{L_{|x|>R_1}^{\frac{2N}{N-\gamma}}}\|u\|^p_{L_{|x|>R_1}^{\frac{2Np}{N+\gamma}}}\quad \text{(H\"older's inequality)}\\\notag
	&\leq c_3 \|u\|^{2p}_{L_{|x|>R_1}^{\frac{2Np}{N+\gamma}}}\quad \text{(HLS inequality; Lemma \ref{HLS})}\\
	&\leq \frac{c_4}{R_1^{\frac{(N-1)(N(p-1)-\gamma)}{N}}}\| u\|^{\frac{N(p-1)-\gamma}{N}}_{\dot{H}^1}\|u\|^{\frac{N(p+1)+\gamma}{N}}_{L^2} <\frac{\epsilon}{8},
	\end{align}
	where the second to last inequality follows from the radial Sobolev inequality and the last one follows from taking $\displaystyle{R_1^{\frac{(N-1)(N(p-1)-\gamma)}{N}}>8\,c_4\,\| u\|^{\frac{N(p-1)-\gamma}{N}}_{\dot{H}^1}\|u\|^{\frac{N(p+1)+\gamma}{N}}_{L^2}}$. 
	
	Let $\epsilon=16E[u]\left(1-(\mathcal{M}\mathcal{E}[u])^{s_c(p-1)}\right)c_1^{-1}$ and take $R=\max\{R_0,R_1\}$, combine \eqref{rigidity6} and \eqref{rigidity7} to obtain
	\begin{align}\label{6+7}
		c_1\left(\int_{|x|>R}^{}|\nabla u|^2+\frac{|u|^2}{R^2}+\left(|x|^{-(N-\gamma)}\ast|u|^p\right)|u|^p\right)\leq 4E[u]\left(1-(\mathcal{M}\mathcal{E}[u])^{s_c(p-1)}\right).
	\end{align}
	Now we invoke Lemma \ref{bndconvexvar} by splitting the integrals on the right side of the expression \eqref{lowbndvar} into the regions $\{|x| >R\}$ and $\{|x| < R\}$ and use \eqref{6+7} to obtain the following bound, $\eqref{rigidity3.5}\geq 12E[u]\left(1-(\mathcal{M}\mathcal{E}[u])^{s_c(p-1)}\right).$
	Next, we estimate the terms \eqref{rigidity4} and \eqref{rigidity5}
	\begin{align}\label{rigidity8}
	\int\int_{\Omega}\left(\left(1-\frac{R}{|x|}\phi'\left(\frac{|x|}{R}\right)\right)x-\left(1-\frac{R}{|x|}\phi'\left(\frac{|x|}{R}\right)\right)y\right)\frac{(x-y)|u(x)|^p|u(y)|^p}{|x-y|^{N-\gamma+2}}\,dxdy,	
	\end{align}
	where we follow the argument as we did in Theorem \ref{Dichotomy} to obtain
	\begin{align*}
		\eqref{rigidity8}\leq \frac{c_5}{R^{\frac{(N-1)(N(p-1)-\gamma)}{N}}}\| \chi_{|x|>R}u\|^{\frac{N(p-1)-\gamma}{N}}_{\dot{H}^1}<\frac{\epsilon}{4}<4E[u]\left(1-(\mathcal{M}\mathcal{E}[u])^{s_c(p-1)}\right)
	\end{align*}
	with $\displaystyle{R^{\frac{(N-1)(N(p-1)-\gamma)}{N}}>4\,c_5\| \chi_{|x|>R}u\|^{\frac{N(p-1)-\gamma}{N}}_{\dot{H}^1}}$. Putting everything together, we obtain
	\begin{align}\label{rigidity11}
		V_{loc}^{\prime\prime}(t)\geq 8E[u]\left(1-(\mathcal{M}\mathcal{E}[u])^{s_c(p-1)}\right)-|I_R|\geq  4E[u]\left(1-(\mathcal{M}\mathcal{E}[u])^{s_c(p-1)}\right).
	\end{align}
	By Lemma \ref{pathcontrol}, there exists $T\geq 0$ such that for all $t \geq T$ , we have $|x(t)| \leq \delta t$, with $\delta>0$ to be chosen later. Taking $R = R_0 + \delta T_1$, we have that \eqref{rigidity4} holds for all $t \in [T, T_1]$, then integrating from $T$ to $T_1$, we obtain
	\begin{align}\label{rigidity12}
	|V_{loc}^{\prime}(T_1)-V_{loc}^{\prime}(T)|\geq 4E[u]\left(1-(\mathcal{M}\mathcal{E}[u])^{s_c(p-1)}\right)(T_1-T).
	\end{align}
	On the other hand, from \eqref{rigidity3} and \eqref{eq:dichotomy2}, we have that
	\begin{align}\label{rigidity13}
		|V_{loc}^{\prime}(t)|\leq CR\|u(t)\|^{2(1-s_c)}_{L^2}\|\nabla u(t)\|_{L^2}^{2s_c}\leq C\Big(R_0 + \delta T_1\Big)\|Q\|^{2(1-s_c)}_{L^2}\|\nabla Q\|_{L^2}^{2s_c}.
	\end{align}
	Combining \eqref{rigidity12} and \eqref{rigidity13}, we get
	$$
	4E[u]\left(1-(\mathcal{M}\mathcal{E}[u])^{s_c(p-1)}\right)(T_1-T)\leq  C\Big(R_0 + \delta T_1\Big)\|Q\|^{2(1-s_c)}_{L^2}\|\nabla Q\|_{L^2}^{2s_c}.
	$$
	Let $\delta=\frac{E[u]\left(1-(\mathcal{M}\mathcal{E}[u])^{s_c(p-1)}\right)}{C\|Q\|^{2(1-s_c)}_{L^2}\|\nabla Q\|_{L^2}^{2s_c}}$, then the above expression can be re-written as
	$$
	3E[u]\left(1-(\mathcal{M}\mathcal{E}[u])^{s_c(p-1)}\right)T_1 \leq CR_0\|Q\|^{2(1-s_c)}_{L^2}\|\nabla Q\|_{L^2}^{2s_c}+4E[u]\left(1-(\mathcal{M}\mathcal{E}[u])^{s_c(p-1)}\right)T,
	$$
	taking $T_1\rightarrow +\infty$ implies that the left hand side of the above expression goes to $\infty$ and we derive a contradiction (right hand side is bounded), which can be resolved only if $E[u] = 0$, implying that $u\equiv 0$.
\end{proof}

\section{Divergence to infinity (Theorem \ref{mainP0} (2) part (b))}\label{S-last}

The argument for part (2)b follows \cite{HR10} and \cite{G} proof verbatim. We give a brief overview here for the sake of completeness. 

Assume that there is no finite time blowup for a nonradial and infinite variance solution
(from Theorem \ref{mainP0} part (2)b), thus, the solutions exists for all time (i.e., $T^* =+\infty$). Under this assumption of global existence, we study the behavior of
$\mathcal{G}[u(t)]$ as $t\rightarrow+\infty$, and use a concentration compactness type argument to establish
the divergence of $\mathcal{G}[u(t)]$ in $H^1$ as it was developed in \cite{HR10}, note that the concentration
compactness and the rigidity arguments are used to prove a blowup property. 

We first restate (in the spirit of \cite{HR10}) the characterization of $Q$ from Lions \cite{Lions84I}, Theorem II.1, which can be considered for any minimizer $Q$.
 \begin{proposition}\label{wblowup1}
	There exists a function $\epsilon(\rho)$, defined for small $\rho>0$ with $\lim\limits_{\rho\rightarrow 0}\epsilon(\rho)=0$, such that for all $u\in H^1(\R^N)$ with 
	\begin{align}\label{LionsVar1}
	\Big|Z(u)-Z(Q)\Big|+\Big|\|u\|_{L^2}-\|Q\|_{L^2}\Big|+\Big|\|\nabla u\|_{L^2}-\|\nabla Q\|_{L^2}\Big|\leq \rho,
	\end{align}
	there is $\theta_0\in\R$ and $x_0\in\R^N$ such that 
	\begin{align}\label{LionsVar2}
	\|u-e^{i\theta_0}Q(\cdot-x_0)\|_{H^1}\leq \epsilon(\rho).
	\end{align}
\end{proposition}
This is equivalent to
\begin{proposition}\label{wblowup2}
	There exists a function $\epsilon(\rho)$ such that $\epsilon(\rho)\rightarrow 0$ as $\rho\rightarrow 0$ satisfying the following: Suppose there exists $\lambda>0$ such that
	\begin{align}\label{wbu1}
	\left|\mathcal{ME}[u]-\frac{s_c(p-1)+1}{s_c(p-1)}\left(1-\frac{\lambda^{2s_c(p-1)}}{s_c(p-1)+1}\right)\lambda^2\right|\leq \rho\lambda^{2s_c(p-1)+2}
	\end{align}
	and
	\begin{align}\label{wbu2}
	\left|\mathcal{G}[u(t)]-\lambda\right|\leq \rho\begin{cases}
	\lambda^{2s_c(p-1)+1}&\text{if}\,\,\lambda\leq 1\\
	\lambda&\text{if}\,\,\lambda\geq 1
	\end{cases}.
	\end{align}
	Then there exists $\theta_0\in\R$ and $x_0\in \R^N$ such that
	\begin{align}\label{wbu3}
	\|u-e^{i\theta_0}\lambda^{N/2}\beta^{-\frac{s_c}{1-s_c}}Q(\lambda(\beta^{-\frac{3s_c}{(1-s_c)N}}x-x_0))\|_{L^2}\leq \beta^{\frac{s_c}{2(1-s_c)}}\epsilon(\rho)
	\end{align}
	and
	\begin{align}\label{wbu4}
	\|\nabla[u-e^{i\theta_0}\lambda^{N/2}\beta^{-\frac{s_c}{1-s_c}}Q(\lambda(\beta^{-\frac{3s_c}{(1-s_c)N}}x-x_0))]\|_{L^2}\leq \lambda\beta^{-\frac{s_c}{2(1-s_c)}}\epsilon(\rho),
	\end{align}
	where $\beta =\left(\frac{M[u]}{M[Q]}\right)^{\theta}$.
\end{proposition}
Suppose that $0\leq\mathcal{ME}[u]<1$ and let $\mathcal{G}[u(t)]=\lambda>0$ be given. The ``mass-energy" horizontal line for this $\lambda$ intersects the graph of parabola, $y=\frac{s_c(p-1)+1}{s_c(p-1)}\left(1-\frac{\lambda^{2s_c(p-1)}}{s_c(p-1)+1}\right)\lambda^2$ at two places, i.e.,  there exists two solutions $0\leq \lambda_1<1<\lambda_2$. The first case produces a solution that is global and scattering (Theorem \ref{mainP0} (1)) and the second case produces a solution which
either blows up in finite time (Theorem \ref{mainP0} (2)(a)) or diverge in infinite time (Theorem \ref{mainP0} (2)(b)) as shown in Section \ref{S-last}.

It is possible that $\mathcal{G}[u(t)]$ is much larger than $1$ or $\lambda_2$. The following Proposition shows that it cannot.
\begin{proposition}\label{boundary}
	Let $\mathcal{G}[u_0]=\lambda_0>1$. Then there exists $\rho_0=\rho_0(\lambda_0)>0$ (with the property that $\rho_0\rightarrow 0$ as $\lambda_0\searrow 1$) such that for any $\lambda\geq \lambda_0$, the following
	holds: There does NOT exist a solution $u(t)$ of \eqref{gH} with $P[u] = 0$ satisfying $\|u\|_{L^2} = \|Q\|_{L^2}$ and 
	\begin{align}\label{bc1}
	\frac{E[u]}{E[Q]}=\frac{s_c(p-1)+1}{s_c(p-1)}\left(1-\frac{\lambda^{2s_c(p-1)}}{s_c(p-1)+1}\right)\lambda^2
	\end{align}
	with 
	\begin{align}\label{bc2}
	\lambda\leq \frac{\|\nabla u(t)\|_{L^2}}{\|\nabla Q\|_{L^2}}\leq \lambda(1+\rho_0)\quad \text{for all}\,\,t\geq 0.
	\end{align}
\end{proposition}
\begin{proof}
	The proof relies on Proposition \ref{wblowup2} and is easy to adapt as done in \cite{HR10} and \cite{G} following the same argument as in Theorem \ref{rigidity} (Section \ref{rigid}) in this paper.
\end{proof}  
This proves that there is NO solution at the ``mass-energy"
line for $\lambda$ satisfying \eqref{bc2}. We want to show that $\mathcal{G}[u(t)]$ on any ``mass-energy" line with $\mathcal{ME}[u_0]<1$ and $\mathcal{G}[u(t)] >1$ will diverge to infinity. By contradiction, we assume that such solutions have bounded renormalized gradient $\mathcal{G}[u(t)]$ for all $t>0$.

We say the solution has a globally bounded gradient if there exists a solution at the ``mass-energy" line for $\lambda$ such that $\lambda\leq \mathcal{G}[u(t)]\leq \sigma$ for all $t>0$. Observe that if the solution does not have a globally bounded gradient for some $\lambda$ and $\sigma$, then for any $\sigma^{\prime}<\sigma$ the solution still does not have globally bounded gradient. We are now in a position to define the threshold.

\begin{definition}
	Fix $\lambda_0 >1$. Let $\sigma_c=\sigma_c(\lambda_0)$ be the supremum of all $\sigma>\lambda_0$ such that the solution does NOT have a globally bounded gradient for all $\lambda$ such that $\lambda_0\leq \lambda\leq \sigma$.
\end{definition}
 
 By Proposition \ref{boundary}, we have that $\lambda\leq \mathcal{G}[u(t)]\leq \lambda(1+\rho_0)$ does not hold for all $\lambda\geq \lambda_0$. We want to prove that $\sigma_c(\lambda_0)=+\infty$. By contradiction,
assume that $\sigma_c(\lambda_0)$ is finite.
Let $u(t)$ be a solution to \eqref{gH} with initial data $u_{n,0}$ at the ``mass-energy" line
for $\lambda>\lambda_0$, satisfying the hypothesis of Proposition \ref{boundary}. Moreover, we want to prove that $\mathcal{G}[u(t)]\rightarrow\infty$ over a sequence of times $\{t_n\}\rightarrow\infty$. Assume that such a sequence of times does not exist. This implies that there is a finite $\sigma$ satisfying $\lambda\leq \mathcal{G}[u(t)]\leq \sigma$ for all $t>0$. Invoking the nonlinear profile decomposition on the sequence $\{u_{n,0}\}$ as done in Theorem \ref{crit.elem} enables us to construct a ``critical threshold solution" $u(t)=u_c(t)$ at
the ``mass-energy" line for $\lambda_c$ with $\lambda_0 <\lambda_c <\sigma_c(\lambda_0)$ and $\lambda_c <\mathcal{G}[u_c(t)]< \sigma_c(\lambda_0)$ for all $t>0$. At this point we note that the nonlinear profile decomposition gives the $\dot{H}^1$ asymptotic orthogonality at $t=0$, but we would need to extend this for $0\leq t\leq T$. This can be done following the argument described in \cite{HR10} (Lemma 6.3) and \cite{G} (Lemma 3.9). This critical threshold solution $u_c(t)$ will satisfy Proposition \ref{precompflow_crit} (precompactness of the flow) and Lemma \ref{unilocalization} (uniform localization). This localization property of $u_c(t)$ implies that  $u_c(t)$ blows-up in finite time. The arguments from \cite{HR10} (Proposition 3.2) and \cite{G} (Lemma 4.10) proves exactly that, which contradicts the boundedness of $u_c(t)$ in $H^1$, and hence, $u_c(t)$ cannot exist, which means that our initial assumption that $\sigma_c(\lambda_0)<\infty$ is false. This completes the proof of Theorem \ref{mainP0}. 

\appendix
\section{Uniqueness of the ground state for $p=2$, $\gamma=2$}
Here for completeness we review the uniqueness of the ground state argument to the nonlocal elliptic (Choquard) equation 
\begin{equation}\label{HQ}
-Q + \Delta Q +\left(|{x}|^{-(N-2)}\ast|Q|^{2}\right)Q=0,
\end{equation}
since the argument is different from that for a local nonlinearity. As it was mentioned in the introduction, for $N=3$ the uniqueness is proved by Lieb \cite{Lieb83}, a slightly different proof using the comparison argument is in Lenzmann \cite{L09}; for $N=4$ it is proved in Krieger-Lenzmann-Raphael \cite{KLR09} via a combination of the above. We also follow the above arguments in 3d and generalize it for $2<N<6$. 
The stationary equation \eqref{HQ} appears in the context of the Hartree equation only in dimensions $2<N<6$: in dimension $N=6$ the Hartree equation is energy-critical, and thus, the corresponding elliptic equation will be different (lacking the linear term). While most of the arguments below work for dimensions 6 and higher, the equation \eqref{HQ} is only needed for $N<6$.

\begin{theorem}\label{uniqueQ}
Let $2 < N < 6$. The equation \eqref{HQ} has the unique positive, radial solution $Q$ in $H^1(\R^N)$. 	
\end{theorem}

The proof uses the following representation of the Newton's potential, which can be found in the textbook \cite[Theorem 9.7]{LiebL2001}.
 
\begin{lemma}\label{newton}
If $f$ is a radial $C^{\infty}$ function on $\R^N$, then 
\begin{align}\label{n1}
- \left( \frac1{|x|^{N-2}} \ast f \right)(r) =  
\int_{0}^{r}K(r,s)f(s)\,ds -\big|\mathbb{S}^{N-1}\big|\int_{0}^{\infty}f(s)s\,ds,
\end{align}
where
\begin{align}\label{n3}
K(r,s)=\big|\mathbb{S}^{N-1}\big|\left(1-\left(\frac{s}{r}\right)^{N-2}\right)s \geq 0~ \text{for} ~ r \geq s.
\end{align}

\end{lemma}

\begin{proof}
\textit{(of Theorem \ref{uniqueQ})}
Using Lemma \ref{newton} for a radial $Q\in H^1(\R^N)$, we rewrite \eqref{HQ} as
\begin{align}\label{uq1}
- Q^{\prime\prime} - \frac{N-1}{r} Q^{\prime} + \left(\int_{0}^{r}K(r,s)Q(s)^{2}\,ds\right) Q = a Q,
\end{align}
where $a=-1 + \big|\mathbb{S}^{N-1}\big| \left(\int_{0}^{\infty}Q(s)^{2} s \, ds \right) > 0$. Using the rescaling $Q(r)\mapsto a^{-1}Q(a^{-1/2}r)$, we obtain the version of \eqref{uq1} with $a=1$, namely, 	
\begin{equation}\label{uq2}
\left(-\frac{d^2}{dr^2}-\frac{N-1}{r}\frac{d}{dr}+U_Q(r)\right)Q(r)=Q(r),
\end{equation}
where 
\begin{equation}\label{uq3}
U_Q(r)=\left(\int_{0}^{r}K(r,s)Q(s)^{2}\,ds\right).
\end{equation}
Suppose $Q_1(r)$ and $Q_2(r)$ are two positive radial solutions of \eqref{uq3} in $H^1(\R^N)$ such that $Q_1\neq Q_2$ that 
solve the IVP
\begin{align}\label{uq4}
\begin{cases}
Q^{\prime\prime}(r)+\frac{N-1}{r}Q^{\prime}(r)+Q(r)-U_Q(r)Q(r)=0,\\
Q(0)=Q_0,\quad Q^{\prime}(0)=0.
\end{cases}
\end{align}
The Volterra integral theory (for example, see Lemmas 2.4-2.6 and Theorem 2.1 in \cite{YRZ2}) guarantees existence and uniqueness of a local $C^2$ solution to the above initial-value problem for a given $Q(0)$ (note that $U_Q(r)$ is bounded, see details below). Therefore, if $Q_1\ne Q_2$, then $Q_1(0)\ne Q_2(0)$. Without loss of generality, assume that $Q_1(0)>Q_2(0)$, and by continuity we have $Q_1(r)>Q_2(r)$ on some interval $r>0$. We now prove that $Q_1(r)>Q_2(r)$ for all $r\geq 0$. 
Multiplying the equation \eqref{uq4} written for $Q_1$ with $Q_2$ and subtracting the same with indexes reversed, we get
$$
Q_1^{\prime\prime} Q_2 - Q_1 Q_2^{\prime\prime} = - \frac{N-1}{r} \left(Q_1^{\prime}Q_2 - Q_1Q_2^{\prime}\right) + \left(U_{Q_1} - U_{Q_2}\right) Q_1Q_2,
$$
or, equivalently (multiplying by $r^{N-1}$),
\begin{equation}\label{uq6}
\frac{d}{dr}\big(r^{N-1}(Q_1^{\prime}Q_2-Q_1Q_2^{\prime})\big)=r^{N-1}\left(U_{Q_1}-U_{Q_2}\right)Q_1Q_2.
\end{equation}
Integrating \eqref{uq6}, we obtain
\begin{equation}\label{uq7}
r^{N-1}\left(Q_1^{\prime}(r)Q_2(r)-Q_1(r)Q_2^{\prime}(r)\right)=\int_{0}^{r}s^{N-1}\big(U_{Q_1}(s)-U_{Q_2}(s)\big)Q_1(s)Q_2(s)ds.
\end{equation}
Suppose that $Q_1(r)$ intersects $Q_2(r)$ at $r_1>0$ for the first time. Then, the left-hand side of \eqref{uq7} at $r_1$ is non-positive due to monotonicity and decay of both $Q_1$ and $Q_2$: 
\begin{equation}\label{uq8}
r_1^{N-1}Q_1(r_1)\left(Q_1^{\prime}(r_1)-Q_2^{\prime}(r_1)\right)\leq 0, 
\end{equation}
however, the right-hand side of \eqref{uq7} satisfies
\begin{equation}\label{uq9}
\int_{0}^{r_1}s^{N-1}Q_1(s)Q_2(s)\big(U_{Q_1}(s)-U_{Q_2}(s)\big)ds>0,
\end{equation}
since both $Q_1(r)$, $Q_2(r)>0$ along with $U_{Q_1}(r)>U_{Q_2}(r)$ for $0<r<r_1$. This leads to a contradiction, thus, $Q_1(r)$ and $Q_2(r)$ do not intersect, which implies that $Q_1(r)>Q_2(r)$ must hold for all $r\geq 0$. Now we show that this fact also leads to a contradiction. 
Consider the two Schr\"odinger operators $H_i = - \Delta + U_{Q_i}$, $i=1,2$, 
with $U_{Q_i}(r) = \int_0^r \left(1-\left(\frac{s}{r} \right)^{N-2} \right) s \, Q_i^2(s) \, ds $. Recalling that a ground state $Q_i(r)$ asymptotically behaves as $r^{-\frac{N-1}2+\varepsilon} e^{-|x|}$ (this is in the case $p=2$), it is easy to observe that $U_{Q_i}$ is not only bounded, but increases to a horizontal asymptote $y=c_N = const$. Hence, we can apply the classical Schrodinger operator theory (for example, \cite[Chapter 13]{RS4}) to show that both equations $H_i \,Q = Q$, $i=1,2$, have the unique positive ground state solution, respectively denoted by $Q_i$ (with the eigenvalue 1 as we rescaled the equation in \eqref{uq2}). This implies that $\langle H_i f, f \rangle \geq \|f\|_{L^2}$ for any $H^1$ function $f$ with equality holding on a multiple of $Q_i$, that is when $f = c_i Q_i$, $i=1,2$, respectively. 
Now, since $H_2 = H_1 - (U_{Q_1} - U_{Q_2})$, we obtain
$$
\|Q_1\|^2_{L^2} \leq \langle H_2 Q_1, Q_1 \rangle  = \langle H_1 Q_1, Q_1 \rangle - \langle (U_{Q_1} - U_{Q_2}) Q_1, Q_1 \rangle  = \|Q_1\|^2_{L^2} - \delta,    
$$
since $U_{Q_1} > U_{Q_2}$, yielding a contradiction. This implies that \eqref{uq1} (and hence \eqref{HQ}) can not have two distinct radial positive $H^1$ solutions. 
\end{proof}

\bibliography{Andy-references}
\bibliographystyle{abbrv}

\end{document}